\documentclass[11pt]{amsart}
\usepackage{amsmath,amssymb,enumerate,bbm,enumerate}
\numberwithin{equation}{section}

 \newtheorem{theorem}{Theorem}[section]
\newtheorem{lemma}[theorem]{Lemma}
 \newtheorem{theo}[theorem]{Theorem}
  \newtheorem{pro}[theorem]{Proposition}
 \newtheorem{thm}[theorem]{Theorem}
\newtheorem{lem}[theorem]{Lemma}
 \newtheorem{cor}[theorem]{Corollary}

\theoremstyle{definition}

\newtheorem{defi}[theorem]{Definition}

\theoremstyle{remark}
\newtheorem{remark}[theorem]{Remark}

\numberwithin{equation}{section}

 \theoremstyle{plain}
\newtheorem*{namedthm}{\namedthmname}
\newcounter{namedthm}

\makeatletter
\newenvironment{named}[1]
  {\def\namedthmname{#1}%
   \refstepcounter{namedthm}%
   \namedthm\def\@currentlabel{#1}}
  {\endnamedthm}

 \newcommand{\B}{\mathbb B}
 
 \newcommand{\R}{\mathbb R}
 
 \newcommand{\C}{\mathbb C}
 \newcommand{\N}{\mathbb N}

 \newcommand{\Tc}{\mathcal T}

 \newcommand{\e}{\varepsilon}
 \newcommand{\f}{\varphi}
 
 \newcommand{\p}{\psi}
 \newcommand{\s}{\sigma}

 \newcommand \PSH {{\rm PSH}}
 
 \newcommand \loc {{\rm loc}}

 \newcommand \Sub {\Subset}
 \newcommand \sub{\subset}

 \newcommand \setdef{\ ; \ }

 \usepackage{xcolor}

\usepackage{hyperref}
\hypersetup{
    bookmarks=true,         
    unicode=false,          
    pdftoolbar=true,       
    pdfmenubar=true,       
    pdffitwindow=false,    
    pdfstartview={FitH},   
    pdftitle={The Cauchy-Dirichlet problem for complex Monge-Amp\`ere flows},    
    pdfauthor={V. Guedj, C.H. Lu, A. Zeriahi},     
    colorlinks=true,       
   linkcolor=black,          
    citecolor=black,        
    filecolor=black,      
    urlcolor=black}           

\author{Vincent Guedj}
\address{Vincent Guedj, Institut de Math\'ematiques de Toulouse  \\ 
Universit\'e de Toulouse, CNRS \\
UPS, 118 route de Narbonne \\
31062 Toulouse cedex 09, France}
\email{vincent.guedj@math.univ-toulouse.fr}

\author{Chinh H. Lu}
\address{Hoang-Chinh  Lu, Laboratoire de Math\'ematiques d'Orsay,
 Univ. Paris-Sud,
 CNRS, Universit\'e Paris-Saclay,
  91405 Orsay, France}
\email{hoang-chinh.lu@math.u-psud.fr}

\author{Ahmed Zeriahi}
\address{Ahmed Zeriahi, Institut de Math\'ematiques de Toulouse,   \\ Universit\'e de Toulouse, CNRS \\
UPS, 118 route de Narbonne \\
31062 Toulouse cedex 09, France}
\email{ahmed.zeriahi@math.univ-toulouse.fr}

\thanks{The authors are partially supported by the ANR project GRACK}
  


\title[Pluripotential Complex Monge-Amp\`ere flows]{The Pluripotential Cauchy-Dirichlet problem \\  for complex Monge-Amp\`ere flows}
\date{\today}

\makeatletter

\begin{document}

\begin{abstract}  
We develop the first steps of a parabolic pluripotential theory in bounded strongly pseudo-convex domains of $\mathbb{C}^n$.
We study certain degenerate parabolic complex Monge-Amp\`ere equations, modelled on the K\"ahler-Ricci flow 
evolving on complex algebraic varieties with Kawamata log-terminal singularities.

 Under natural assumptions on the Cauchy-Dirichlet boundary data, we show that the envelope of pluripotential subsolutions
is semi-concave in time and continuous in space, and provides the unique pluripotential solution with such regularity.
\end{abstract} 

 \setcounter{tocdepth}{1}

\maketitle

\tableofcontents

\section*{Introduction} 

The Ricci flow, first introduced by Hamilton \cite{Ham82} is the equation
$$
\frac{\partial}{\partial t} g_{ij}=-2R_{ij},
$$
evolving a Riemannian metric by its Ricci curvature. If the Ricci flow starts from a K\"ahler metric,
the evolving metrics remain K\"ahler and the resulting PDE is called the K\"ahler-Ricci flow.

It is expected that the K\"ahler-Ricci flow can be used 
to give a geometric classification of 
complex algebraic and K\"ahler manifolds, and produce canonical metrics at the same time.
Solving the K\"ahler-Ricci flow  boils down to solving a parabolic scalar equation modeled on
$$
\det \left (\frac{ \partial^2 u_t}{\partial z_j \partial \bar{z}_k}(t,z) \right)=e^{\partial_t u_t(z)+H(t,z)+\lambda u_t(z)}
$$
where $t \mapsto u_t(z)=u(t,z)$ is a smooth family of strictly plurisubharmonic functions in $\C^n$,
$\lambda \in \R$ and $g=e^{H}$ is a smooth and positive density.

It is important for geometric applications to study {\it degenerate} versions of these complex Monge-Amp\`ere flows,
where the functions $u_t$ are no longer smooth nor strictly  plurisubharmonic,
and the densities may vanish or blow up
(see \cite{SW13, CT15, Song_Tian_2017KRflowInvent,Eyssidieux_Guedj_Zeriahi_2017FlowIII}
and the references therein).

A viscosity approach has been developed recently in \cite{Eyssidieux_Guedj_Zeriahi_2015FlowI},
following its elliptic counterpart \cite{EGZ11,HL11,HL13}. 
While the viscosity theory is very robust, it requires the data to be continuous hence has a limited scope of applications.
Several geometric situations encountered in the Minimal Model program (MMP)
 necessitate one to deal with Kawamata log-terminal (klt) singularities.
The viscosity approach breaks down in these cases and a more flexible method is necessary.  

\smallskip
 
There is a well established pluripotential 
theory of weak solutions to degenerate elliptic complex Monge-Amp\`ere equations,
following the pioneering work of Bedford-Taylor \cite{Bedford_Taylor_1976Dirichlet,Bedford_Taylor_1982Capacity}.
This theory allows to deal with $L^p$-densities as established in a corner stone result of Ko{\l}odziej \cite{Kol98},
which provides a great generalization of
 \cite{Yau78}.

No similar theory has ever been developed on the parabolic side.
The purpose of this article, the first of a series on this subject, is to develop a pluripotential theory for
degenerate complex Monge-Amp\`ere flows.
This article settles the foundational material for this theory and
 focuses on solving the Cauchy-Dirichlet problem in domains of $\mathbb{C}^n$. 

\smallskip

We consider the following family of Monge-Amp\`ere flows
\begin{equation}\tag{CMAF} \label{eq: CMAF}
 d t \wedge (dd^c u)^n =e^{\partial_t {u} + F (t,z,u)} g(z) d t \wedge d V ,
\end{equation}
in $\Omega_T :=]0,T[ \times \Omega$, where $d V$ is the euclidean volume form on $\C^n$ and
  \begin{itemize}
 \item $T>0$ and $\Omega \Subset \C^n$ is a bounded strictly pseudoconvex domain;
  \item $F (t,z,r)$ is continuous in $[0,T[ \times \Omega \times \R$, increasing  in $r$, bounded in $[0,T[ \times \Omega \times J$, for each $J\Subset \mathbb{R}$;
\item $(t,r) \mapsto F(t,\cdot,r)$ is uniformly Lipschitz and  semi-convex  in $(t,r)$;
  \item  $g \in L^{p} (\Omega)$, $p > 1$, and $g>0$ almost everywhere ; 
   \item $u : [0,T[ \times \Omega \rightarrow \R$ is the unknown function.
   \end{itemize}
   
   Here $d=\partial+\overline{\partial}$ and $d^c =i(\overline{\partial}-\partial)/2$ so that 
   $dd^c =i\partial \overline{\partial}$ and 
   $(dd^c u)^n$ represents the determinant of the complex Hessian of $u$ in space (the complex Monge-Amp\`ere operator)
   whenever $u$ is ${\mathcal C}^2$-smooth.
   
For less regular functions $u$, the equation \eqref{eq: CMAF} should be understood in the weak sense of pluripotential theory as we  explain in Section \ref{sec: MAP operator}. 

We let $\mathcal P (\Omega_T)$ denote the set of  {\it parabolic potentials}, i.e. 
those functions $u:{\Omega}_T \rightarrow [- \infty, + \infty[$ defined  in $\Omega_T = ]0,T[ \times {\Omega}$ and satisfying the following conditions:
\begin{itemize} 
\item for any $t \in ]0,T[$, $u (t,\cdot) $ is plurisubharmonic in $\Omega$;
\item  the family   $\{u (\cdot,z) \setdef  z\in \Omega\}$ is locally uniformly Lipschitz in $]0,T[$.

\end{itemize}

We study in Section \ref{sec:recap}  basic properties of parabolic potentials. We show  in Lemma \ref{lem: usc} that if $u  \in \mathcal P (\Omega_T)$ and is bounded from above in $\Omega_T$ then it can be uniquely  extended as an upper-semicontinuous function in $[0,T[ \times \Omega$ such that $u (0,\cdot)$ is plurisubharmonic in $\Omega$.
We show that  parabolic potentials satisfy
approximate submean value inequalities (Lemma \ref{lem: approximate sub-mean})
and enjoy good compactness properties (Proposition \ref{pro: Montel property of P}).

\smallskip

We show in Section \ref{sec: MAP operator} that 
parabolic complex Monge-Amp\`ere operators are well defined on 
$\mathcal P (\Omega_T) \cap L^{\infty}_{\loc} (\Omega_T)$ and enjoy 
nice continuity properties, allowing to make sense of pluripotential 
sub/super/solutions to \eqref{eq: CMAF} (see Definition \ref{def: subsolution}).  
A crucial convergence property is obtained in Proposition 
\ref{pro: convergence semiconcave}, under a semi-concavity
assumption on the family of parabolic potentials.

 

\smallskip
 
 	A {\it Cauchy-Dirichlet boundary data} is a function $h$ defined on the parabolic boundary of $\Omega_T$  denoted by
 \begin{equation*} 
  \partial_0{\Omega_T} : = ([0,T[ \times \partial{\Omega}) \cup (\{0\} \times {\Omega}),
 \end{equation*}
 such that  
 \begin{itemize}
 \item  the restriction of  $h$ on $[0,T[ \times \partial \Omega$ is continuous; 
 \item  the family $\{h (\cdot, z) \setdef z\in  \partial \Omega\}$ 
  is locally uniformly Lipschitz in $]0,T[$ ;
 \item $h$ satisfies the following compatibility condition : $ \forall \zeta\in \partial \Omega$, 
 \begin{equation} \label{eq: weak compatibility condition}
h_0 := h(0,\cdot) \in \PSH(\Omega)\cap L^{\infty}(\Omega) \ \textrm{and}\ \lim_{\Omega\ni z\to \zeta} h(0,z) = h (0 ,\zeta). 
\end{equation}
 \end{itemize}

 The Cauchy-Dirichlet problem for the parabolic equation \eqref{eq: CMAF} with Cauchy-Dirichlet boundary data $h$
consists in finding $u \in  \mathcal P (\Omega_T) \cap L^{\infty} (\Omega_T)$ such that \eqref{eq: CMAF} holds in the pluripotential sense in $\Omega_T$ and  the following Cauchy-Dirichlet boundary conditions are satisfied :

 \begin{equation}  \label{eq: Dirichlet condition}
 \forall (\tau,\zeta)\in  [0,T[ \times \partial {\Omega}, \, \, \, \, \lim_{\Omega_T\ni (t,z)\to (\tau,\zeta)}u(t,z) = h(\tau,\zeta).
 \end{equation}

 \begin{equation}  \label{eq: Cauchy condition}
   \lim_{t \to 0^+} u_t  =  h_0  \, \, \, \, \mathrm{in} \, \, \, \, L^1 (\Omega).
 \end{equation}
In this case we  say that 
{\it $u$ is a solution to the Cauchy-Dirichlet problem for the  equation \eqref{eq: CMAF} with  boundary values $h$.}

  Observe that  a solution $u$ to the  equation \eqref{eq: CMAF} has plurisubharmonic slices in $\Omega$ and   the Cauchy condition (\ref{eq: Cauchy condition}) implies by a classical result in pluripotential theory  that $(\limsup_{t \to 0} u_t)^* = h_0^*  \in \PSH (\Omega)$, hence $h_0 = h_0^* \in \PSH (\Omega)$. 
This observation shows that the Cauchy data $h_0$ must be plurisubharmonic as it is required in   the compatibility condition (\ref{eq: weak compatibility condition}).

For a solution to   the Cauchy-Dirichlet problem for   the  equation \eqref{eq: CMAF}, the Cauchy condition (\ref{eq: Cauchy condition}) implies that 
 \begin{equation*} 
 \forall z \in \Omega, \, \, \,  \lim_{t \to 0^+} u _t (z) = h_0 (z).
 \end{equation*}
 
\smallskip
It is possible to consider less regular initial Cauchy data $h (0,\cdot)$ (see \cite{Son1,Son2}), but we will not pursue this here.

\smallskip

We try and construct a solution to the Cauchy-Dirichlet problem by the Perron method, considering the upper envelope $U$ of pluripotential subsolutions. 

 The technical core of the paper lies in Section \ref{sect: Perron envelope boundary value} and Section \ref{sect: time regularity}. In Section \ref{sect: Perron envelope boundary value}
  we construct subbarriers and controls from above to ensure that $U$ has the right boundary values 
(see Theorem \ref{thm: boundary value of U}).  In Section \ref{sect: time regularity} we prove that the Perron envelope of subsolutions is locally uniformly Lipschitz and semiconcave in time.

\begin{named}{Theorem A}\label{thm: main thm A}
	Assume $h$ is a Cauchy-Dirichlet boundary data in $\Omega_T$  such that for all $0<S<T$,
	and for all $(t,z) \in ]0,S]\times \partial \Omega$,
	\begin{equation}\tag{\dag}
		\label{eq: Lip and concave h}
		t |\partial_t h(t,z)| \leq C(S) 
\;	\; 	\text{ and } \; \;
	  t^2 \partial_t^2 h(t,z) \leq C(S), 
	\end{equation}
		 	
	 Then the envelope $U=U_{h,g,F}$ is locally uniformly Lipschitz and locally uniformly semi-concave in $t\in ]0,T[$.  
	 Moreover, $U$ satisfies the Cauchy-Dirichlet boundary conditions \eqref{eq: Dirichlet condition}, \eqref{eq: Cauchy condition}.
\end{named}

Here $C(S)$ is a positive constant depending on $S$ which may blow up as $S\to T$.
The proof of \ref{thm: main thm A}, which shows in particular that $U$ satisfies \eqref{eq: Lip and concave h},  is given in Theorem \ref{thm: U is local Lip in t}, Theorem \ref{thm: U is subsolution} and Theorem \ref{thm: U is local semiconcave}. The Lipschitz and semi-concave constants of $U$ depend explicitly on $C(S)$. 

We prove in  Theorem \ref{thm: U is Lip z} that the envelope $U$ is moreover (Lipschitz) continuous in space 
 if so are the data $(h_0,\log g,F)$.

\smallskip

Focusing for a while on the case of the unit ball with regular boundary data,
we obtain
the following parabolic analogue of Bedford and Taylor's celebrated result  \cite{Bedford_Taylor_1976Dirichlet} :

\begin{named}{Theorem B}
	\label{thm: main thm B}
	  Assume $\Omega=\mathbb{B}$ is the unit ball in $\mathbb{C}^n$  and
   \begin{itemize}
   	\item $G := \log g$ is ${\mathcal C}^{1,1}$ in $\bar{\B}$; 
   	\item $h$ is uniformly Lipschitz in $t\in [0,T[$, satisfies $\partial_t^2 h(t,z) \leq C/t^2$, $z\in \partial \B$, and $h$ is uniformly  ${\mathcal C}^{1,1}$  in $z\in \bar{\B}$; 
   	\item $F$ is Lipschitz and semi-convex  in $[0,T[\times \bar{\B}\times J$, for each $J\Subset \mathbb{R}$.  
   \end{itemize}
   
 Then the upper envelope $U := U_{h,g,F}$ is  locally  uniformly ${\mathcal C}^{1,1}$ in $z $
 and locally uniformly Lipschitz  in $t\in ]0,T[$.
For almost any $(t,z) \in \B_T$, we have 
 $$
 \mathrm{det} \, \left(\partial_j \bar \partial_k U (t,z)\right) = e^{\partial_t U (t,z) + F (t,z,U(t,z))} g(z).
 $$
 
In particular $U$ is a pluripotential solution to the Cauchy-Dirichlet problem for the parabolic equation \eqref{eq: CMAF} with boundary values $h$.
\end{named}

This result is obtained as a combination of Theorem \ref{thm: U is C1,1}
and Theorem \ref{thm: U is solution ball reg}.
Using an approximation and balayage process we then treat the case of more general domains $\Omega$
with less regular boundary data, obtaining the following solution to our original problem :
 
\begin{named}{Theorem C}
	\label{thm: main thm C}
	Assume $h$ is a Cauchy-Dirichlet boundary data in $\Omega_T$  such that for all $0<S<T$,
	and for all $(t,z) \in ]0,S]\times \partial \Omega$,
	\begin{equation}\tag{\dag}
		t |\partial_t h(t,z)| \leq C(S) 
\;	\; 	\text{ and } \; \;
	  t^2 \partial_t^2 h(t,z) \leq C(S), 
	\end{equation}
		
	The envelope of all subsolutions to  $\eqref{eq: CMAF}$ with Cauchy-Dirichlet boundary data $h$
	is a pluripotential solution to this Cauchy-Dirichlet problem.
\end{named}

 The proof of this fundamental result is given in Theorem \ref{thm: U is solution general}.
We eventually establish a comparison principle, which shows that 
$U_{h,g,F}$ is unique:

 \begin{named}{Theorem D}
	\label{thm: main thm D}
 	Same assumptions as in \ref{thm: main thm A}.  Let $\Phi$ be a bounded pluripotential subsolution  to  \eqref{eq: CMAF} with boundary values $h_{\Phi}$. Let $\Psi$ be a bounded pluripotential supersolution  with boundary values $h_{\Psi}$, such that 
 $\Psi$ is locally uniformly semi-concave in $t\in ]0,T[$
and  $h_{\Phi}$ satisfies \eqref{eq: Lip and concave h}. Then
$$
h_{\Psi} \geq h_{\Phi}
\text{  on } \partial_{0}\Omega_T
\Longrightarrow
\Phi \leq \Psi
\text{  in }
\Omega_T.
$$

In particular, there is a unique  
pluripotential solution to the Cauchy-Dirichlet problem for \eqref{eq: CMAF} with boundary data $h$,
which is locally uniformly semi-concave in $t$. 
 \end{named}
 
 The proof of \ref{thm: main thm D} is given in Section \ref{sect: uniqueness};
 it uses some ideas from \cite{GLZ_stability,DiNezza_Lu_2017KRflow}.
 When all the data $(h,F,g,u)$ are continuous, one can show that the solution $U$ coincides with the viscosity solution 
 constructed in \cite{Eyssidieux_Guedj_Zeriahi_2015FlowI}. We refer the reader to \cite{GLZ3} for a detailed comparison of viscosity 
 and pluripotential concepts.


\bigskip

\centerline{{\it Notations and assumptions on the data}}

\medskip

 We  finish this introduction by fixing some notations that will be used throughout the paper. 

 \subsubsection*{The domain}
 
In the whole article we let $d V$  denote the euclidean volume form in $\C^n$
and
$\Omega \Subset \C^n$ be a strictly pseudoconvex domain : there exists a smooth function $\rho$ in a neighborhood $V$ of $\bar{\Omega}$ such that 
$$
\Omega = \{z\in V \setdef \rho(z)<0\},
$$
where  $\partial_z \rho \neq 0$ on $\partial \Omega$ and $\rho$ is strictly plurisubharmonic in $V$. We set $\Omega_T:=]0,T[ \times \Omega$ with $T > 0$. Most of the time we will assume that $T<+\infty$. 

Recall that if a function $u:\Omega \rightarrow [- \infty , + \infty[$ is plurisubharmonic, then $dd^c u \geq 0
 $ is a positive current on $\Omega$. Here $d=\partial+\overline{\partial}$
and  $d^c= (i \slash 2) (\overline{\partial}-\partial)$ are both real operators so that $dd^c = i \partial \overline{\partial} $.

We let $\mathbb{B}$ denote the euclidean unit ball in $\mathbb{C}^n$ and $\lambda_{\mathbb{B}}$ denote the normalized Lebesgue measure on $\mathbb{B}$. 



\subsubsection*{The function $F$}    
\label{subsect: F}
 We assume that $F: [0,T[\times \Omega \times \mathbb{R}\rightarrow \mathbb{R}$ is continuous and
\begin{itemize}
	\item bounded in $[0,T[\times \Omega \times J$ for each $0<S<T$, $J\Subset \mathbb{R}$;
	\item increasing in $r$:  $r \mapsto F(t,x,r)$ is increasing for all $(t,x) \in \Omega_T$ fixed;
	\item locally uniformly Lipschitz in $(t,r)$ : for each compact  $J\Subset \mathbb{R}$ and each $0<S<T$ there exists a constant $\kappa=\kappa(S,J)>0$ such that for all $t,\tau \in [0,S]$, $z\in \Omega$, $r,r'\in J$,
   \begin{equation}
   	\label{eq: Lip F}
   	   |F(t,z,r)-F(\tau,z,r')| \leq \kappa ( |t-\tau| + |r-r'|);
   \end{equation}
	\item locally uniformly semi-convex in $(t,r)$ : for each compact subset $[0,S]\times J \Subset [0,T[\times \mathbb{R}$ there exists a constant $C=C(S,J)>0$ such that, for any $z\in \Omega$, the function \begin{equation}
	\label{eq: semi convex F} 
	(t,r)\mapsto F(t,z,r)+C(t^2+r^2)\ \textrm{is convex in}\  [0,S]\times J.
\end{equation}      
\end{itemize}

\subsubsection*{The density $g$} 
\label{subsect: g}

We assume that
\begin{itemize}
\item  $0\leq g \in L^p(\Omega)$ for some $p>1$ that is fixed thoughout the paper ;
\item     the set $\{z\in \Omega \setdef g(z)=0\}$ has Lebesgue measure zero.  
\end{itemize}

 \subsubsection*{Boundary data $h$}
 \label{subsect: h}

We assume throughout the article that 
\begin{itemize}
	\item $h: \partial_0 \Omega_T \rightarrow \R$ is bounded, upper semi-continuous on $\partial_0\Omega_T$; 
	\item the restriction of $h$ on $[0,T[\times \partial \Omega$ is continuous;
	\item 	$t \mapsto h (t,z)$ 
	is locally uniformly Lipschitz in $]0,T[$:  for all $0<S<T$ there is $C(S)>0$ such that
	 for all $(t,z) \in ]0,S] \times \partial \Omega$, 
	$$
			t |\partial_t h(t,z)| \leq C(S);
    $$
	\item 
	$h(0,\cdot)$ is bounded, plurisubharmonic in $\Omega$,
	and satisfies
\[
\lim_{\Omega\ni z\to \zeta} h(0,z) =h(0,\zeta), \ \forall \zeta \in \partial \Omega. 
\]
\end{itemize}

We eventually also assume that $t \mapsto h(t,z)$ is locally uniformly semi-concave in $]0,T[$ : 
for all $0<S<T$ there is $C(S)>0$ such that 
	$$
	  t^2 \partial_t^2 h(t,z) \leq C(S), \ \forall (t,z)\in [0,S]\times \partial \Omega. 
	$$

\subsubsection*{The constants} 

We fix once and for all various uniform constants:  
	\begin{equation}
		\label{eq: MF and Mh}
		M_h:= \sup_{\partial_0 \Omega_T} |h| \ , \  M_F:= \sup_{\Omega_T} F(\cdot,\cdot, M_h). 
	\end{equation}
We fix a   plurisubharmonic function $\rho$ in $\Omega$, continuous in $\bar{\Omega}$ so that 
 \begin{equation} \label{eq: rho}
 (dd^c \rho)^n = g d V, \, \, \, \rho = 0 \, \, \text{in} \, \, \partial \Omega,
 \end{equation} 
 in the weak sense  in $\Omega$. 
 Such a function exists by \cite{Kol95,Kol98}
 and there is moreover a uniform a priori bound on $\rho$,
 \begin{equation*} \label{eq: rho uniform bound}
  \Vert \rho\Vert_{L^ {\infty} (\Omega)} \leq c_n \Vert f\Vert_{L^ p (\Omega)}^ {1 \slash n},
  \end{equation*}
where $c_n > 0$ is a uniform constant depending on $n, \Omega$.


\section{Families of plurisubharmonic functions} \label{sec:recap}
Parabolic potentials form the basic objects of our study. They can be seen as weakly regular family of plurisubharmonic functions. In this section we define them and establish their first properties.

\subsection{Basic properties}

\subsubsection{Parabolic potentials}

We start with some basic definitions which will be used throughout all the paper.

 
 \begin{defi} Let  $ u : \Omega_T :=  ]0,T[ \times  {\Omega} \longrightarrow [- \infty , + \infty[$ be a given function.
 
  We say that  the family $\{u(\cdot,z) \setdef z \in \Omega\}$ is locally  uniformly  Lipschitz in $]0,T[$ if for any subinterval $J \Subset ]0,T[$ there exists a constant $\kappa := \kappa_J (u) > 0$ such that 
 \begin{equation} \label{eq: def Lip}
  u (t,z)  \leq u (s,z)  + \kappa \vert t - s\vert,
\text{  for all } s, t \in J
\text{ and } z \in {\Omega}.
 \end{equation}

\end{defi}
 
 
 \begin{defi} \label{defi: parabolic potential}
The set of parabolic potentials  $ \mathcal P (\Omega_T)$ is the set of functions  
$ u : \Omega_T :=  ]0,T[ \times  {\Omega} \longrightarrow [- \infty , + \infty[$  such that
\begin{itemize}
 \item for all $t  \in ]0,T[$, the slice $u_t : z \mapsto u (t,z)$ is plurisubharmonic in $\Omega$;
 \item  the family $\{u(\cdot,z) \setdef z \in \Omega\}$ is locally  uniformly  Lipschitz in $]0,T[$.
\end{itemize}  
\end{defi}

 \smallskip
 
 $\PSH(\Omega)$ embeds in $\mathcal P (\Omega_T)$ as the class of time independent potentials.
 Basic operations on plurisubharmonic functions extend naturally to parabolic potentials:
 \begin{itemize}
 \item  if $u,v \in  \mathcal P (\Omega_T)$  then 
 $u+v \in  \mathcal P (\Omega_T)$ and $\max(u,v) \in  \mathcal P (\Omega_T) $;
 \item   if $u  \in  \mathcal P (\Omega_T)$ and $\lambda \geq 0$
  then $\lambda u$ is also a parabolic potential. 
 \end{itemize}
 
 It follows from the next Lemma that parabolic potentials extend naturally as  
 upper semi-continous functions in $[0,T[ \times \Omega$.

 \begin{lem}
\label{lem: usc} 
 Let $u:  ]0,T[ \times \Omega \longrightarrow [- \infty , + \infty[$ be a function bounded from above and satisfying the following conditions :
	
	$ (i)$  for any $t \in ]0,T[$ the function $u_t := u (t,\cdot)$ is plurisubharmonic in $\Omega$ ; 
	
	$(ii)$ for all $z \in \Omega$  the function $ u (\cdot,z)$ is upper semicontinuous in $]0,T[$.
	
\noindent	For $z \in \Omega$ we set
	$$
	u_0(z) := (\limsup_{t\to 0^+} u_t)^* (z) = \limsup_{\zeta \to z} \left(\limsup_{t\to 0^+} u_t(\zeta)\right).
	$$

 Then $u_0$ is plurisubharmonic in $\Omega$ and the extension $u : [0,T[ \times \Omega \rightarrow [-\infty,+\infty[$ is upper semicontinuous 
 in $[0,T[ \times \Omega$. 
 \end{lem}

 \begin{proof} 
 Since $u$ is bounded from above we can assume that $u \leq 0$. Fix $(t_0,z_0) \in \Omega_T$  and let $r > 0$ be such that $B (z_0,2 r) \Subset \Omega$. Fix  $\delta \in ]0,r[$.  
 
 Since $u_t\leq 0$, by the submean value inequality for psh functions, we have for $\vert z - z_0\vert \leq \delta$ and $t \in ]0,T[$,
 \begin{equation*}
  u(t,z)   \leq \frac{1}{\mathrm{Vol}(B(z,r+\delta))}\int_{B(z_0,r)} u(t,\zeta) dV(\zeta).
 \end{equation*} 
It thus follows from Fatou's Lemma and assumption $(ii)$ that
 \begin{equation}
 	\label{eq: usc extension}
 	 \limsup_{(t,z) \to (t_0,z_0)} u (t,z) \leq \frac{1}{\mathrm{Vol} (B(z_0,r+\delta))}  \int_{B(z_0,r)} u (t_0,\zeta)dV(\zeta). 
 \end{equation}
Since $u (t_0,\cdot)$ is plurisubharmonic in $\Omega$,  letting $\delta \to 0^+$ and $r \to 0^+$ we obtain
  
  $$
   \limsup_{(t,z) \to (t_0,z_0)} u (t,z) \leq  u (t_0,z_0),
   $$
   which proves that $u$ is upper semi-continuous at $(t_0,z_0)$. 
 
 Now if $t_0 = 0$, since  $\{ u_t \setdef t \in ]0,T[ \}$ is a  family of plurisubharmonic functions  in $\Omega$ which is uniformly bounded from above, it follows that  $u_0$ is plurisubharmonic in $\Omega$. 
 Then by \eqref{eq: usc extension},
 $$
 \limsup_{(t,z) \to (0,z_0)} u (t,z) \leq \frac{1}{\mathrm{Vol} (B(z_0,r+\delta))}  \int_{B(z_0,r)} u_0 (\zeta)dV(\zeta). 
 $$
 
 Letting $r \to 0^+$ we obtain, by the plurisubharmonicity of $u_0$,
 $$
  \limsup_{(t,z) \to (0,z_0)} u (t,z)   \leq u_0 (z_0) =: u (0,z_0),
  $$
   which proves the semi-continuity of the extension at the point $(0,z_0)$.    
 \end{proof}

The next result provides a parabolic analogue of a classical result of Lelong about negligible sets 
for plurisubharmonic functions; it will play an important role in section \ref{sect: Perron envelope boundary value}.

 \begin{lem} \label{lem: negligible}
  Let $\mathcal U  \subset \mathcal P (\Omega_T)$ be a  family of functions  which is locally uniformly bounded from above. 
 Assume $ U := \sup \{u \setdef u \in  \mathcal U \} $
  is  locally uniformly Lipschitz in  $t \in ]0,T[$.
  Then
  \begin{itemize}
  \item  the upper semicontinuous regularization  $U^{*}$ (in $\Omega_T$) 
 belongs to $\mathcal P (\Omega_T)$;
 \item for any $t \in ]0,T[$,  $U^* (t,\cdot) = (U_t)^*$ in $\Omega$ and the exceptional set  
 $$
 E (U) := \{ (t,z) \in \Omega_T \setdef  \ U (t,z) < U^{*}  (t,z) \}
 $$ 
 has zero $(2 n +1)$-dimensional Lebesgue measure in $\Omega_T \subset \R^{2 n +1}$.  
  \end{itemize}
 \end{lem}

The smallness of the exceptional set $E(U)$  can be made more precise:  all the $t$-slices of   $E (U)$ have zero $2 n$-dimensional Lebesgue measure in $\Omega$.

 \begin{proof} 
 Our assumption ensures that the function $U$ is locally bounded from above. The first statement follows immediately from (\ref{eq: def Lip}).  Since $U$ is locally Lipschitz in $t$, there is no need to regularize in the $t$ variable: it follows from Lemma \ref{lem: partial continuity} that for all $(t,z) \in \Omega_T$,
 $$
 U^{*}  (t,z) = (U_t)^* (z),
 $$ 
where the upper semicontinuous regularization in the LHS is in the $(t,z)$-variable, while the upper semicontinuous regularization in the RHS is in the $z$-variable only, $t$ being fixed.  A classical theorem of  Lelong (see \cite[Proposition 1.40]{GZbook}) ensures that $E_t = \{z \in \Omega \setdef U_t(z) < (U_t)^* (z)\}$ has zero Lebesgue measure in $\mathbb{C}^n$.  Since $E= \{(t,z) \in \Omega_T  \setdef  \ z \in E_t \}$ the second statement of the lemma follows from Fubini's theorem.
 \end{proof}

 \subsubsection{Semi-continuous regularization}
  Given a   function $u$ on a metric space $(Z,d)$ which is locally bounded from above, we define $\text{usc}_Z u$ 
  to be the smallest upper semi-continuous function lying above $u$,
  $$
{\rm usc}_Z u \,  (z): = \limsup_{z' \to z} u(z') = \inf_{r > 0} \left(\sup_{B (z,r)} u \right).
  $$
  
   Fix $I \subset \R$ an interval, $(Y,d)$   a metric space and 
   $\phi : I \times Y \longrightarrow [- \infty, \infty[$ a   function. 
For any  $\delta >0$, we denote by $\kappa_I (\phi,\delta)$  the smallest constant $\kappa > 0$ such that for any $s, t \in I$ with $\vert s -t \vert \leq \delta$ and any $y \in Y$, 
 $$
 \phi (t,y) \leq \phi (s,y) + \kappa.
 $$
  
 \begin{lem} \label{lem: partial continuity}
 Assume  $\phi : I \times Y \longrightarrow [- \infty, + \infty[$ 
 satisfies $ \lim_{\delta \to 0} \kappa_I (\phi,\delta)= 0$.
Then for all $t \in I$ and $y \in Y$,  
$$
({\rm usc}_Y \phi_t) (y) = ({\rm usc}_{I \times Y} \phi)  (t,y).
$$ 

  In particular if   $\phi (t_0,\cdot)$ is upper semi-continuous  at a point $y_0 \in Y$
  for some $t_0 \in I$, then $\phi$ is upper semi-continuous  at the point $(t_0,y_0) \in I \times Y$.
 \end{lem}
 
 It follows from Lemma \ref{lem: partial continuity} that parabolic potentials are upper semi-continuious in $\Omega_T$.
 
 We can introduce similarly the lower semi-continuous regularization ${\rm lsc}_Z u$ of a  function $u$ which is locally bounded from below on $Z$. A similar conclusion holds.
 
 \begin{proof} 
 Observe that $ \text{usc}_{I \times Y} \phi  (t,y) \geq  \text{usc}_Y (\phi_t)  (y)$ for any $(t,y) \in I \times Y$.  To prove the reverse inequality, fix $t \in I$, $y \in Y$, $A \geq \text{usc}_Y (\phi_t)  (y)$ and $\varepsilon>0$. Then there exists a neighborhood $V$ of $y$ in $Y$ such that for $z \in V$ we have
 \[
\phi (t,z) \leq A+ \varepsilon.
\]
Since $\lim_{\delta \to 0} \kappa_I (\phi,\delta)= 0$, there exists $\delta > 0$ small enough such that $\kappa_I (\phi,\delta) \leq \e$.
For  $(s,z) \in I \times V$ with $\vert s - t\vert \leq \delta$ we have
$$
  \phi (s,z) \leq  \phi (t,z) + \kappa_I (\phi,\delta) \leq A + \e + \kappa_I (\phi,\delta) \leq A+ 2\e.
 $$
 This proves the inequality $ \text{usc}_{I \times Y} \phi  (t,y) \leq  \text{usc}_Y (\phi_t)  (y)$.
 \end{proof}

\subsubsection{Approximate submean value inequalities} 
 
Parabolic potentials satisfy approximate sub-mean value inequalities:

  \begin{lem}  \label{lem: approximate sub-mean}
 Let $\Omega \subset \C^n$ be a domain and $u \in \mathcal P (\Omega_T)$.
 Fix $(t_0,x_0) \in \Omega_T$ and $\e_0, r_0 > 0$ so that $[t_0 - \e_0 , t_0 + \e_0] \times \bar B (x_0,r_0) \Subset \Omega_T$. Then for any $0 < \e \leq \e_0$, $0 < r \leq r_0$, 
 $$
 u (t_0,x_0) \leq  \int_{-1}^1 \int_{\B} u (t_0 +  \e s, x_0 + r \xi) \ d \lambda_\B (\xi) \ d s \slash 2  +  \kappa_0 \e,
 $$
 where $\kappa_0 > 0$ is the uniform Lipschitz constant of $u$ in $[t_0 - \e_0 , t_0 + \e_0]\times B (x_0 ,r)$.  
 
 \end{lem}
 
\begin{proof}
 Since $u (t_0,\cdot)$ is psh in $\Omega$, the submean-value inequality  yields,  for all $ 0 < r \leq r_0$,
 $$
 u (t_0,z_0) \leq   \int_{\mathbb{B}} u (t_0 , z_0 + r \xi) \ d \lambda_\B (\xi).
 $$
The Lipschitz condition ensures that for $ 0 < r \leq r_0$, $0 < \e \leq \e_0$, and $-1 \leq s \leq 1$,
 \[
  \int_{\B} u (t_0 , z_0 + r \xi) \ d \lambda_\B (\xi) \leq   \int_{\B} u (t_0 +  \e s, z_0 + r \xi)  \ d \lambda_\B (\xi) +  \kappa_0 \e \vert s\vert .
\]
 Integrating in $s$ we obtain the required inequality.
 \end{proof}
 
 Parabolic potentials therefore enjoy  interesting integrability properties:
 
 \begin{cor} 
 We have $\mathcal P (\Omega_T) \subset L^q_{\loc} (\Omega_T)$ for any $q \geq 1$.
  Moreover if $u \in \mathcal P (\Omega_T)$ then
 for all $(t,z) \in \Omega_T$,
 $$
 u (t,z) = \lim_{\e, r \to 0}  \int_{-1}^1  \int_{\B} u (t +  \e s, z + r \xi) \ d \lambda_\B (\xi) \ d s \slash 2.
 $$
 In particular if $u , v \in \mathcal P (\Omega_T)$ and $u \leq v$ a.e. in $\Omega_T$, then $u \leq v$ everywhere.
 \end{cor}
 
 \begin{proof}
Let $u\in \mathcal{P}(\Omega_T)$ and fix $K\Subset \Omega_T$ a compact subset. Then there exists a compact  interval  $J\Subset ]0,T[$ and a compact subset $D\Subset \Omega$ such that $K\subset J\times D$. 

Fix $t_0\in J$. Since $u(t_0,\cdot)$ is plurisubharmonic in $\Omega$ we have that $u(t_0,\cdot) \in L^q(D)$. Since $u(\cdot,z)$ is uniformly Lipschitz in $J$ we infer  $|u(t,z)-u(t_0,z)| \leq \kappa_J|t-t_0|$ for all $t\in J$ and $z\in D$. It thus follows from Fubini's theorem that 
\begin{flalign*}
\int_{J\times D} |u(t,z)|^q& dV(z)dt = \int_{J\times D} |u(t,z)|^q dV(z) dt\\
& \leq  2^{q-1} \left( \int_{J\times D} \left( |u(t_0,z)|^q +  |u(t,z)-u(t_0,z)|^q\right)  dV(z) dt \right)  	\\
&\leq   2^{q-1} \int_{J\times D}  |u(t_0,z)|^q + 2^{q-1} \kappa_J^q {\rm Vol}(D) \int_J |t-t_0|^q dt. 
\end{flalign*}
This proves that $u\in L^q(K)$, hence $u\in L^q_{\rm loc}(\Omega_T)$.  

\smallskip

 Fix $(t_0,z_0)\in \Omega_T$ and $\delta>0$. Since $u(t_0,\cdot)$ is psh in $\Omega$ we have 
\[
u(t_0,z_0) = \lim_{r\to 0^+}  \int_{\mathbb{B}} u(t_0,z_0+r\xi) d\lambda_{\mathbb{B}} (\xi). 
\]
For small enough $r>0$ we infer
\[
u(t_0,z_0) \geq   \int_{\mathbb{B}} u(t_0,z_0+r\xi) d\lambda_{\mathbb{B}} (\xi) -\delta. 
\]
Fix $\varepsilon_0>0$ such that $[t_0-\varepsilon_0,t_0+\varepsilon_0]\Subset ]0,T[$. 
Let $\kappa_0$ be the uniform Lipschitz constant of $u$ in $[t_0-\varepsilon_0,t_0+\varepsilon_0] \times \Omega$. 
Then 
for $\varepsilon \in ]0,\varepsilon_0[$,  
\begin{flalign*}
u(t_0,z_0)& \geq \int_{\mathbb{B}} u(t_0+ \varepsilon s,z_0+r\xi) d\lambda_{\mathbb{B}} (\xi)  - \kappa_0\varepsilon |s| -\delta.
\end{flalign*}
Integrating in $s\in [-1,1]$ we obtain 
\[
u(t_0,z_0) \geq \int_{-1}^1\int_{\mathbb{B}} u(t_0+ \varepsilon s,z_0+r\xi) d\lambda_{\mathbb{B}} (\xi)ds/2  - \kappa_0\varepsilon  -\delta.
\]
Letting $r\to 0, \varepsilon\to 0$ and then $\delta\to 0$ we get one inequality. 
The reverse inequality was already obtained in Lemma  \ref{lem: approximate sub-mean}.
\end{proof}

 \subsection{Behaviour on slices}

We now estimate the $L^1$-norm on slices in terms of the global $L^1$-norm. 

 \begin{lem} \label{lem:L1Slice-L1} 
Fix $u, v \in \mathcal P (\Omega_T)$ and  $0 < T_0 < T_1  <S< T$. 
Then  for all $T_0 \leq t \leq T_1$,
$$
 \Vert u (t,\cdot) - v (t,\cdot)\Vert_{L^1 (\Omega)} \leq 2  M \max \left\{\Vert u - v\Vert_{L^1 (\Omega_{T_1})}^{1 \slash2}, \Vert u - v\Vert_{L^1 (\Omega_{T_1})} \right\}, 
$$ 
 where 
 $
 M :=\max \{\sqrt{ \kappa {\rm Vol}(\Omega)}, (S - T_1)^{- 1}\},
 $
 and 
 $$
 \kappa:=\sup_{z \in \Omega} {\sup}^{*}_{t,s \in [T_0,S]} \frac{\left|  (u-v)(t,z)-(u-v)(s,z) \right|}{|t-s|}
 $$ 
 is the uniform Lipschitz constant of $u-v$ in $[T_0,S]$.
 \end{lem}
 
 Here  $\sup^*$ is the essential sup with respect to Lebesgue measure. 
 This lemma quantifies the following facts : for functions in $\mathcal P (\Omega_T) $, 
 \begin{itemize}
 \item convergence in $L^1 (\Omega_T)$   implies  convergence of their slices in $L^1 (\Omega)$;
 \item boundedness in $L^1 (\Omega_T)$   implies  compactness of their slices in $L^1 (\Omega)$.
 \end{itemize}

 \begin{proof} 
  Since  $u, v$ are uniformly Lipschitz in $[T_0,S] \times \Omega$,  we deduce that for any $T_0 \leq t \leq S$, and  $T_0\leq s \leq   S$,
$$
 \vert u (t,x) - v (t,x)\vert    \leq    \kappa \vert s - t\vert + \vert u (s,x) - v (s,x)\vert,
$$
where  $\kappa > 0$ is the uniform Lipschitz constant of $u - v$ on $[T_0,S]$. 
We infer
$$
\int_\Omega \vert u (t,z) - v (t,z)\vert  d V (z)  \leq    
\kappa \vert s - t\vert {\rm Vol}(\Omega)+ \int_\Omega \vert u (s,z) - v (s,z)\vert d V (z).
$$
Thus the function
$$
t \mapsto \theta (t) := \int_\Omega \vert u (t,z) - v (t,z)\vert  d V (z),
$$
is a Lipschitz function in $[T_0,S]$ with  Lipschitz constant $ \kappa {\rm Vol}(\Omega)$. 
The  conclusion follows from the next lemma, an  elementary result in one real variable.
\end{proof}

We have used the following  inequality :

\begin{lem}
Fix $0<S_0 < S_1 < S$ and let $\theta : [S_0,S] \longrightarrow \R$ be such that for all 
$s, \s \in [S_0,S]$ with $s\leq \sigma$,
$
 \theta (s) \leq \theta (\sigma) + \kappa (\sigma - s).
$
Then 
$$
 \max_{S_0 \leq s \leq S_1} \theta (s)  \leq 2  M \max \{\sqrt{\Vert \theta \Vert}, \Vert \theta\Vert \},
$$
where $M :=  \max \{\sqrt{ \kappa}, (S - S_1)^{- 1}\}$ and  $\Vert \theta \Vert := \Vert \theta \Vert_{L^1 ([S_0,S])}$.
\end{lem}

\begin{proof}
Fix $0<  \delta \leq S - S_1$. Then for  $\s, s \in [S_0,S_1]$ with $s\leq \sigma$,
$$
\theta (s) \leq \theta (\s) +  \kappa (\sigma - s).
$$
Fix $S_0 \leq s \leq S_1$. Integrating in $\s $ in $[s , s + \delta] \sub [S_0,S ]$,  we get 
\begin{equation} \label{eq:Sup-L1ineq}
\theta (s)  \leq  \frac{ \kappa \, \delta}{2}  +  \int_{s}^{s + \delta}  \theta (\s) \frac{d \s }{\delta} \\
 \leq   \frac{\delta \kappa}{2}  \,   + \delta^{- 1} \Vert \theta \Vert.
\end{equation}
The minimum  of $ \tau \longmapsto  \kappa \tau/2    + \tau^{- 1} \Vert \theta \Vert$
is achieved  at $\tau_0 := \sqrt{2}\Vert \theta \Vert^{1 \slash 2} \slash \sqrt {\kappa}$.
If $ 2\Vert \theta \Vert \leq \kappa  (S - S_1)^2 $ i.e. $\tau_0 \leq S - S_1$, then 
$
\theta (t) \leq 2 \sqrt{\kappa \Vert \theta\Vert},
$
 for $t \in [S_0, S_1]$.
If  $ 2\Vert \theta \Vert \geq \kappa  (S - S_1)^2 $, applying  (\ref{eq:Sup-L1ineq}) with $\delta = S - S_1$ yields
$$
\max_{S_0 \leq t \leq S_1} \theta (t) \leq \kappa (S - S_1)/2 + \Vert \theta\Vert (S - S_1)^ {- 1} \leq  2  \Vert \theta\Vert (S - S_1)^ {- 1}.
$$
Altogether we obtain
$$
 \max_{S_0 \leq t \leq S_1} \theta (t)  \leq 2  \max \{\sqrt{ \kappa \Vert \theta \Vert}, \Vert \theta\Vert (S - S_1)^ {- 1}\}.
$$
\end{proof}

 \subsection{Time derivatives and semi-concavity}
 
 In this section we observe that a parabolic potential $\f$ has well defined time derivatives $\partial_t \f$  almost everywhere. Let $\ell$ denote the Lebesgue measure on $\R$, and fix a positive Borel measure  $\mu$ on $\Omega$.

\begin{lem} \label{lem:Radmacher}
  Fix $\f \in \mathcal P (\Omega_T)$. 
Then there exists a Borel set $E \subset \Omega_T$ 
 $\ell \otimes \mu$-negligible such that $\partial_t \f (t,z)$ exists for all $(t,z) \notin E$.

 In particular  $\partial_t \f \in L_{\loc}^{\infty} (\Omega_T)$ 
 and for any continuous function 
 $\gamma \in {\mathcal C}^0(\R,\R)$, \, $ \gamma (\partial_t \f) \, \ell \otimes \mu$ is a well defined Borel measure in $\Omega_T$.  
\end{lem}

\begin{proof}
We set, for $(t,z) \in \Omega_T$,     
 $$
\partial_t^u \f  (t,z) := \limsup_{s \to 0} \frac{\f (t + s,z) - \f (t,z)}{s} =\limsup_{\mathbb{Q}^* \ni s\to 0} \frac{\f(t+s,z)-\f(t,z)}{s},
 $$
 and 
 $$
\partial_t^l \f (t,z) :=\liminf_{s \to 0} \frac{\f (t + s,z) - \f (t,z)}{s}=\liminf_{\mathbb{Q}^* \ni s\to 0} \frac{\f(t+s,z)-\f(t,z)}{s}.
 $$
The  equalities above follow from the Lipschitz property of $\varphi$.
These two functions are measurable in $(\Omega_T,\ell \otimes \mu)$, hence the set
$$
E:= \{(t,x) \in \Omega_T \setdef  \partial^l \varphi(t,z) < \partial^u \varphi (t,z)\}
$$ 
is $\ell \otimes \mu$-measurable.  
For each $z \in \Omega$ fixed, $t\mapsto \varphi(t,z)$ is locally Lipschitz hence differentiable almost everywhere in $]0,T[$. Hence, for all $z\in \Omega$,
$$
E_z:= \{t\in ]0,T[  \setdef (t,z) \in E\}
$$ 
has zero $\ell$-measure. Fubini's theorem thus ensures that $\ell \otimes \mu (E)=0$.
\end{proof}  

The previous lemma  shows that $\partial_t^u \f = \partial_t^l \f $, $\ell \otimes \mu$-almost everywhere in $\Omega_T$. 
These thus define a function which we denote by $\partial_t \f \in L_{\loc}^{\infty} (\Omega_T)$.

 When $\f$ is semi-concave (or semi-convex) in $t$, we can improve  the previous result.
 
 \begin{defi}
 We say that  $\f : \Omega_T \longrightarrow \R $ is  uniformly semi-concave in $]0,T[$ if for any compact $J \Subset ]0,T[$, there exists  
 $\kappa = \kappa (J,\f) > 0$ such that for all $z \in \Omega$, the function  $t \longmapsto \f (t,z) - \kappa t^ 2$ is concave in $J$. 
 \end{defi}
 The definition of uniformly semi-convex functions is analogous.
Note that such functions are automatically locally uniformly Lipschitz.

 \begin{lem} \label{lem: left and right derivative}  
 Let $\f : \Omega_T \longrightarrow \R $ be a continuous function which is uniformly semi-concave
 in $]0,T[$. Then $(t,z) \mapsto \partial_t^- \f (t,z)$
  is upper semi-continuous  while $(t,z) \mapsto \partial_t^+ \f (t,z)$
 is lower semi-continuous  in $\Omega_T$.
 In particular, there exists a Borel set $E \subset \Omega_T$ which is $\ell\otimes \mu$-negligible,  such that $\partial_t^+ \f$ and $\partial_t^- \f$ coincide and are continous at each point in $\Omega_T \setminus E$. 
 \end{lem}
 
  By replacing $\f$ with $-\f$ one obtains similar conclusions for uniformly semi-convex functions.

   For a bounded parabolic potential $\varphi$ which is locally  semi-concave in $t$  the left and right derivatives 
 $$
 \partial_t^+ \f (t,z) = \lim_{s \to 0^+} \frac{\f (t + s,z) - \f (t,z)}{s},
 $$ 
 and
 $$
 \partial_t^- \f (t,z):= \lim_{s \to 0^-} \frac{\f (t + s,z) - \f (t,z)}{s}
 $$
 exist for all $t\in ]0,T[$, and $\partial_t \varphi(t,z)$ exists if $ \partial_t^+ \f (t,z) = \partial_t^- \f (t,z)$.

 \begin{proof} 
  For simplicity we only treat the semi-convex case.
  It suffices to consider the case when  $t \mapsto \f (t,z)$ is convex  in $]0,T[$, for all $z \in \Omega$.
In this case for all $(t,z) \in \Omega_T$, the slope function
  $$
   s \longmapsto p_s (t,z) := \frac{\f (t + s,z) - \f (t,z)}{s},
  $$
  is monotone increasing on each interval not containing $0$. It is moreover continuous in $(t,z)$.
 In particular, 
  $$
  \partial_t^+ \f (t,z) = \lim_{s \to 0^ +} p_s (t,z) = \inf_{s > 0} p_s (t,z),
  $$
  is upper semi-continuous in $\Omega_T$
  and
  $$
  \partial_t^- \f (t,z) = \lim_{s \to 0^-} p_s (t,z) = \sup_{s < 0} p_s (t,z),
  $$
  is lower semi-continuous in $\Omega_T$. 
  This proves the first part of the lemma.
  
  The second part follows from  the fact that convex functions are locally Lipschitz
  in their domain and Lemma~\ref{lem:Radmacher}.
 \end{proof}

\subsection{Compactness properties}

 We introduce a natural complete metrizable topology on the convex set $\mathcal P (\Omega_T)$. 
 
We recall the definition of the  Sobolev space $W^{1,0}_{ \infty, \loc} (\Omega_T)$ : 
this is the set of functions $u \in L^1_{\loc} (\Omega_T)$ 
whose  partial time derivative  (in the sense of distributions) satisfies  
$\partial_t u \in L^{\infty}_{\loc} (\Omega_T)$. 
It follows from Lemma~\ref{lem:Radmacher} that
  $$
  \mathcal P (\Omega_T)  \subset W^{1,0}_{ \infty, \loc} (\Omega_T).
  $$
  
Let    $K \subset \Omega$ be a compact subset. The local uniform Lipschitz constant of $\f \in \mathcal P (\Omega_T) $  on a  compact subinterval $J \Subset ]0,T[$ is given by
  $$
  {\sup_{t, s \in J, s \neq t} {\sup_{z \in K}}^*} \frac{\vert \f (s,z) - \f (t,z)\vert}{\vert s - t\vert} = 
  \Vert \partial_t \f \Vert_{L^{\infty} (J \times \Omega)}.
  $$
  
  \begin{defi}
   We endow $\mathcal P (\Omega_T)$ with the semi-norms associated to  $W^{1,0}_{\infty, \loc} (\Omega_T)$:
given a compact subset $J \Sub ]0,T[$ and $u \in W^{1,0}_{ \infty, \loc} (\Omega)$, we set
$$
 \rho_{J,K} (u) := \Vert \partial_t u \Vert_{L^{\infty} (J \times K)}
 + \int_J \int_{K} \vert u (t,z) \vert d V (z) d t.
$$
 \end{defi}
 
The spaces $L^q (\Omega_T)$ are defined with respect to the $(2 n + 1)$-dimensional  Lebesgue measure in $\Omega_T$. 
For  $k, \ell \in \N$ and   $q \geq 1$, we   denote by $ W^{k,\ell}_{q,\loc} (\Omega_T)$ the Sobolev space of Lebesgue measurable functions whose partial derivatives with respect to $t$ up to order $k$ and partial derivatives with respect to $z$ up to order $\ell$ in the sense of distributions are in $L^q_{\loc} (\Omega_T)$.

\smallskip

Parabolic potentials enjoy useful compactness  properties :

\begin{pro} \label{pro: Montel property of P}   
Let $(\varphi_j) \subset \mathcal P (\Omega_T)$ be a sequence which 
\begin{itemize}
\item is locally uniformly bounded from above in $\Omega_T$;
\item is locally uniformly Lipschitz in $]0,T[$;
\item does not converge locally uniformly to $- \infty$ in $\Omega_T$. 
\end{itemize}

Then  $(\f_j)$ is bounded in $L^1_{\loc} (\Omega_T)$ and there exists a subsequence which converges to some function 
$\varphi \in  \mathcal P (\Omega_T)$ in $L^1_{\loc} (\Omega_T)$-topology.

If  $(\f_j)$ converges weakly  to $\f$ in the sense of distributions in $\Omega_T$, then it converges in $L^q_{\loc} (\Omega_T)$  for all $q \geq 1$. 
\end{pro}

The proof is an extension of Hartog's lemma for sequences of plurisubharmonic functions (see e.g. \cite[Theorem 1.46]{GZbook}).

\begin{proof}
We first prove that $(\varphi_j)$ is bounded in $L^1_{\loc}(\Omega_T)$. Fix $J\Subset ]0,T[$ and $K\Subset \Omega$. From the assumptions it follows that, for each $t\in J$ fixed, $\varphi_j(t,\cdot)$ does not converge locally uniformly in $\Omega$ to $-\infty$. Hence $\varphi_j(t,\cdot)$ is bounded in $L^1_{\loc}(\Omega,dV)$.  The second condition thus ensures that $\{\varphi_j\}$ is uniformly bounded in $L^1(J\times K)$.

For each $r\in \mathbb{Q}\cap ]0,T[$, there exists a subsequence of $\varphi_j(r,\cdot)$ which converges in $L^1_{\loc}(\Omega)$ to some plurisubharmonic function  $\varphi(r,\cdot)$ in $\Omega$.
After a Cantor process we can extract a subsequence from $\{\varphi_j\}$, still denoted by $\{\varphi_{j}\}$, such that for each $r\in \mathbb{Q}\cap ]0,T[$, the sequence $\{\varphi_j(r,\cdot)\}$ converges in $L^1_{\loc}(\Omega)$ to $\varphi(r,\cdot)$. Since the sequence $\{\varphi_j\}$ is locally uniformly Lipschitz in $t$, it follows that the function $(r,z)\mapsto \varphi(r,z)$ is also locally uniformly Lipschitz in $r$. The function $\varphi$ therefore uniquely extends to $]0,T[\times \Omega$ by
\[
\varphi(t,z) := \lim_{\mathbb{Q}\ni r\to t} \varphi(r,z). 
\]
Since $\{\varphi_j\}$ is uniformly Lipschitz in $t$ it follows that $\{\varphi_j(t,\cdot)\}$ converges in $L^1_{\loc} (\Omega)$ to $\varphi(t,\cdot)$, for all $t\in ]0,T[$ and $\varphi$ is locally uniformly Lipschitz in $t\in ]0,T[$. The latter then implies that $\varphi \in \mathcal{P}(\Omega_T)$.  By Fubini's theorem and dominated convergence it follows that  $\{\varphi_j\}$ converges in $L^1_{\loc}(\Omega_T)$ to $\varphi$. 

We now prove the last statement, assuming that $\varphi \in \mathcal{P}(\Omega_T)$ and that the sequence $\{\varphi_j\}$ converges in the weak sense of distributions to $\varphi$. We claim that for each $t\in ]0,T[$,  $\{\varphi_j(t,\cdot)\}$ converges in the sense of distributions in $\Omega$ to $\varphi(t,\cdot)$. Indeed, fix $t_0\in ]0,T[$ and let $\chi: \Omega\rightarrow \mathbb{R}$ be a smooth test function in $\Omega$. Let $\varepsilon>0$ be a small constant and let $\eta_{\varepsilon} : \mathbb{R}\rightarrow \mathbb{R}^+$ be a smooth test function which is supported in $[t_0-\varepsilon,t_0+\varepsilon]$ and such that $\int_{\mathbb{R}} \eta_{\varepsilon}(t)dt=1$. By assumption, 
\begin{equation}\label{eq: Lp convergence 1}
\lim_{j\to +\infty} \int_{\Omega_T} \varphi_j(t,z) \chi(z)\eta_{\varepsilon}(t) dtdV(z) = \int_{\Omega_T} \varphi(t,z) \chi(z)\eta_{\varepsilon}(t) dtdV(z).
\end{equation}
Since the sequence $\{\varphi_j\}$ is locally uniformly Lipschitz in $t$, there exists a constant $\kappa_0$ depending on $\varepsilon_0:= \min(t_0,T-t_0)\slash 2$ such that 
\[
|\varphi_j(t,z)-\varphi_j(t_0,z)| + |\varphi(t,z)-\varphi(t_0,z)| \leq \kappa_0|t-t_0|, 
\]
for all $t\in [t_0-\varepsilon_0,t_0+\varepsilon_0]$ and $\ z\in \Omega$. We infer
\begin{flalign}
&\left | \int_{\Omega_T} \varphi_j(t,z) \chi(z)\eta_{\varepsilon}(t) dtdV(z)  - \int_{\Omega_T} \varphi_j(t_0,z) \chi(z)\eta_{\varepsilon}(t) dtdV(z)\right| \nonumber \\
& \ \ \ \ \ \ \ \ \ \ \ \ \ \ \ \ \leq  \kappa_0 \varepsilon \int_{\Omega} |\chi(z) |dV(z).
\label{eq: Lp convergence 2}
\end{flalign}
The same estimate holds for $\varphi$. 
Combining \eqref{eq: Lp convergence 1} 
and \eqref{eq: Lp convergence 2} yields
\[
\lim_{j\to +\infty} \int_{\Omega} \varphi_j(t_0,z) \chi(z) dV(z) = \int_{\Omega} \varphi(t_0,z) \chi(z)dV(z) +O(\varepsilon).
\]
We finally let $\varepsilon\to 0$ to conclude the proof of the claim. 

Classical properties of plurisubharmonic functions now ensure  that  $\{\varphi_j(t_0,\cdot)\}$ converges in $L^q_{\loc}(\Omega)$ to $\varphi(t_0,\cdot)$. Since  $\{\varphi_j\}$ is locally uniformly  Lipschitz in $t$, we  conclude as above that $\{\varphi_j\}$ converges in $L^q_{\loc}(\Omega_T)$ to $\varphi$. 
\end{proof}

\begin{cor}
The class $\mathcal P (\Omega_T)$ is a  subset of $ L^q_{\loc} (\Omega_T)$ for all $q \geq 1$, and the inclusions
$\mathcal P (\Omega_T) \hookrightarrow  L^q_{\loc} (\Omega_T)$ are continuous. 
\end{cor}

The weak topology and the $L^q_{\loc}$-topologies  are thus all equivalent when restricted to the class $\mathcal P (\Omega_T)$. The set $\mathcal{P}(\Omega_T)$ is thus a complete metric space when endowed with any of these topologies.  

\begin{lemma}
	\label{lem: Parabolic potentials are in W11loc}
	We have $\mathcal{P}(\Omega_T) \subset W^{1,1}_{\loc}(\Omega_T)$. 
\end{lemma}
\begin{proof}
	Fix $u\in \mathcal{P}(\Omega_T)$. The goal is to prove that $u$ has partial derivative (in $t$ and $z$) in $L^1_{\loc} (\Omega_T)$.

We first recall a basic estimate for the gradient of a plurisubharmonic function. Fix $z_0 \in \Omega$ and $r>0$ such that the polydisc $D(z_0,2r)$ is contained in $\Omega$. It follows from  \cite[Theorem 4.1.8]{Hor07} (see also \cite[Theorem 1.48]{GZbook} and its proof at page 32-33) that the derivative of any plurisubharmonic function $z\mapsto \varphi(z)$  exists in $L^p_{\loc}(\Omega)$ for any $p<2$ and the uniform estimate  
\[
\left(\int_{D(z_0,r)} |\nabla_z \varphi |^p dV\right)^{1 \slash p} \leq  C(p, r) \int_{D(z_0,2r)} |\varphi| dV  
\]
holds for a positive constant $C(p,r)$ depending only on $r$ and $p$.

Fix $J\times K$  a compact subset of $\Omega_T$. Then by our previous analysis and the compactness of $K$ there exists a constant $C>0$  depending on $K$ and 
${\rm dist}(K,\partial \Omega)$  and a compact subset $K\Subset L \Subset \Omega$ such that 
\[
\left(\int_{K} |\nabla_z \varphi |^p dV\right)^{1 \slash p} \leq  C \int_{L} |\varphi| dV,
\] 
for every $\varphi \in \PSH(\Omega)$.   

Now,  for each $t\in ]0,T[$ the derivative of $u$ in $z$ exists and belongs to $L^p_{\loc}(\Omega)$ for any $p<2$ (with uniform bound). Since $u$ is locally uniformly Lipschitz in $t$ it follows that $\partial_t u(t,z)$ is bounded in $J\times K$ and  $u\in L^1(J \times L, dt dV)$. Altogether we obtain  $u\in W^{1,1}_{\loc} (\Omega_T)$ as desired. 
\end{proof}

\section{Parabolic Monge-Amp\`ere operators}  \label{sec: MAP operator}

\subsection{Parabolic Chern-Levine-Nirenberg inequalities}

We assume here that $\f \in \mathcal P (\Omega_T) \cap L^{\infty}_{\loc} (\Omega_T)$. For all $t \in ]0,T[$, the  function 
$$
\Omega \ni z \mapsto \f_t(z) = \f (t,z) \in \R
$$
 is psh and locally bounded, hence the  Monge-Amp\`ere measures $(dd^c \varphi_t)^n$
are well defined  Borel measures in the sense of Bedford and Taylor \cite{Bedford_Taylor_1976Dirichlet}.
  
 We now show that this family depends continuously on $t$ :

 \begin{lem}  \label{lem: continuity of MA integral against test function}  
 Fix $\f \in \mathcal P (\Omega_T)\cap L^{\infty}_{\loc} (\Omega_T)$ and   
 $\chi $  a continuous test function in $\Omega_T$. Then the function
$$
\Gamma_{\chi} : t \longmapsto \int_{\Omega} \chi (t,\cdot) (dd^c \f_t)^n
$$
is  continuous  in  $]0,T[$. 
Moreover if  ${\rm Supp} (\chi) \Subset E_1 \Subset E_2 \Subset \Omega_T$, then 
\begin{equation} \label{eq: uniform}
\sup_{ 0 \leq t < T} \left \vert \int_{\Omega} \chi (t,\cdot)(dd^c \varphi_t)^n \right \vert 
\leq  C \max_{\Omega_T} |\chi| (\max_{ E} \vert \varphi \vert)^n,
\end{equation}
 where $C > 0$ is a constant depending only on $(E_1,E_2,\Omega_T)$.
 \end{lem}
 
In particular,  $t \longmapsto (dd^c  \f_t)^n$ is continuous,
 as a map from $]0,T[$ to the space of positive Radon measures in $\Omega$ endowed with the weak$^*$-topology.
 
 \begin{proof}
 We can reduce to the case when the support of $\chi$ is contained in a product of compact subsets
$J \times K \subset E^{\circ} \subset ]0,T[ \times X$.

We first prove \eqref{eq: uniform}. For any fixed $t \in ]0,T[$,  
$$
\left \vert \int_{\Omega} \chi (t,\cdot) (dd^c \f_t)^n \right \vert \leq  \max_{\Omega_T} \vert \chi  \vert  \int_{K} (dd^c \f_t)^n.
$$ 
Let $L \Subset \Omega$ be a compact subset such that $K \Subset L^{\circ}$ and $J \times L \subset E$. 
The classical Chern-Levine-Nirenberg inequalities (see \cite[Theorem 3.9]{GZbook}) ensure that
there exists a constant $C = C (K,L) > 0$ such that 
$$
 \int_{\Omega} \chi_t  (dd^c \f_t)^n \leq C \max_{\Omega}|\chi|  (\max_{L} \vert \f_t \vert)^n
  \leq  C \max_{\Omega_T} |\chi|  (\max_E \vert \f \vert)^n,
 $$
 where $C > $ depends only on $(K, E,\Omega_T)$. This yields \eqref{eq: uniform}.
 \smallskip

 We now prove that $\Gamma_{\chi}$ is continous in $]0,T[$. Fix compact sets $J\Subset ]0,T[, K\Subset \Omega$ such that ${\rm Supp}(\chi) \subset J \times K$. The continuity of $\Gamma_{\chi}$ on $]0,T[\setminus J$ is clear. Fix now $t_0\in J$. By the CLN inequality (see \cite[Theorem 3.9]{GZbook}),
 $$
\lim_{t\to 0} \int_{\Omega} |\chi(t,\cdot)-\chi(t_0,\cdot)|(dd^c \varphi_t)^n =0. 
 $$
 Since $\chi$ is  a continuous  test function we also have 
 $$
 \lim_{t\to t_0} \int_{\Omega} \chi(t_0,\cdot) (dd^c \varphi_t)^n = \int_{\Omega} \chi(t_0,\cdot) (dd^c \varphi_{t_0})^n.
 $$
 This proves the continuity of $\Gamma_{\chi}$ at $t_0$, finishing the proof.
\end{proof}

 \begin{defi}  
 Fix $\f \in \mathcal P (\Omega)\cap L^{\infty}_{\loc} (\Omega_T)$. 
The map
 \begin{equation*} 
\chi \mapsto  \int_{\Omega_T} \chi d t\wedge (dd^c \f)^n  := \int_0^T d t \left(\int_{\Omega} \chi (t,\cdot)  ( dd^c \f_t)^n\right).
 \end{equation*}
defines a positive distribution in $\Omega_T$ denoted by $ d t \wedge  (dd^c \f)^n$, which
can be  identified with  a positive Radon measure in $\Omega_T$.
 \end{defi}

 \begin{pro}   \label{prop: dt MA converge}
 Fix $\f \in \mathcal P (\Omega_T)\cap L^{\infty}_{\loc} (\Omega_T)$ and let 
 $(\f^j)$ be a monotone sequence of  functions  in $\mathcal P (\Omega_T)\cap L^{\infty}_{\loc} (\Omega_T)$  
converging to $\f$ almost everywhere in $\Omega_T$. Then 
$$
d t \wedge  (dd^c\f^j)^n \to d t \wedge (dd^c \f)^n
$$
in the weak sense of measures in $\Omega_T$. 
\end{pro}

\begin{proof} 
Let $\chi$ be a continuous test function in $\Omega_T$. By definition,
$$
 \int_{\Omega_T} \chi d t \wedge  (dd ^c \f^j)^n = \int_0^T d t \left(\int_{\Omega} \chi (t,\cdot)  (dd^c \f^j(t,\cdot))^n \right)=: \int_0^T F_j(t)dt.
$$

It follows from \cite[Theorem 2.1 and Proposition 5.2]{Bedford_Taylor_1982Capacity} that $F_j$ converges to $F$ pointwise in $]0,T[$.  Lemma \ref{lem: continuity of MA integral against test function} ensures that $F_j$ is uniformly bounded hence the conclusion follows from Lebesgue  convergence theorem. 
\end{proof}

\begin{remark}\label{rem: dt MA converge}
	The conclusion of Proposition \ref{prop: dt MA converge} also holds if the sequence $(\varphi^j)$ uniformly converges to $\varphi\in \mathcal{P}(\Omega_T)$.
\end{remark}

\subsection{Semi-continuity properties}

It is difficult to pass to the limit in the parabolic equation, due to the time derivative.
We have the following  general semi-continuity property :

\begin{lem}  \label{lem: semicontinuity}  
   Let $(\nu_j)$ be positive Borel measures on a topological manifold $Y$ which converge weakly to $\nu$ in the sense of Radon measures on $Y$.    
   Let $v_j : Y \longrightarrow \R$ be a locally uniformly bounded  sequence of measurable functions  which  weakly converge to a measurable function $v$ 
   in $L^2 (Y,\nu)$.
 \begin{enumerate}
 \item If $\Vert \nu_j - \nu \Vert \to 0$ (total variation) then $\lim_j \int_Y v_j  \nu_j  = \int_Y v  \nu$ and 
 $$
 \liminf_{j \to + \infty} e^{v_j} \nu_j \geq e^v \nu
 $$
 in the weak sense of Radon measures in $Y$.
 
 \item If  $v_j \to v$ $\nu$-a.e. in $Y$  and $\mathcal M := \{\nu_j ; j \in \N \} \cup \{\nu \}$ 
 is uniformly absolutely continuous with respect to a fixed positive Borel measure $\tilde \nu$ on $Y$, 
 then for any continuous function $\theta : \R \to \R$,
 $$
  \theta (v_j) \nu_j \longrightarrow  \theta (v) \nu 
 $$
as $j \rightarrow +\infty$, in the weak sense of Radon measures in $Y$.
 \end{enumerate}
 \end{lem}

 Recall that a set $\mathcal M$ of positive Borel measures  is uniformly absolutely continuous with respect to a  positive Borel measure 
 $\tilde \nu$ on $Y$ if for any $\delta > 0$ there exists $\alpha > 0$ such that  $\sup_{\sigma \in \mathcal{M}} \sigma (B) \leq \delta$
whenever  $B \subset Y$ is a Borel subset with $\tilde \nu (B) \leq \alpha$
 
 A typical example is when  $\sigma = f_{\sigma} \tilde  \nu$, where $\sup_{\sigma \in \mathcal{M}} f_{\sigma} $ is $\tilde \nu$-integrable.
When $\Vert \nu_j - \nu \Vert \to 0$ in the sense of total variation, then the set 
$\mathcal{M} := \{ \nu_j \setdef j \in \N \} \cup \{ \nu \}$ is uniformly absolutely continuous with respect to $\nu=\tilde  \nu$.
 
\begin{proof} 
We first prove (1). Recall  Young's formula which states that 
 $$
 e^t = \sup_{s >0} \{ s t - s \log s + s \}
 $$
 for all $t \in \R$.
It therefore suffices to prove that for all $s  > 0$,
 $$
 \liminf_{j \to + \infty} e^{v_j} \nu_j \geq (s v - s \log s + s) \nu
 $$
 in the weak sense of Radon measures on $Y$. Now  for all $s > 0$ 
  $$
   e^{v_j} \nu_j =\sup_{s >0} \{ (s v_j - s \log s + s) \nu_j  \},
  $$
 so it suffices to prove that $\liminf_j v_j \nu_j \geq v \nu$ in the sense of Radon measures.
 
  Let $\chi$ be a  test function on $Y$.  
 Observe that
  $$
\int_Y  \chi v_j d \nu_j - \int_Y \chi  v d \nu = \int_Y \chi (v_j - v) d \nu + \int_Y \chi v_j d ( \nu_j - \nu).
  $$
The  first term converges to zero by weak convergence.
Since $\chi v_j$ is uniformly bounded by a constant $M$ the absolute value of the second term is less than $M  \Vert \nu_j - \nu\Vert_{{\rm Supp} (\chi)}$, which converges to $0$.
 
\smallskip

We now prove (2). Set $f_j := \theta (v_j)$ and $f:=  \theta (v)$ and write
  $$
\int_Y  \chi f_j d \nu_j - \int_Y \chi  f d \nu = \int_Y \chi (f_j - f) d \nu_j + \int_Y \chi f d ( \nu_j - \nu).
  $$
Observe that $g_j := \chi (f - f_j) \to 0$ $\nu$ a.e. in $Y$ 
since $v_j \to v$ $\nu$-a.e. in $Y$. It follows from Egorov's theorem that the sequence $(f_j)$ converges $\tilde{\nu}$-quasi uniformly to $f$. Since the sequence $(\nu_j)$ is uniformly absolutely continuous with respect to $\tilde{\nu}$ it follows that the first term above converges to $0$ as $j\to +\infty$.  By Lusin's theorem, the function  $f$ is $\tilde{\nu}$-quasi continuous in $Y$, hence the second term also converges to $0$ as $j\to +\infty$, completing the proof of the lemma. 
 \end{proof}

\begin{pro}  \label{prop: semi continuity}
 Let $J \subset \R$ be a bounded open interval, $D$ be a bounded open set in $\mathbb{R}^m$, $m\in \N^*$, and $0 \leq g \in L^p (D)$ with $p> 1$.  Let  $(\psi_j)$ be a sequence of Borel functions in $J\times D$ such that $(e^{\psi_j}g)$ is uniformly bounded in $L^1(J\times D, dt d V)$. Assume that there exists $E\subset D$ with zero Lebesgue measure such that for all $z\in D\setminus E$,  $\psi_j(\cdot, z)$ converge to a bounded Borel function $\psi(\cdot,z)$   in the sense of distributions on $J$ and 
 \begin{equation}
 	\label{eq: uniform bound psij}
 	\sup_{j\in \mathbb{N},z \in D\setminus E} \left |\int_J \chi(t,z) \psi_j(t,z) dt \right|  <+\infty, \ \ \text{for all}\ \chi \in \mathcal{C}^{\infty}_0(J\times D). 
 \end{equation}  
Then for any positive smooth test function $\chi \in \mathcal{C}^{\infty}_0(J\times D)$, 
\begin{equation}
	\label{eq: semi-continuity}
\int_{J\times D} \chi(t,z) e^{\psi(t,z)} g(z) dt d V  \leq \liminf_{j \to + \infty} \int_{J\times D} \chi(t,z)e^{\psi_j(t,z)} g(z) dt d V.
\end{equation}
\end{pro}

\begin{proof}
Fix $C>0$ such that $|\psi| \leq C$ in $X$.  
Set $\varphi_j:= \max(\psi_j,-C)$, $j\in \mathbb{N}$.  Then $e^{\varphi_j}$ is uniformly bounded in $L^1(J\times D,gdtdV)$. 
It follows that $(\varphi_j)$ is bounded in $L^2 (J\times D,gdtdV)$. 
Up to extracting and relabelling, it follows from Banach-Saks theorem that 
the arithmetic  mean sequence
$$
\Psi_N :=\frac{1}{N}\sum_{j=0}^N \varphi_j
$$
converges almost everywhere and in $L^2(J\times D,g dtdV)$ towards a function $\Psi \in L^2 (J\times D, g dtdV)$.

Condition \eqref{eq: uniform bound psij} and Lebesgue's theorem ensure that $\psi_j g$ converges in the sense of distributions on $J\times D$ to $\psi g$. This together with the convergence of $\Psi_N$ towards $\Psi$ ensure that for any  positive smooth test function $\chi$ in $J\times D$, 
\begin{flalign*}
	\int_{J\times D}\chi \Psi gdt dV & =  \lim_{N\to +\infty} \int_{J\times D} \chi \Psi_N gdt dV
	\\
	&=  \lim_{N\to +\infty} \frac{1}{N} \sum_{j=1}^N\int_{J\times D} \chi \max(\psi_j,-C) gdtdV\\
	&\geq   \lim_{N\to +\infty} \frac{1}{N} \sum_{j=1}^N\int_{J\times D} \chi \psi_j gdt dV 
	= \int_{J\times D}  \chi \psi g dt dV.
\end{flalign*}
This implies that   $\Psi g \geq \psi g$ in $L^1 (J\times D)$, hence $e^{\Psi} g \geq e^{\psi} g$ in $L^1(J\times D)$.  

It thus follows from Fatou's lemma that 
$$
\liminf_{N\to +\infty} \int_{J\times D} e^{\Psi_N} \chi gdtdV \geq \int_{J\times D} e^{\Psi} \chi gdtdV \geq \int_{J\times D} e^{\psi} \chi gdtdV.
$$
It follows now from the convexity of the exponential   that  
\begin{flalign*}
	\int_{X} e^{\Psi_{N}}  \chi g dt d V  & \leq  \frac{1}{N} \sum_{j=1}^N \int_{X} e^{\varphi_{j}} \chi g d\mu\\
	&  \leq \frac{1}{N} \sum_{j=1}^N \int_{X} e^{\psi_j} \chi g d\mu +   \int_X e^{-C} \chi g dt dV, 
\end{flalign*}
hence letting $N\to +\infty$ we get
$$
\int e^{\Psi}  \chi g dtdV  \leq \liminf_{j \to + \infty} \int_{X} \chi  e^{\psi_j} gdt dV + e^{-C}\int_X\chi g ddt dV.
$$
Letting $C\to +\infty$ we obtain \eqref{eq: semi-continuity}.
\end{proof}

\subsection{Semi-concavity and convergence}

In the sequel we need more precise convergence results
which require stronger assumptions :
 
 \begin{defi}  
 A function $\gamma : I \rightarrow \R$ is  $\kappa$-concave if $t \mapsto \gamma (t) - \kappa t^2$ is concave. It is called locally semi-concave in $I$ if for any subinterval $J \subset I$, there exists $\kappa = \kappa (J,\gamma) > 0$ such that $\gamma$ is $\kappa$-concave in $J$.
  \end{defi}
  
  A family $\mathcal A$ of semi-concave  functions in some interval $I \subset \R$  is called {\it locally uniformly semi-concave}  if for any compact subinterval $J \Subset I$,  there exists a constant  
  $\kappa = \kappa (J,\mathcal A) > 0$ such that any $\gamma \in \mathcal A$   is  $\kappa$-concave in $J$. 

\smallskip
 
 The following elementary lemma is useful :

 \begin{lem} \label{lem: convergence derivative of concave functions} 
 Let $(\gamma_j)$ be  a  sequence of uniformly  semi-concave functions in an interval $I \subset \R$ 
 which converges  pointwise to a function $\gamma$. 
   Then there exists a countable subset $S \subset I$ such that for all $t \in  I \setminus S$, 
 the derivatives $\dot{\gamma}_j (t), \dot{\gamma} (t)$ exist  and $ \lim_{j \to + \infty} \dot{\gamma}_j (t) = \dot{\gamma} (t)$.
 Moreover if  $\dot{\gamma} (t_0)$ exists then
 $$
  \lim_{j \to + \infty} \partial_t^- \gamma_j (t_0) =  \lim_{j \to + \infty} \partial_t^+\gamma_j (t_0) = \dot{\gamma} (t_0).
 $$
 \end{lem}
 
 We  include a proof for the reader's convenience. 
 
 \begin{proof} 
 We can assume that $\gamma_j$ is concave in $I$ for all $j$ and $t_0 = 0$.
 Thus for all $j \in \N$ and $t < 0$, 
 $$
   t \partial_t^-\gamma_j (0) \geq \gamma_j (t)-\gamma_j(0).
 $$
 
Dividing by $t < 0$ and taking limits (first $j \to + \infty$, then $t \to 0^-$), we obtain 
 $
 \partial_t^-\gamma (0) \geq \limsup_{j \to + \infty} \partial_t^-\gamma_j (0).
 $
Similarly
 $
  \liminf_{j \to + \infty} \partial_t^+\gamma_j (0) \geq \partial_t^+\gamma (0).
 $
 Since $\partial_t^-\gamma_j (0) \geq \partial_t^+\gamma_j (0)$ we conclude that
 $$
 \partial_t^-\gamma (0) \geq \limsup_{j \to + \infty} \partial_t^-\gamma_j (0) \geq \liminf_{j \to + \infty} \partial_t^+\gamma_j (0) \geq  \partial_t^+\gamma (0).
 $$
 
 If $\dot{\gamma}(0)$ exists, $\partial_t^-\gamma (0) = \dot{\gamma}(0) = \partial_t^+\gamma (0)$, hence
 $$
  \lim_{j \to + \infty} \partial_t^-\gamma_j (0) =  \lim_{j \to + \infty} \partial_t^+\gamma_j (0) = \dot{\gamma}(0).
 $$
 
 Observe now that   the derivatives of a concave function
 $\partial_t^{\pm} \gamma (t)$ are monotone decreasing, hence continuous outside a countable subset of $I$. 
 Since $\partial_t^{+} \gamma (t) = \partial_t^{-} \gamma (t)$ almost everywhere by Lemma \ref{lem: left and right derivative}, 
 it follows that they are equal outside a countable set in $I$.  
 \end{proof}
 
 We now prove a convergence result that will play a key role in the sequel.
 We fix $\mu$ a  positive Borel measure on $\Omega$ and let $\ell$ denote the Lebesgue measure on $\R$. 
 
\begin{pro} \label{pro: convergence semiconcave} 
Let $(f_j)$ be  a sequence of positive functions 
converging to  $f$ in $L^1 (\Omega_T,\ell \otimes \mu)$. 
 Let $(\f^j)$ be  a sequence of functions in ${\mathcal P} (\Omega_T)$ which
 \begin{itemize}
 \item converges $\ell \otimes \mu$-almost everywhere in $\Omega_T$ 
 to a function $\f \in  {\mathcal P} (\Omega_T)$;
 \item   is locally uniformly semi-concave in $]0,T[$. 
 \end{itemize}
Then $\lim_{j \to + \infty} \dot{\f}^j (t,x) = \dot{\f} (t,x)$ for $\ell \otimes \mu$-almost any $(t,x)  \in \Omega_T$, and 
 $$
  \theta (\dot{\f}^j) \, f_j \, \ell \otimes \mu \to \theta (\dot{\f}) \, f \, \ell \otimes \mu, 
 $$
 in the weak sense of Radon measures in $\Omega_T$,
 for all $\theta \in {\mathcal C}^0(\R,\R)$.
\end{pro}

\begin{proof}
Fix a compact subinterval $J\Subset ]0,T[$. By definition there exists a constant $\kappa > 0$ such that all the functions 
$t \longmapsto u^j (t,x) := \f^j (t ,x) - \kappa t^2$ are concave in $J$. By our hypothesis there exists a $\mu$-negligible subset 
$E_1 \subset \Omega_T$ such that for  any $(t,x) \notin E_1$, the sequence $u^j (t,x)$ converges to $u (t,x) := \f (t,x) - \kappa t^2 $. 
It follows from Lemma \ref{lem: left and right derivative} and Lemma \ref{lem: convergence derivative of concave functions} that there exists a 
$\ell \otimes \mu$-negligeable subset $E_2 \subset \Omega_T$ containing $E_1$ such that
$\dot{\f}^j (t,x)$ and $\dot{\f}(t,x)$ are well defined for all $j$ and all $(t,x) \notin E_2$, with
$$
\lim_{j\to +\infty}\dot{\f}^j (t,x)= \dot{\f} (t,x).
$$
Since $f_j \to f$ in $L^ 1 (\Omega_T,\ell \otimes \mu)$ we can find $g \in L^ 1 (\Omega_T,\ell \otimes \mu)$ such that 
$0  \leq f_j \leq g$ in $\Omega_T$. The  measures $(f_j \ell \otimes \mu)$ are thus
uniformly absolutely  continuous with respect to the positive Borel measure $g \ \ell \times \mu$.
The conclusion of the theorem follows therefore from Lebesgue's theorem. 
\end{proof}

\subsection{Elliptic tools}
 \begin{lem}\label{lem: mixed MA}
 Let $u,v$ be bounded psh functions in $\Omega$ such that 
 $$
 (dd^c u)^n \geq e^{f_1} \mu \  \text{and} \ \ (dd^c v)^n \geq e^{f_2} \mu,
 $$	
 where $f_1,f_2$ are bounded Borel functions in $\Omega$ and $\mu$ is a positive Radon measure with $L^1$ density with respect to Lebesgue measure. Then
 $$
 (dd^c (\lambda u+(1-\lambda)v))^n \geq e^{\lambda f_1+ (1-\lambda) f_2} \mu, \ \text{for all}\ ,\lambda\in [0,1]. 
 $$
 \end{lem}
 
 \begin{proof}
 	Observe first that
 	\[
 	 (dd^c (\lambda u+(1-\lambda)v))^n = \sum_{k=0}^n a_k (dd^c u)^k \wedge (dd^c v)^{n-k},
 	\]
 	where $a_k\in (0,1)$, for all $k$ and $\sum_{k=0}^n a_k=1$. 
 	It follows from the mixed Monge-Amp\`ere inequalities \cite{Kol03} (see also \cite{Dinew_2009_Mixed_Inequality}) that 
 	for all $ k=0,\cdots n$,
 	\[
 	(dd^c u)^k \wedge (dd^c v)^{n-k} \geq e^{(kf_1 + (n-k)f_2)\slash n} \mu.
 	\]
 	Summing up the above inequalities and using the convexity of the exponential yields the desired inequality.
 \end{proof}

  	
  	\begin{lemma}
  		\label{lem: continuous approximation}
  		Let $u$ be a psh function in $\Omega$ such that $\lim_{z\to \zeta} u(z)=\phi(\zeta)$ where $\phi$ is a continuous function on $\partial \Omega$. There exists  a decreasing sequence $(u_j)$ of  
  		 plurisubharmonic functions which are continuous  on $\bar{\Omega}$ and such that 
  		 $u_j=\phi$ on $\partial \Omega$ and $u_j\searrow u$ in $\Omega$. 
  	\end{lemma}
  
  This result is classical but we include a proof for the reader's convenience.
  
  	\begin{proof}
  		It follows from the strictly pseudoconvex assumption on $\Omega$ that there exists a harmonic function $\Phi$ in $\Omega$ with boundary value $\phi$. 
  		We first take a sequence of continuous functions $\{f_j\} \subset \mathcal{C}(\bar{\Omega})$ which decreases pointwise to $u$ in $\bar{\Omega}$. By considering $\min(f_j,\Phi)$ we can assume that  $f_j=\phi$ on $\partial \Omega$.  For each $j$, consider  the psh envelope
	$$
	u_j := P(f_j):= \sup\{v\in \PSH(\Omega) \setdef  v^* \leq f_j \text{ in} \ \bar{\Omega}\}. 
	$$
	Then $u \leq u_j \leq f_j$ and $u_j\downarrow u$. Hence  $(u_j)_*=(u_j)^* = \phi$ on $\partial \Omega$. It thus follows from \cite[Lemma 1]{Walsh_1969Envelopes} (see also \cite[Proposition 3.2]{Blo05}) that $u_j$ is continuous in $\bar{\Omega}$.
  	\end{proof}

 \section{Boundary behavior of parabolic  envelopes}
 \label{sect: Perron envelope boundary value}

Our aim is to solve the Cauchy-Dirichlet problem for \eqref{eq: CMAF} with compatible boundary data $h$ using the Perron method of upper envelopes.  In this section we prove that, under natural assumptions, the parabolic Perron envelope has the right boundary values.  We assume  $T<+\infty$.

\subsection{Parabolic pluripotential subsolutions}
Recall that for   $u \in  \mathcal P (\Omega_T)$ the time derivative $\partial_t u$ 
exists a.e. in $\Omega_T$ and satisfies the local uniform bound  $\vert \partial_t  u \vert \leq \kappa_J (u)$  in $J \times \Omega$, for each  $J\Subset ]0,T[$ (see Lemma \ref{lem:Radmacher}). 

 \begin{defi} \label{def: subsolution}
 Fix $u \in \mathcal P (\Omega_T) \cap L^{\infty}(\Omega_T)$. The function $u$ is called a pluripotential subsolution to  \eqref{eq: CMAF} if it satisfies the inequality
 \begin{equation*}
 d t \wedge (dd^c u)^n   \geq e^{\dot{u} + F (t,x,u) } g  d t \wedge d V  
 \end{equation*}
 in  the sense of measures in $\Omega_T$. It is called a pluripotential supersolution to \eqref{eq: CMAF} if the reverse inequality holds in the sense of measures in $\Omega_T$. 
 \end{defi} 
 
  If moreover $u^* \leq h$ in $\partial_0 \Omega_T$, we say that $u$ is a 
  {\it pluripotential subsolution to the Cauchy-Dirichlet problem}
   for the parabolic complex Monge-Amp\`ere equation \eqref{eq: CMAF} with boundary data $h$. Here 
   $$
   u^*(\tau,\zeta): = \limsup_{\Omega_T\ni (t,z) \to (\tau,\zeta)} u(t,z), \, \, \, \, (\tau,\zeta) \in  \partial_0 \Omega_T.
   $$

\begin{pro} \label{pro: subsolution slice}
Fix $u \in \mathcal P (\Omega_T) \cap L^{\infty}_{\loc} (\Omega_T)$.

1) $u$ is a pluripotential subsolution to \eqref{eq: CMAF} if and only if for a.e. $t$,
\begin{equation} \label{eq:subsolslide}
(dd^c u_t)^n \geq e^{\partial_t u (t,\cdot)  + F (t,\cdot,u_t) } g d V,
\end{equation}
in the  sense of measures in $\Omega$. 

2) If $u$ is moreover locally semi-concave in $t$, it is  a
pluripotential subsolution to \eqref{eq: CMAF} if and only if {\it for all} $t$,
$$
(dd^c u_t)^n \geq e^{\partial_t^+ u (t,\cdot) + F (t,\cdot ,u_t) } g  d V, 
$$
in the  sense of measures in $\Omega$.
\end{pro}

 \begin{proof}
 Recall that $\partial_t u$  makes sense almost everywhere and, in case $u$ is semi-concave, coincides with $\partial_t^+u$ which is well defined at every point.

 Assume first that \eqref{eq:subsolslide} holds for a.e. $t$. Let $\chi \in {\mathcal C}^{\infty}_0(\Omega_T)$ be a nonnegative test function. Multiplying \eqref{eq:subsolslide} by $\chi$ and integrating in $x$ we obtain
 $$
\int_{\Omega} \chi(t,x)  (dd^c u_t)^n \geq \int_{\Omega} \chi(t,x) e^{\partial_t u + F (t,x,u_t) } g(x) d V (x).
 $$
 Integrating with respect to $t$, we infer
 $$
\int_{\Omega_T} \chi(t,x)  (dd^c u_t)^n \wedge dt \geq 
\int_{\Omega_T} \chi(t,x) e^{\partial_t u + F (t,x,u_t) } g(x) d V (x) \wedge dt,
 $$
 i.e. $u$ is a pluripotential subsolution to \eqref{eq: CMAF}.
 
 Assume now that $u$ is a pluripotential subsolution to \eqref{eq: CMAF}. We consider  product
 of nonnegative test functions
 $$
 \chi(t,x)=\alpha(t) \theta_j(x),
 $$
 where $(\theta_j)$ is a  sequence of test functions in $\Omega$ which generates a dense subspace of the space of test functions (for the ${\mathcal C}^0$-topology). It follows from Fubini theorem that
 $$
\int_0^T \left\{ \int_{\Omega} \theta_j(x)  (dd^c u_t)^n \right\} \alpha(t) dt
\geq \left\{  \int_{\Omega} \theta_j(x) e^{\partial_t u + F (t,x,u_t) } g(x) d V (x) \right\} \alpha(t) dt.
 $$
 
 We  infer that for all $t \in B_j \subset [0,T[$,
 
 $$
  \int_{\Omega} \theta_j(x)  (dd^c u_t)^n  
\geq    \int_{\Omega} \theta_j(x) e^{\partial_t u + F (t,x,u_t) } g(x) d V (x),
 $$
 where $B_j$ has full measure in $[0,T[$. The set $B=\cap_j B_j \subset [0,T[$ has full measure
 and the previous inequality holds for all $t \in B$ and for all $j \in \N$. Approximating
 an arbitrary nonnegative test function $\theta \in \mathcal{C}^{0}(\Omega)$  by convex combinations of the $\theta_j$'s, we infer
 that for almost every $t$,
 $$
 (dd^c u_t)^n \geq e^{\partial_t u (t,\cdot) + F (t,\cdot,u_t) } g d V.
 $$

 When $u$ is moreover locally semi-concave in $t$ the function $\partial_t^+ u$ is lower semi-continuous
 (see Lemma \ref{lem: left and right derivative}), hence
 $$
 t \mapsto \int_{\Omega} \chi(x) e^{\partial_t^+u (t,x) + F (t,x,u_t(x)) } g(x) d V (x)
 $$
 is lower semi-continuous by Fatou's lemma. Since
 $
 t \mapsto  \int_{\Omega} \chi  (dd^c u_t)^n
 $
 is continuous (by Lemma \ref{lem: continuity of MA integral against test function}), we infer that
 \eqref{eq:subsolslide} holds for almost every $t$ if and only if it holds for every $t$.
 \end{proof}

 \begin{remark}\label{rem: supersolution slice}
 A similar result holds in this case, using
the partial derivative  $\partial_t^- u$ which is upper semi-continuous when $u$ is locally semi-concave
(by Lemma \ref{lem: left and right derivative} again). As a consequence, if $u\in \mathcal{P}(\Omega_T)\cap L^{\infty}_{\loc}(\Omega_T)$ solves \eqref{eq: CMAF} and $u$ is locally uniformly semi-concave in $t\in ]0,T[$ then for almost all $t\in ]0,T[$,
$$
(dd^c u_t)^n =e^{\partial_t u_t +F(t,\cdot,u_t)}gdV.
$$
 \end{remark}

\begin{lem} \label{lem: time derivarive of max} 
For any  $u, v \in \mathcal P (\Omega_T) \cap L^ {\infty}_{\loc} (\Omega_T)$, we have
$$
 {\bf 1}_{\{u \geq v\}} \partial_t \max (u,v) = {\bf 1}_{\{u \geq v\}} \partial_t u \ \text{and}\  {\bf 1}_{\{u > v\}} \partial_t \max (u,v) = {\bf 1}_{\{u > v\}} \partial_t u
$$  
almost everywhere in $\Omega_T$ and
$$
 (dd^c \max (u,v) )^n \wedge d t \geq {\bf 1}_{\{u > v\}}  (dd^c u)^n \wedge dt  + {\bf 1}_{\{u \leq v\}} (dd^c v)^n \wedge d t.
 $$ 
In particular the maximum of two  subsolutions is again a   subsolution.
\end{lem}

\begin{proof} 
It follows from Lemma \ref{lem: Parabolic potentials are in W11loc} that $\mathcal P (\Omega_T)  \subset W^{1,1}_{\loc} (\Omega_T)$.   
The first identity is then a classical result in the theory of Sobolev spaces  (see e.g. \cite[Lemma 7.6 page 152]{GT01}). The second inequality is a consequence of the elliptic maximum principle 
for psh functions (see  e.g. \cite[Corollary 3.28]{GZbook}). 
\end{proof}

 It is therefore natural to consider the Perron envelope of  subsolutions :
 
 \begin{defi} \label{def: envelope of subsolutions}  
  We let  $\mathcal S_{h,g,F}(\Omega_T)$ denote the set of  $u \in \mathcal P (\Omega_T)$ such that 
  \begin{enumerate}
  	\item $u$ is a pluripotential subsolution to \eqref{eq: CMAF}  in $\Omega_T$;
  	\item $u^*\leq h$ on $\partial_0 \Omega_T$, i.e.  for all $(s,\zeta) \in \partial_0 \Omega_T$,
   $$
  \limsup_{\Omega_T \ni (t,z) \to (s,\zeta)} u(t,z) \leq h(s,\zeta).
  $$
  \end{enumerate}
  We let
 $$
   U = U_{h,g,F,\Omega_T} :=  \sup \{ u \setdef  u \in \mathcal S_{h,g,F}(\Omega_T) \}
 $$
denote the upper envelope of all  subsolutions.
\end{defi}

\begin{lem}  
\label{lem: uniform bound}
The set $\mathcal S_{h,g,F}(\Omega_T)$ is not empty, uniformly bounded in $\Omega_T$, stable under finite  maxima. 
The envelope  $U := U_{h,g,F,\Omega_T}$ and its upper semi-continuous regularization $U^*$
 satisfy  for all $(t,z) \in \Omega_T$,
 $$
 B \rho (z) - M_h \leq U (t,z) \leq U^*(t,z) \leq M_h, 
 $$
 where $B = e^{M_F \slash n}$.
 In particular  
 \begin{equation} \label{eq: uniform bound}
 \Vert U \Vert_{L^{\infty} (\Omega_T)} \leq M_U:= M_h  + c_n e^{M_F} \Vert g \Vert^{1 \slash n}_{L^p (\Omega)}.
 \end{equation}
 \end{lem}
 Recall that 
 $$
 M_h:= \sup_{\partial_0 \Omega_T} |h| \ , \  M_F:= \sup_{\Omega_T} F(\cdot,\cdot, M_h).
 $$
 \begin{proof}
Fix $B = e^{M_F/n}$. Since $g d V = (dd^c \rho)^n$ we obtain
 $$
 e^{ M_F} g dV \leq B^n  (dd^c \rho)^n.
 $$
 
Set, for $(t,z) \in \Omega_T$,
 $$
 \underline{u} (t,z) := B \rho (z) - M_h.
 $$
Then $\underline{u} \in \mathcal{S}_{h,g,F}(\Omega_T)$, hence $\underline{u} \leq U_{h,g,F,\Omega_T}$. 
 
 Fix $u \in \mathcal S_{h,g,F}(\Omega_T)$ and fix $t\in ]0,T[$. Then $\limsup_{z\to \zeta} u(t,z) \leq h(t,\zeta)$, for every $\zeta\in \partial \Omega$. It thus follows from the classical maximum principle for plurisubharmonic functions that $u(t,z) \leq M_h$ for every $z\in \Omega$. Thus $ U (t,\cdot) \leq M_h$ for any $t \in ]0,T[$.
 
  Therefore, the upper envelope $U$ is well defined and satisfies the uniform estimates
  $
   \underline{u}   \leq U \leq M_h, \, \, \, \text{in} \, \, \, \,  \Omega_T,
  $
hence
 $$
  U (t,z) := \sup \{ u (t,z)  \setdef    u \in \mathcal S_{h,g,F}(\Omega_T),\; \underline{u} \leq u \leq M_h\}.
 $$
The stability under finite maxima follows from Lemma \ref{lem: time derivarive of max}. 
\end{proof}

\subsection{Construction of sub-barriers} 
The family $t \longmapsto h (t,z)$ ($ z \in \partial \Omega)$ is  uniformly Lipschitz in $]0,T[$ if there exists a constant $\kappa (h) >0$ such that 
  \begin{equation}
  	\label{eq: h is uniformly Lipschitz t}
  	|h(t,z)-h(s,z)| \leq \kappa (h) \,  |t-s| ,\ \forall (t,s) \in [0,T[^2, \forall z\in \partial\Omega. 
  \end{equation}

 The parabolic boundary of $\Omega_T$ consists in two different types of points.
 We provide barriers for each type.

 \subsubsection{Sub-barriers at boundary points of Dirichlet type}
 
 We first construct subbarriers at Dirichlet boundary points in $[0,T[ \times \partial \Omega$.
 
\begin{lem} \label{lem: sub-barriers Dirichlet} Assume $h$ satisfies \eqref{eq: h is uniformly Lipschitz t}. Then there exists  $u \in \mathcal S_{h,g,F}(\Omega_T)$ such that $u (\cdot,z)$ ($z \in \Omega$) is  uniformly Lipschitz in $[0,T[$ and satisfies :  for any $(s,\zeta) \in [0,T[ \times \partial \Omega$,
 $$
 \lim_{(t,z) \to (s,\zeta)} u (t,z) = h (s,\zeta).  
 $$ 
If $h_0$ is continuous  on $\bar{\Omega}$  then $u$ can be chosen to be continuous in $[0,T[ \times \bar{\Omega}$.
\end{lem}

  \begin{proof}
 Fix $t\in [0,T[$ and set $h_t:= h(t,\cdot)\in \mathcal{C}(\partial\Omega)$. 
 Let $\phi_t$ be the unique continuous plurisubharmonic function in $\Omega$ such that 
  \begin{equation}\label{eq: maximal psh}
  	\begin{cases} 
  	  	(dd^c \phi_t)^n=0\  \textrm{in}\ \Omega, \\
  	\lim_{z\to \zeta} \phi_t(z)=h_t(\zeta)-h_0(\zeta),\ \forall \zeta\in \partial \Omega.
  \end{cases}
  \end{equation}
  The existence  and continuity of $\phi_t$ on $\bar{\Omega}$ follows from classical results in pluripotential theory (see \cite{Bedford_Taylor_1976Dirichlet, Bedford_Taylor_1982Capacity}, \cite[Theorem 5.12]{GZbook}). 
Moreover, $\phi_t$ can be characterized as the supremum of all subsolutions to \eqref{eq: maximal psh}. Since $t \longmapsto h (t,z)$ ($ z \in \partial \Omega$) is uniformly Lipschitz in $[0,T[$, the tautological maximum principle reveals that the family of functions  $t \longmapsto  \phi(t,z):=\phi_t(z)$ $(z \in \Omega$) is  uniformly Lipschitz in $[0,T[$ with a Lipschitz constant $\kappa({\phi})\leq \kappa (h)$.  By Lemma \ref{lem: partial continuity}, $(t,z) \longmapsto \phi_t (z)$ is continuous in $[0,T[ \times \bar \Omega$.
  Consider now, for $(t,z) \in \Omega_T$,
  $$
   u(t,z) := \phi_t(z) + h_0(z) + A \rho (z),
  $$
  where $A > 0 $ is a large constant to be chosen later, and $\rho$ is defined in \eqref{eq: rho}. Observe that  $u \in \mathcal P (\Omega_T)$ and $u^* \leq h$ in $\partial_0 {\Omega}_T$.  It is clear that $t \longmapsto u (t,z)$ ($z \in  \Omega$) is  uniformly Lipschitz in $ [0,T[$ with $\kappa (u) \leq \kappa (h)$.  Moreover 
  $$
  dt \wedge (dd^c u)^n  \geq A^n dt \wedge  (dd^c \rho)^n \geq  A^n dt \wedge g d V
  $$
 in the weak sense of measures in $\Omega_T$. We choose $A>0$ so that $n \log A  \geq \kappa (h) + M_F $.  It is then clear that $u \in \mathcal S_{h,g,F} (\Omega_T)$.
  By definition,  $u$ is continuous in $[0,T[ \times \Omega $ provided that $h_0$ is continuous on $\bar{\Omega}$.
 \end{proof}

 \subsubsection{Sub-barriers at boundary points of Cauchy type}
  
   We  now construct sub-barriers at  boundary points in $\{0\} \times {\Omega}$.

  \begin{lem} \label{lem: sub-barriers Cauchy} Assume $h$ satisfies \eqref{eq: h is uniformly Lipschitz t}. 
   Then there exists $v \in \mathcal S_{h,g,F}(\Omega_T)$  such that  for all   $\zeta \in \bar{\Omega}$, 
   $$
   \limsup_{\Omega_T \ni (t,z) \to (0,\zeta)}  v (t,z) = h_0 (\zeta), \ \text{and}\ \lim_{t\to 0^+} v(t,\zeta) =h_0(\zeta). 
   $$
   If $h_0$ is continuous on $\bar{\Omega}$ then $v$ can be chosen to be continuous on $[0,T[\times \bar{\Omega}$.
  \end{lem}

 \begin{proof}
By assumption on $h$ we have, for all $(t,z) \in [0,T[ \times \partial \Omega$,
  $$
  h (0,z) \leq h (t,z) + \kappa t.
  $$
Set, for $ (t,z)\in \Omega_T$,
 $$
  v (t,z) := h_0(z) + t(\rho (z) - C) + n [(t\slash T) \log (t\slash T) - t\slash T)], 
 $$
 where $C:=  \kappa_h+ M_F-\min(n\log T,0)$. Then $v\in  \mathcal{S}_{h,g,F}(\Omega_T)$ and $v$ is continuous on $[0,T[\times \bar{\Omega}$ if $h_0$ is continuous on $\bar\Omega$. 
\end{proof}

 \subsection{Super-barriers} \label{sect: superbarrier}
 
 \subsubsection{Super-barriers at boundary points of Dirichlet type}  
 \label{subsect: Dirichlet boundary superbarrier}

For each $t\in [0,T[$, we let $H_t$ be the unique harmonic function in $\Omega$ with boundary value $h_t$ on $\partial \Omega$
and set $H(t,z):=H_t(z)$ (the existence of $H_t$ is a classical fact;  see e.g. \cite[Theorem 2.14]{GT01}). Recall that $H_t$ can be defined as the upper envelope of all subharmonic functions  in $\Omega$ with boundary values $\leq h_t$.  
Observe that $h_0 \leq H (0,\cdot)$ in $\Omega$, with equality at the boundary.


\begin{lem} \label{lem: super-barrier Dirichlet} 
For all $(t,z) \in [0,T[\times {\Omega}$ we have 
$U^*(t,z) \leq H(t,z)$. In particular, for all $(s,\zeta)\in  [0,T[ \times \partial \Omega$, 
$$
\limsup_{(t,z)\to (s,\zeta)} U^*(t,z) \leq h(s,\zeta). 
$$
\end{lem} 

\begin{proof}
It follows directly  from the maximum principle for subharmonic functions that $U_t\leq H_t$, for all $t\in [0,T[$. Fix $S\in ]0,T[$. Since the family $\{h (\cdot,z) \setdef  z  \in  \partial \Omega\}$ is  equicontinuous in $[0,S]$, it follows by definition   that  the family $\{H (\cdot,z) \setdef z  \in   \bar \Omega\}$ is equicontinuous in $[0,S]$, hence the function $H$ is continuous in $[0,T[ \times \bar{\Omega}$, by Lemma \ref{lem: partial continuity}.   Then $ U^*(t,z) \leq H (t,z)$ for any $(t,z) \in \Omega_T$. From the continuity of $H$, it follows that $U^*\leq H$ in $[0,T[\times \bar{\Omega}$.
\end{proof}

\subsubsection{Boundary behaviour at Cauchy boundary points}  

So far we have constructed enough barriers to ensure that the envelope of subsolutions either matches the boundary data (at Dirichlet points), or stays below it. The following average argument will allow us to conclude that it also coincides with the boundary data at Cauchy points :

\begin{lem} \label{lem: monotonicity of mean values}
Let $\varphi \in \mathcal{P}(\Omega_T)\cap L^{\infty}(\Omega_T)$ be a subsolution  to  \eqref{eq: CMAF}  such that $\int_{D} (dd^c \varphi_t)^n \leq C$, for every $t\in [0,T[$, 
 for some  $C>0$, where $D$ is an open set in $\Omega$. 
Then, for each positive continuous test function $\chi$ in $D$,  there exists $A>0$ such that
$$
t \mapsto \int_{D} \chi \f_t gdV-At
$$
is decreasing in $]0,T[$.
\end{lem}

\begin{proof}
Since $\f$ is a subsolution to \eqref{eq: CMAF} we  obtain for a.e. $t\in ]0,T[$, 
 \[
  	\int_{D} \chi e^{\dot{\varphi}_t+ m_F} gdV \leq \int_{D} \chi e^{\dot{\varphi}_t+ F} gdV \leq \int_{D} \chi (dd^c \varphi_t)^n \leq C,
\]
  where $m_F:=\inf_{[0,T[\times \bar{\Omega} \times [-M_U,M_U]} F$. 
  It follows from Jensen's inequality that 
  \[
  \int_{D} \chi\dot{\varphi}_t gdV \leq C_2, 
  \]
 for a.e. $ t\in ]0,T[$, where $C_2>0$ is a uniform constant. We then infer that the function $t\mapsto \int_{D} \chi \varphi_tgdV-C_2t$ is decreasing in $]0,T[$.
\end{proof}

\begin{cor}
\label{cor: stability of boundary inequality}
Assume  $\{u^j\}\subset \mathcal{S}_{h,g,F} (\Omega_T)$ is a bounded sequence which is locally uniformly Lipschitz in $]0,T[$ (with Lipschitz constant independing of $j$). If $\{u^j\}$ converges in $L^{1}_{\loc}(\Omega_T)$ to $u\in \mathcal{P}(\Omega_T)$ then 
$$
\limsup_{(t,z)\to (s,\zeta)} u(t,z) \leq h(s,\zeta), \ \forall (s,\zeta) \in \partial_0 \Omega_T.
$$
\end{cor}
\begin{proof}
For $(s,\zeta)\in ]0,T[\times \partial \Omega$ the desired inequality holds thanks to Lemma \ref{lem: super-barrier Dirichlet}. Fix $D\Subset \Omega$ and let $\chi$ be a positive continuous test function in $\Omega$. It follows from the Chern-Levine-Nirenberg inequality \cite[Theorem 3.9]{GZbook} that $\int_D (dd^c u^j_t)^n$ is uniformly bounded.  Lemma \ref{lem: monotonicity of mean values} therefore provides us with a uniform constant $A>0$ such that 
$$
\int_D \chi u^j_t gdV \leq \int_D \chi h_0 gdV + At, \ \forall t\in ]0,T[, \ \forall j.
$$
Letting $j\to +\infty$, Lemma \ref{lem:L1Slice-L1} ensures that
$$
\int_D \chi u_t gdV \leq \int_D \chi h_0 gdV + At, \ \forall t\in ]0,T[.  
$$ 
If $v$ is a cluster point of $u_t$ as $t\to 0$ then the above estimate yields $v\leq h_0$ on $D$. Since $D$ was chosen arbitrarily, $v\leq h_0$ on $\Omega$.  The conclusion thus follows from Lemma \ref{lem: usc}.
\end{proof}

\subsection{Boundary behavior of the Perron envelope}

 \begin{thm} \label{thm: boundary value of U} 
 Assume $h$ satisfies \eqref{eq: h is uniformly Lipschitz t}. Then 
  the upper semi-continuous regularization 
  of the envelope $U = U_{h,g,F,\Omega_T}$ 
  satisfies

\smallskip

$(i)$ for any $(s,\zeta) \in [0,T[ \times \partial \Omega$,
$
 \lim_{\Omega_{T} \ni (t,z) \to (s,\zeta)} U^* (t,z) = h (s,\zeta).
 $
 
 \smallskip
 
 $(ii)$ for any $z_0 \in \Omega$,
 $$
 \lim_{ t \to 0^+} U^*(t,z_0) = h (0,z_0),  \, \, \, \,  \mathrm{and} \, \, \, \,
\limsup_{\Omega_{T} \ni (t,z) \to (0, z_0)} U^* (t,z) = h (0,z_0).
 $$
 \end{thm}
 
 Here $U^*$ denotes the u.s.c. regularization of $U$ in the  variable $(t,z)$ in $\Omega_T$.
 
 \begin{proof} 
Fix $(s,\zeta)\in [0,T[\times \partial \Omega$.
Lemma \ref{lem: sub-barriers Dirichlet} and Lemma \ref{lem: super-barrier Dirichlet} yield $(i)$.  
 
In view of Lemma \ref{lem: sub-barriers Cauchy}  it remains to prove that  for all $z_0 \in \Omega$,
$$
\limsup_{\Omega_{T} \ni (t,z)\to (0,z_0)} U^* (t,z) \leq h_0 (z_0).
$$
The envelope $U$ is locally uniformly Lipschitz in $]0,T[$,
as  follows from Theorem \ref{thm: U is local Lip in t}. 
We can thus apply Lemma \ref{lem: negligible} to conclude that $U^* (t,\cdot) = U_t^*$ in $\Omega$ for any $t \in ]0,T[$, where $U_t^* = (U_t)^*$ is the u.s.c.  regularization of the function $U_t$ ($t $ fixed) in $\Omega$. Using Lemma \ref{lem: usc} it is then enough to show that 
$$
\limsup_{t\to 0} U^*_t (z_0) \leq h_0 (z_0), \forall z_0 \in \Omega.
$$
Observe that $U$ can be seen as the upper envelope of all  $\varphi\in \mathcal{S}_{h,g,F} (\Omega_T)$ such that  $\sup_{\Omega_T} |\varphi | \leq M_U$, where $M_U$ is given in Lemma \ref{lem: uniform bound}. 

Fix $\chi$ a continuous positive test function in $\Omega$. We claim that there exists a constant $C>0$ such that for all $t\in ]0,T[$,
\begin{equation}
\label{eq: Cauchy boundary 1}
\int_{\Omega_1} \chi U_t^*  gdV \leq  \int_{\Omega_1} \chi h_0 gdV + C t.
\end{equation}

Indeed, fix $t_0\in ]0,T[$. Since the set of subsolutions is stable under maximum, by Choquet's lemma, $U_{t_0}^{*} = (\lim_{j\to +\infty} \varphi^j_{t_0})^{*}$ in $\Omega$, where $\{\varphi^j\}$ is an increasing sequence  in $\mathcal S_{h,g,F} (\Omega_T)$ with $|\varphi^j|\leq M_U$.  
The sequence $\{\varphi^j\}$ depends on $t_0$ but, as will be shown later, the constant $C$ does not depend on $t_0$.  
Now fix $j \in \N$,  $\Omega_1\Subset \Omega_2\Subset \Omega$ compact subsets of $\Omega$. 
  It follows from the Chern-Levine-Nirenberg inequality \cite[Theorem 3.9]{GZbook} that 
\begin{equation*}
	\int_{\Omega_1} \chi (dd^c \varphi_t^j)^n \leq C_1, \ \text{for all } \ t\in ]0,T[,
\end{equation*}
where $C_1$ depends only on $\Omega_1, \Omega_2, \chi$ and $M_U$. 
It thus follows from  Lemma \ref{lem: monotonicity of mean values} that
  $$
  \int_{\Omega_1} \chi \varphi^j_t g dV \leq  \int_{\Omega_1} \chi h_0 gdV + C_2 t, \ \text{for all} \ t\in ]0,T[,
  $$
  for a uniform constant $C_2>0$. 
A classical theorem of Lelong (see \cite[Proposition1.40]{GZbook}) ensures that  
 $$
 \left \{ z\in \Omega \setdef  \lim_{j \to + \infty} \varphi^j_{t_0} (z) < (U_{t_0})^* (z) \right \}
 $$
 has volume zero in $\Omega$. Therefore taking the limit as $j\to +\infty$ in the previous inequality for $t=t_0$, we deduce that
  \[
  \int_{\Omega_1} \chi U_{t_0}^*  gdV \leq  \int_{\Omega_1} \chi h_0 gdV + C_2 t_0.
  \]
 Since $C_2$ does not depend on $t_0$,  the claim is proved. 
 
  Let $w_0\in \PSH(\Omega)$ be any cluster point of $U_t^*$ as $t\to 0^+$. 
  We can assume that $U_t^*$ converge to $w_0$ in $L^q(\Omega)$ for any $q>1$. Then $U_t^* g$ converge to $w_0g$ in $L^1(\Omega)$. 
  Thus, by \eqref{eq: Cauchy boundary 1}, $\int_{\Omega_1} \chi w_0gdV \leq \int_{\Omega_1} \chi h_0gdV$. Since $\chi\geq 0$ was chosen arbitrarily, we infer that $w_0\leq h_0$ almost everywhere in $\Omega_1$ with respect to $gdV$. 
  The assumption on $g$ finally yields $w_0\leq h_0$ on $\Omega_1$. By letting $\Omega_1\to \Omega$ we can then conclude that $\limsup_{t\to 0}U_t^* \leq h_0$ in $\Omega$.  
   \end{proof}

 \begin{lemma}
 	\label{lem: uniform convergence}
 	If  $h_0$ is continuous on $\bar{\Omega}$ then $U^*(t,\cdot)$ uniformly converges to $h_0$ as $t\to 0^+$. 
 \end{lemma}
 
 Note that in Lemma \ref{lem: uniform convergence} we merely assume that $h$  is  locally uniformly Lipschitz in $t\in ]0,T[$. 
 \begin{proof}
 We first assume that $h$ satisfies \eqref{eq: h is uniformly Lipschitz t}.
 		It follows from Lemma \ref{lem: sub-barriers Cauchy} that there exists a continuous subsolution 
 	$u\in \mathcal{S}_{h,g,F}(\Omega_T)$ : 
 	\[
 	u(t,z) := h_0(z) + t(\rho(z) -C) + \eta(t),
 	\]
 	where $C$ is a uniform constant, $\eta(t) \to 0$ as $t\to 0$ and $\rho$ is defined by \eqref{eq: rho}.
 	
 	For each $t\in [0,T[$, let $H_t$ be the unique continuous harmonic function in $\Omega$ with boundary value $h_t$. 
 	Then 
 	$$
 	u\leq U^*\leq H.
 	$$
 	
 	 	  It follows moreover from Theorem \ref{thm: boundary value of U}  that  $U^*(t,\cdot)$ converges in $L^1(\Omega)$ to $h_0$ as $t\to 0$.  Hartog's lemma thus yields
 	$$
 	\limsup_{t\to 0} \max_{z\in K} (U^*(t,z)-h_0(z)) \leq 0,
 	$$
 	for any compact   $K\Subset \Omega$.  
 	Since $u_t$   uniformly converges  to $h_0$ as $t\to 0^+$ we infer,
 	$\textrm{for any compact}\ K\Subset \Omega$,
 	\begin{equation}
 		\label{eq: U converge uniformly at zero compact}
 		\lim_{t\to 0^+}	\sup_{z\in K} |U^*(t,z)-h_0(z)| =0.
 	\end{equation}
 	
 	Fix $\varepsilon>0$. Since $H_0$ and $h_0$ are continuous on $\bar{\Omega}$ with $h_0=H_0$ on $\partial \Omega$, there exists $\delta>0$ small enough such that 
 	\[
 	\sup_{z\in \Omega_{\delta}} |H_0(z)-h_0(z)| \leq \varepsilon, 
 	\]
 	where $\Omega_{\delta}:= \{z\in \bar{\Omega}  \setdef {\rm dist}(z,\partial \Omega) <\delta\}$. We also have, for $(t,z)\in [0,T[\times \bar{\Omega}$, 
 	\[
 	U^*(t,z) -h_0(z) \leq H_t(z) -h_0(z) \leq H_0(z) -h_0(z) + \kappa_h t. 
 	\]
  Using this and    the uniform convergence of $u_t$  to $h_0$ as $t\to 0$ we obtain  
   \[
 \lim_{t\to 0}  \sup_{z\in  \Omega_{\delta}} |U^*(t,z)-h_0(z)| \leq \varepsilon. 
   \]
   Using  \eqref{eq: U converge uniformly at zero compact} we infer
     \[
	 \lim_{t\to 0}  \sup_{z\in \bar{ \Omega}} |U^*(t,z)-h_0(z)| \leq \varepsilon. 
   	\]
Letting $\varepsilon \to 0^+$ yields the conclusion. 

For the general case (i.e. $h$  is locally uniformly Lipschitz in $]0,T[$ with $h_0$ continuous on $\bar{\Omega}$), we proceed by approximation.  Fix $S\in ]T/2,T[$, $\varepsilon>0$ small enough. 
Proposition \ref{pro: identity principle} ensures that $U_{h,g,F,\Omega_S}=U_{h,g,F,\Omega_T}$ in $\Omega_S$.  Set
$$
\left\{
\begin{array}{ll}
h^{\epsilon }(t,\zeta) :=h(t+\epsilon,\zeta)  &  \text{ if } (t,\zeta)\in [0,S]\times \partial \Omega \\ 
h^{\epsilon } (0,z)=h_0(z)+ \psi^{\epsilon}(z)  & \text{ if } \ z\in \Omega, 
\end{array}
 \right.
$$
where $\psi^{\epsilon}$ is the maximal plurisubharmonic function in $\Omega$ such that  $\psi^{\epsilon}(\zeta)= h(\epsilon, \zeta)-h (0,\zeta)$ in $\partial \Omega$. Recall that $\psi^{\epsilon}$ is the upper envelope of all psh functions $\psi$ in $\Omega$ whose boundary values satisfy $\psi^* \leq  h(\epsilon, \zeta)-h (0,\zeta)$ on $\partial \Omega$. 

Since $h^{\epsilon} (0, \cdot) = h (\e,\cdot)  \to h (0,\cdot)$ uniformly on $\partial \Omega$ as $\epsilon\to 0$,  it follows that $\psi^\epsilon \to 0$ uniformly in $\bar \Omega$ as $\epsilon\to 0$. Therefore $\{h^{\epsilon}\}$   uniformly converges  on $\partial_0 \Omega_{S}$ to $h$ as $\epsilon\to 0$. Set $U^{\varepsilon}:= U_{h^{\varepsilon},g,F,\Omega_S}$. Then $(U^{\varepsilon})^*$ uniformly converges to $U^*$ in $\Omega_S$. Since $h^{\epsilon}$ is uniformly Lipschitz in $t\in [0,S]$, 
 the previous step (using Theorem \ref{thm: boundary value of U})  guarantees that $(U^{\varepsilon})^*(t,\cdot)$ uniformly converges to $h_0$ as $t\to 0$, hence $U^*_t$ uniformly converges to $h_0$ as $t\to 0$. 
 \end{proof}

\section{Time regularity of parabolic envelopes} 
\label{sect: time regularity}
We establish in this section time regularity  of the envelope $U:=U_{h,g,F,\Omega_T}$ by using and adapting some classical ideas of pluripotential theory. 

We work in $\Omega_S$ for each $0<S<T$ and eventually let $S\to T$.  We thus assume $T<+\infty$, the family $\{ F(\cdot, z, \cdot) \setdef z\in \Omega\}$ is uniformly Lipschitz and semi-convex in $[0,T]\times J$ for each $J\Subset \mathbb{R}$, and $h$ satisfies  
\begin{equation}
	\label{eq: local Lipschitz condition h}
t|\partial_t h(t,z)| \leq \kappa_h, \ \text{for all}\ (t,z) \in ]0,T[\times \partial \Omega,
\end{equation}
for some positive constant $\kappa$. The condition \eqref{eq: local Lipschitz condition h} is equivalent to the fact that for all $(t,z) \in\Omega_T$ and $s>0$ with $st<T$, we have
	\begin{equation}  \label{eq: local Lipschitz condition h reformulation}
		|h(t,z)-h(st,z)| \leq \kappa_h \frac{|s-1|}{\min(s,1)}. 
	\end{equation}
If $h$ is uniformly Lipschitz in $t\in [0,T[$ (as in \eqref{eq: h is uniformly Lipschitz t}) then the above condition is automatically satisfied. On the other hand the condition above implies that $h (\cdot,z), z \in \Omega$ is locally uniformly Lipschitz in $ ]0,T[$.

\subsection{Lipschitz control in the time variable}

The following identity principle  plays a crucial role in the sequel. For simplicity we will denote the restriction of $h$ on $\partial_0\Omega_S$, for $0<S<T$, by $h$.

\begin{pro}
	\label{pro: identity principle}
	For all $S\in ]0,T[$ we have $U_{h,g,F,\Omega_T}=U_{h,g,F,\Omega_S}$ in $\Omega_S$. 
\end{pro}

\begin{proof}
Set $V:= U_{h,g,F,\Omega_S}$ and $U:=U_{h,g,F,\Omega_T}$. Fix  $u\in \mathcal{S}_{h,g,F}(\Omega_S)$ and $t_0\in ]0,S[$ such that 
	$$
	(dd^c u(t_0,\cdot))^n \geq e^{\partial_t u(t_0,\cdot) + F(t_0,\cdot,u(t_0,\cdot)} gdV. 
	$$
	Set $M_1 := \sup_{\Omega}|\partial_t u(t_0,\cdot)|<+\infty$.
	If  $ A \geq   M_1$ the function 
	$$
	\Omega_T \ni (t,z) \mapsto v(t,z) :=  \begin{cases}
		u(t,z) , \text{if}\ t \in [0,t_0], \\
		u(t_0,z) -A(t-t_0)\ \text{if}\ t\in [t_0,T[,
	\end{cases}
	$$
	is again a subsolution to \eqref{eq: CMAF} in $\Omega_T$.  
	Applying (\ref{eq: h is uniformly Lipschitz t}) on  the interval $J := [t_0,T[$,
	we obtain that   $v^*\leq h$ on $\partial_0 \Omega_T$ if $A \geq \kappa_J(h)$. 
	
	We therefore choose $A \geq \max \{M_1, \kappa_J (h)\}$.
	Then $v \in  \mathcal{S}_{h,g,F} (\Omega_T)$ hence $v \leq U$ in $\Omega_T$.
	In particular $u\leq U$ on $\Omega_{t_0}$. 
	Taking supremum over all candidates $u$ we obtain $V\leq U$ in $\Omega_{t_0}$. Using Proposition \ref{pro: subsolution slice}  we can let $t_0\to S$ to obtain $V\leq U$ in  $\Omega_{S}$. The reverse inequality is clear. 
\end{proof}

\begin{thm} \label{thm: U is local Lip in t} 
If $h$  satisfies \eqref{eq: local Lipschitz condition h}, then
the  envelope $U := U_{h,g,F,\Omega_T}$ satisfies 
$$
t|\partial_t U(t,z)| \leq \kappa_U, \ \forall (t,z)\in \Omega_T, 
$$
where $\kappa_U>0$ is a uniform constant.
\end{thm}

 We will show that the constant $\kappa_U$ is actually explicit,
 \begin{equation}
 	\label{eq: kappa U}
 	\kappa_U=  (T+1)(3M_U + 2 \kappa_h +2n + \kappa_F (T+M_U)).
 \end{equation}
This quantitative information will be crucial in perturbation arguments, to obtain uniform Lipschitz constants of the approximants. 

\smallskip

The proof of this theorem follows and adapt ideas developed by Bedford and Taylor in their study
of Dirichlet problems for  elliptic complex Monge-Amp\`ere equations (see \cite[Theorem 6.7]{Bedford_Taylor_1976Dirichlet}, \cite{Dem91}).

 \begin{proof} 
 
 By the assumption on $F$, there exists a constant $\kappa_F$ such that, for all $z\in \Omega$ and $ (t_j,r_j) \in [0,T[ \times [-2M_U,2M_U]$, $j=1,2$,
\begin{equation}
	\label{eq: Lip F}
 |F(t_1,z,r_1) -F(t_2,z,r_2)| \leq \kappa_F (|t_1-t_2|+|r_1-r_2|).    
\end{equation}
 Fix $ u \in \mathcal{S}_{h,g,F} (\Omega_T)$ such that $\sup_{\Omega_T}|u|\leq M_U$, where $M_U$ is defined in Lemma \ref{lem: uniform bound}. Fix $0<S<T$ and  $s\geq 1 \slash 2$ close to $1$ enough such that  $sS<T$.  Set, for $(t,z) \in \Omega_S$, 
 $$
 v^s (t,z) := s^{-1}u(st,z) - C |s-1|(t+1),
 $$
 where 
 \begin{equation}\label{eq: constant C Lipschitz U}
 	C:= 2M_U + 2 \kappa_h +2n + \kappa_F (T+M_U). 
 \end{equation}
We are going to prove that $v^s \in \mathcal{S}_{h,g,F}(\Omega_S)$.  Since $u$ is a subsolution to  \eqref{eq: CMAF}, for a.e. $t \in ]0,S[$ we have 
 \begin{flalign*}
 	(dd^c v^s (t,\cdot))^n &  =   s^{-n} (dd^c u (st,\cdot)^n \\
   & \geq  e^{-n\log s + \partial_\tau u (st,\cdot) + F (st, \cdot, u (st,\cdot))} g d V\\
   & \geq  e^{ \partial_t v^s (t,\cdot) +C |s-1| + F(t,\cdot,s^{-1}u(st,\cdot)) - n\log s -\kappa_F(T|s-1| + |s^{-1}-1| M_U)}gdV\\
   &\geq e^{ \partial_t v^s (t,\cdot) + F(t,\cdot,v^s(t,\cdot)) }gdV,
 \end{flalign*}
 where in the last line we use \eqref{eq: constant C Lipschitz U} and the fact that $F$ is increasing in $r$. 

  We now take care of the boundary values.   For $t\in [0,S], z \in \partial \Omega$ we have 
   \begin{flalign*}
   	v^{s}(t,z) & \leq -C |s-1| + |s^{-1}-1|M_U + h(st,z)\\
   	& \leq -C|s-1| + 2|s-1|M_U + h(t,z) + 2\kappa_h|s-1|\\
   	& \leq h(t,z),
   \end{flalign*}
   where in the second line we use \eqref{eq: local Lipschitz condition h reformulation}, and in the last line we use again  \eqref{eq: constant C Lipschitz U}. 
   For $z\in \Omega$ we similarly get $(v^s)^*(0,z)\leq h_0(z)$. 

   The computations above  show that $v^s\in \mathcal{S}_{h,g,F}(\Omega_S)$. Proposition \ref{pro: identity principle} thus yields  $v^s\leq U$ in $\Omega_{S}$.   
   Taking supremum over $u$ we arrive at 
   $$
   s^{-1}U(st,z) -C|s-1|(t+1) \leq U(t,z), \  \text{for all}\  (t,z)\in \Omega_{S}. 
   $$
   Letting $s\to 1$ we infer, for all $(t,z) \in \Omega_{S}$,
   $$
   |\partial_t U(t,z) | \leq M_U+ C(T+1).
   $$
   Letting $S\to T$ yields the conclusion.   
\end{proof}

\begin{defi}
  Given a constant $\kappa>0$  we let $\mathcal{S}^{\kappa}:=\mathcal{S}^{\kappa}_{h,g,F}(\Omega_T)$ denote the set of all $u\in \mathcal{S}_{h,g,F}(\Omega_T)$ such that, for all $t\in ]0,T[$,
\begin{equation}
	\label{eq: Lip constraint}
\sup_{\Omega} |\partial_t u(t,z)|\leq \kappa/\min(t,b),
\end{equation}
where $b =\min(1,T/2)$,
and we set 
	$$
	U^{\kappa} :=U^{\kappa}_{h,g,F,\Omega_T} := \sup \{u \setdef u \in \mathcal{S}^{\kappa}_{h,g,F}(\Omega_T)\}.
	$$
\end{defi}

We will need the following identity principle :

\begin{pro}
	\label{pro: identity principle 2}
	For all $S\in ]T/2,T[$ and $\kappa\geq 2T \kappa_h$ we have 
	$$
	U^{\kappa}_{h,g,F,\Omega_T}=U^{\kappa}_{h,g,F,\Omega_S} \ \text{ in} \  \Omega_S.
	$$ 
\end{pro}

\begin{proof}
The proof is similar to that of Proposition \ref{pro: identity principle}.
Fix $S\in ]T/2,T[$ and set $V:= U^{\kappa}_{h,g,F,\Omega_S}$, $W:=U^{\kappa}_{h,g,F,\Omega_T}$.   
Fix $u\in \mathcal{S}^{\kappa}_{h,g,F}(\Omega_S)$.
Using Proposition \ref{pro: subsolution slice} we fix $t_0\in ]T/2,S[$ such that 
	$$
	(dd^c u(t_0,\cdot))^n \geq e^{\partial_t u(t_0,\cdot) + F(t_0,\cdot,u(t_0,\cdot)} gdV. 
	$$
	Since $\sup_{\Omega}|\partial_t u(t_0,\cdot)|\leq \kappa/b$, the function 
	$$
	\Omega_T \ni (t,z) \mapsto v(t,z) :=  \begin{cases}
		u(t,z) , \text{if}\ t \in [0,t_0], \\
		u(t_0,z) -\kappa b^{-1} (t-t_0)\ \text{if}\ t\in [t_0,T[,
	\end{cases}
	$$
	is still a subsolution to \eqref{eq: CMAF} in $\Omega_T$. It folllows from \eqref{eq: local Lipschitz condition h reformulation} that
	$$
	|h(t,z) -h(t_0,z)|\leq \frac{2\kappa_h}{T} |t-t_0|, \ \text{for all} \ t\in [t_0,T[.
	$$ 
	Using that $\kappa\geq 2T \kappa_h$ and $b<1$, we thus obtain  $v^*\leq h$ on $\partial_0 \Omega_T$. 
	By construction, $v$ satisfies \eqref{eq: Lip constraint}. Therefore $v \in  \mathcal{S}^{\kappa}_{h,g,F} (\Omega_T)$,
	hence $v \leq W$ in $\Omega_T$. 
	We infer in particular $u\leq W$ on $\Omega_{t_0}$. Taking supremum over all candidates $u$ we obtain $V\leq W$ in $\Omega_{t_0}$. 
	Using Proposition \ref{pro: subsolution slice}  we can let $t_0\to S$ to obtain $V\leq W$ in  $\Omega_{S}$. The reverse inequality is obvious. 
\end{proof}

\begin{theorem}
	\label{thm: Lipschitz constraint}
	There exists an explicit 
	 $\kappa_0>0$ such that, for all $\kappa>\kappa_0$,
 $$
 \sup_{\Omega_T} t |\partial_t U^{\kappa}|\leq \kappa_0.
 $$ 
\end{theorem}
\begin{proof}
	We use the same notations as in the proof of Theorem \ref{thm: U is local Lip in t}. Define 
	\begin{equation}
		\label{eq: Constant C of Thm Lip constraint}
		C := \kappa_FT+ 2\kappa_F M_U+ 2M_F +2 \kappa_h+ 2M_h +2n,
	\end{equation}
	and
	\begin{equation}
		\label{eq: Constant kappa0 of Thm Lip constraint}
		\kappa_0 := 2M_U + 3C(T+1) + 2 \sup_{\Omega}|\rho|.
	\end{equation}
	Fix $\kappa>\kappa_0$. By definition of $\kappa_0$ we have, for all $t\in ]0,T[$, $2\kappa_h \leq \kappa_0/t$.  Proposition \ref{pro: identity principle 2} thus ensures that, for all $T/2<S<T$, 
	\begin{equation}
		\label{eq: identity principle}
		U^{\kappa}_{h,g,F,\Omega_T}=U^{\kappa}_{h,g,F,\Omega_S}\ \text{in} \ \Omega_S.
	\end{equation}
	Fix $u\in \mathcal{S}^{\kappa}$, $T/2<S<T$,  
	$s>0$ close enough to $1$ and set, for $(t,z)\in \Omega_S$, 
	$$
	w(t,z):= as^{-1}u(st,z) + (1-a) \rho -C(1-a)(t+1),
	$$
	where $a=1-2|s-1|>0$, $\rho$ is defined in \eqref{eq: rho}.  
	
	Since $u$ is a subsolution to \eqref{eq: CMAF} we have, for almost all $t\in ]0,T[$, 
	$$
	(dd^c s^{-1} u(st,\cdot))^n \geq \exp\{-n\log s + \partial_{\tau} u(st,z) + F(st,z,u(st,z))\}gdV. 
	$$
	 It thus follows from Lemma \ref{lem: mixed MA} that 
	\begin{flalign*}
		(dd^c w)^n & \geq \exp \{a \partial_t u(st,z) + a F(st,z,u(st,z)) -a n\log s\} gdV\\
		&= \exp \{\partial_t w(t,z) + C(1-a) -an\log s +  aF(st,z,u(st,z))\} gdV.
	\end{flalign*}
	From \eqref{eq: Lip F} and the assumption that $F$ is increasing in $r$ we obtain
	\begin{flalign*}
	a F(st,z,u(st,z)) & = F(st,z,u(st,z)) + (1-a) F(st,z,u(st,z)) \\
	& \geq  F(t,z,as^{-1}u(st,z)) - |s-1|( \kappa_F T+ 2\kappa_F M_U+2M_F)\\
	& \geq 	F(t,z,w(t,z)) -|s-1|( \kappa_F T+ 2\kappa_F M_U+2M_F).
	\end{flalign*}
	For $(\tau,\zeta) \in \partial_0 \Omega_S$ we have 
	\begin{flalign*}
		w(\tau,\zeta) & \leq as^{-1}h(st,\zeta) -2C|s-1|\\
		& \leq h(st,\zeta) + |as^{-1}-1|M_h -2C|s-1|\\
		& \leq h(t,\zeta)+ 2|s-1|\kappa_h + 2M_h|s-1| -2C|s-1|. 
	\end{flalign*}
	
	The choice of $C$ in \eqref{eq: Constant C of Thm Lip constraint} and the previous computations 
	ensure   that $w\in \mathcal{S}_{h,g,F}(\Omega_S)$. 
	 Moreover, for $s\in [1/2,3/2]$, $t\in ]0,S[$, 
	\begin{flalign*}
		\sup_{\Omega}|\partial_t w(t,z)| \leq \frac{(1-2 |s-1|)\kappa}{\min(st,b)}   + 2C|s-1|.
	\end{flalign*}
	  Since $\kappa/t > \kappa_0/t >3C$, it follows that  for $s\in [1,3/2]$, $t\in ]0,S[$, 
	\begin{eqnarray*}
		\sup_{\Omega}|\partial_t w(t,z)| & \leq &   \frac{(3-2 s) \kappa}{\min(t,b)}   + \frac{2 (s-1) \kappa}{ 3 t}   \leq   \frac{\kappa}{\min(t,b)}.
	\end{eqnarray*}
		  Hence $w\in \mathcal{S}^{\kappa}_{h,g,F}(\Omega_S)$. By definition of $U^{\kappa}$  and \eqref{eq: identity principle} we have $w\leq U^{\kappa}$ on $ 
 \Omega_S$. Taking supremum over  $u\in \mathcal{S}^{\kappa}_{h,g,F}(\Omega_T)$ we obtain, for all $(t,z)\in \Omega_S$,
	$$
	a s^{-1}U^{\kappa}(st,z) -2C|s-1|(t+1) + 2|s-1|\rho(z)  \leq U^{\kappa}(t,z).
	$$
	Letting $s\to 1$ yields
	$$
	t|\partial_t U^{\kappa}(t,z)| \leq 2M_U + 2C(T+1) + 2\sup_{\Omega}|\rho|\leq \kappa_0,
	$$ 
	where in the last inequality we use \eqref{eq: Constant kappa0 of Thm Lip constraint}. This concludes the proof. 
\end{proof}

\subsection{The maximal subsolution}
We now prove that  $U\in \mathcal{S}_{h,g,F}(\Omega_T)$.

 \begin{theorem} \label{thm: U is subsolution}  
 Assume  $h$ satisfies \eqref{eq: local Lipschitz condition h} and set $U := U_{h,g,F,\Omega_T}$. Then $U\in \mathcal S_{h,g,F}(\Omega_T)$ and satisfies the following properties:
 \begin{enumerate}
\item  $ \lim_{\Omega_{T} \ni (t,z) \to (s,\zeta)} U (t,z) =  h (s,\zeta)$ for all $(s,\zeta) \in [0,T[ \times \partial \Omega$; \label{item: max subsolution Dirichlet}
\item  $ \limsup_{\Omega_{T} \ni (t,z) \to (0,z_0)} U (t,z) =  h (0,z_0)$ for all $(0,z_0) \in \{0\} \times {\Omega}$; 
 \label{item: max subsolution limsup 0}
 \item $ \lim_{t \to 0} U_t (z) =  h_ 0 (z)$ for all $z \in \Omega$.  
 \label{item: max subsolution limit t to 0}
\end{enumerate}

If $h_0$ is continuous then  for all $(s,\zeta) \in \partial_0 \Omega_T,$  
$$
\lim_{\Omega_{T} \ni (t,z) \to (s,\zeta)} U (t,z) =  h (s,\zeta).
$$
\end{theorem}
 
 \begin{proof} 
We proceed in several steps.

\smallskip

\noindent {\it Step 1.} {\it Assume $h$ satisfies \eqref{eq: h is uniformly Lipschitz t}.}

Theorem \ref{thm: boundary value of U} ensures that $U^*$ has the desired boundary values. 
We are going to prove that $U^{*}$ is a subsolution to \eqref{eq: CMAF}. 

\medskip

\noindent{\it Step 1.1.}  {\it Assume  $h_0$ is continuous on $\bar{\Omega}$. }

Fix $\kappa\geq \kappa_0$, where $\kappa_0$ is defined in Theorem \ref{thm: Lipschitz constraint}. 

\smallskip

\noindent{\bf Claim 1: $U^{\kappa}=(U^{\kappa})^*\in \mathcal{S}^{\kappa}_{h,g,F}(\Omega_T)$. }  

Indeed, since $U^{\kappa}\leq U$, the boundary condition $(U^{\kappa})^* \big |_{\partial_0\Omega_T}\leq h$ is satisfied. We now prove that $(U^{\kappa})^*$  is a subsolution to \eqref{eq: CMAF}.  A classical 
 lemma of Choquet ensures that there exists a  sequence $\{u^j\}$ in $\mathcal{S}^{\kappa}(h,g,F,\Omega_T)$ such that
 $$
  (U^{\kappa})^*  = \left(\sup_{j\in \mathbb{N}} u^j\right)^* \  \text{in} \, \, \Omega_T.
 $$ 
 By Lemma \ref{lem: uniform bound}, we can assume $\sup_{\Omega_T} |u^j|\leq M_U$. 
 Since $\mathcal{S}^{\kappa}$ is stable under taking maximum we can  assume that $\{u^j\}$ is increasing. By definition of $\mathcal{S}^{\kappa}$,  $\lim_{j} u^j$ is locally uniformly Lipschitz in $t\in ]0,T[$. Hence from  Lemma \ref{lem: negligible}  it follows that $u^j$ increases to $(U^{\kappa})^*$ almost everywhere in $\Omega_T$.   We infer that $ d t \wedge (dd^c u_j)^n   \to  d t \wedge (dd^c (U^{\kappa})^*)^n  $ weakly  in $\Omega_T$. Moreover,  the sequence $\{\psi_j\}:= \{\partial_t u^j+ F(t,z,u^j)\}$ is bounded and converges in the sense of distributions to $\partial_t (U^{\kappa})^* + F(t,z,(U^{\kappa})^*)$.  Proposition  \ref{prop: semi continuity} thus yields
 \begin{flalign*}
e^{\partial_t (U^{\kappa})^* + F (t,z,(U^{\kappa})^*)} g d t \wedge d V  \leq \liminf_j e^{\partial_t u^j + F (t,z,u^j)} g d t \wedge d V,
 \end{flalign*}
 weakly in $\Omega_T$.
 Therefore,  $(U^{\kappa})^*$ is a subsolution to  \eqref{eq: CMAF}  in $\Omega_T$. Hence $(U^{\kappa})^*=U^{\kappa}$ and Claim 1 is proved. 
 
  It now follows from Theorem \ref{thm: Lipschitz constraint} that $U^{\kappa} = U^{\kappa_0}$, for all $\kappa>\kappa_0$.

\medskip

{\noindent \bf Claim 2: $U=U^{\kappa_0}$ in $\Omega_T$.}
  
Fix $v \in \mathcal{S}_{h,g,F}(\Omega_T)$,  $S\in ]T/2,T[$, $\varepsilon>0$ small enough. Define, for $(t,z) \in [0,S]\times \Omega$, 
$$
u(t,z) := v(t+\varepsilon,z) - C \varepsilon (1+t) - \theta(\varepsilon), 
$$
where $C>0$ is a uniform constant and $\theta(\varepsilon):= \sup_{\bar \Omega} |U^*(\varepsilon,z)-h_0(z)|$ 
converges to $0$  (by Lemma \ref{lem: uniform convergence}). 
Since $v^*\leq h$ on $\partial_0 \Omega_T$, we obtain for all $(\tau,\zeta) \in [0,S]\times \partial \Omega$,
 $$
 u(\tau,\zeta) \leq h(\tau+\varepsilon ,\zeta) -C\varepsilon \leq h(\tau,\zeta), 
 $$
 if $C\geq \kappa_h$. By definition of $\theta(\varepsilon)$ we also have $u(0,z) \leq h_0(z)$ in $\Omega$.

A direct computation shows that, for $C>0$ large enough, $u \in \mathcal{S}_{h,g,F}(\Omega_S)$. Since $u$ is uniformly Lipchitz in $[0,S]$, $u\in \mathcal{S}^{\kappa}_{h,g,F}(\Omega_S)$ for some $\kappa>0$ large enough. Hence $u \leq U^{\kappa_0}$ in $\Omega_S$. Letting $\varepsilon\to 0$ we obtain $v\leq U^{\kappa_0}$ in $\Omega_S$. Letting $S\to T$ we arrive at $v \leq U^{\kappa_0}$, hence $U\leq U^{\kappa_0}$. Therefore $U=U^{\kappa_0}$ is  the maximal  subsolution to \eqref{eq: CMAF} with boundary value $h$. 

\smallskip

\noindent{\it Step 1.2.} We now remove the continuity assumption on $h_0$. 

Using Lemma \ref{lem: continuous approximation}  we can find  a sequence $h^j_0$ of psh functions in $\Omega$ such that  $h^j_0$ is continuous on $\bar{\Omega}$, $h^j_0=h_0$ on $\partial\Omega$, and $h^j_0 \downarrow h_0$ in $\Omega$.
We then define $h^j(t,z):= h(t,z)$ for $(t,z)\in [0,T[ \times \partial \Omega$ and $h^j(0,z)=h^j_0(z)$ for $z\in \Omega$. 
 We thus obtain a sequence of continuous  Cauchy-Dirichlet boundary data for $\Omega_T$ such that $h^j=h$ on $[0,T[\times \partial\Omega$ and $h^j$ decreases pointwise to $h$.   The previous step ensures that $U^j:= U_{h^j,g,F,\Omega_T}$ is a subsolution to \eqref{eq: CMAF}.  Theorem \ref{thm: U is local Lip in t} and Theorem \ref{thm: U is local semiconcave} provide a uniform Lipschitz constant for $U^j$. Since $h^j$ decreases to $h$,  $U\leq U^j$ decreases to some $V\in \mathcal{P}(\Omega_T)$. We thus have $V^*\big |_{\partial_0 \Omega_T}\leq h$, and  Proposition \ref{prop: semi continuity} reveals that  $V$ is a subsolution to  \eqref{eq: CMAF}.  It then follows that $V=U$.

\smallskip

\noindent{\it Step 2.} To treat the general case we proceed by approximation as in the proof of Lemma \ref{lem: uniform convergence}. 
Fix $0<S<T$ and $0 < \varepsilon < (T-S)/2$.  Define 
$$
\left\{
\begin{array}{ll}
h^{\epsilon }(t,\zeta) :=h(t+\epsilon,\zeta)  &  \text{ if } (t,\zeta)\in [0,S]\times \partial \Omega \\ 
h^{\epsilon } (0,z)=h_0(z)+ \psi^{\epsilon}(z)  & \text{ if } \ z\in \Omega, 
\end{array}
 \right.
$$
where $\psi^{\epsilon}$ is the maximal psh function in $\Omega$ such that  $\psi^{\epsilon}(\zeta)= h(\epsilon, \zeta)-h (0,\zeta)$ in $\partial \Omega$. 
Then $\{h^{\epsilon}\}$   uniformly converges  on $\partial_0 \Omega_{S}$ to $h$ as $\epsilon\to 0$. Since $h^{\epsilon}$ is uniformly Lipschitz in $t\in [0,S]$,
  the previous step and Theorem \ref{thm: boundary value of U} ensure that $U^\e := U_{h^\e,g,F,\Omega_S} \in \mathcal S_{h^\e,F,g} (\Omega_{S})$   satisfies the boundary conditions \eqref{item: max subsolution Dirichlet}, \eqref{item: max subsolution limsup 0}, \eqref{item: max subsolution limit t to 0}. Moreover, it  follows Proposition \ref{pro: identity principle}  that the envelopes $U^\e$ uniformly converge in $\Omega_{S}$ to $U$ as $\e \to 0$.
Hence, Proposition \ref{prop: semi continuity} and Proposition \ref{prop: dt MA converge} (together with Remark \ref{rem: dt MA converge})  yield that  $U$ is a subsolution to \eqref{eq: CMAF} and $U$ satisfies the boundary conditions \eqref{item: max subsolution Dirichlet}, \eqref{item: max subsolution limsup 0}, \eqref{item: max subsolution limit t to 0}.

If $h_0$ is continuous on $\bar{\Omega}$ then Lemma \ref{lem: uniform convergence} and the three boundary conditions \eqref{item: max subsolution Dirichlet}, \eqref{item: max subsolution limsup 0}, \eqref{item: max subsolution limit t to 0} give the last statement. 
\end{proof}

\subsection{Semi-concavity in the time variable}

In this section we assume that  $h$ satisfies \eqref{eq: local Lipschitz condition h} 
and  there exists  $C_h>0$ such that,  for all $z\in \partial \Omega$,
\begin{equation}
	\label{eq: local semiconcave condition h}
	\partial_t^2 h(t,z) \leq C_h t^{-2} 
\end{equation}
in the sense of distributions in $]0,T[$. 
Condition \eqref{eq: local semiconcave condition h}  is equivalent to the fact that $t\mapsto h(t,z)+C_h\log t$ is concave in $]0,T[$. It implies in particular that $h$ is locally uniformly semi-concave in the $t$-variable.

\begin{thm}\label{thm: U is local semiconcave}
 Assume $h$ satisfies \eqref{eq: local Lipschitz condition h} and 
 \eqref{eq: local semiconcave condition h}. 
 The envelope $U := U_{h,g,F,\Omega_T}$  is locally uniformly semi-concave in $]0,T[$ :
 for all $z\in \Omega$,  
 $$
 \partial_t^2 U(t,z) \leq C_U t^{-2}
$$
in the sense of distributions in $]0,T[$,
for some uniform constant $C_U>0$.  
\end{thm}

We will show that the constant $C_U$ is actually explicit, 
 \begin{equation}
 	\label{eq: constant C U}
 	C_U:= C_h  + 2M_h +8 \kappa_h +  (2\kappa_F+3) (M_U+5\kappa_U+ 1 + C_F T^2 + 16 \kappa_U^2).
 \end{equation}
 This quantitative information is important in perturbation arguments, to obtain uniform semi-concavity constants of the approximants. 
 
By the assumption on $F$, there is a constant $C_F>0$ such that for all $z\in \Omega$, the function 
\begin{equation} \label{eq: F local semi-convex}
	(t,r) \mapsto F(t,z,r) + C_F (t^2+r^2) \text{ is convex in}\  [0,T] \times [-2M_U,2M_U]. 
\end{equation}

\begin{proof} 
It follows from Theorem \ref{thm: U is subsolution} that $U \in \mathcal{S}_{h,g,F}(\Omega_T)$. Fix $0<S<T$, and $s>1/2$ close to $1$ enough such that $sS<T$. 
Set, for $(t,z) \in \Omega_S$,
$$
v^s (t,z) := \frac{s^{-1}U(st,z) + s U(s^{-1}t,z)}{2} -C (t +1)(s-1)^2,
$$
where $C>0$ is defined as 
\begin{equation}
	\label{eq: constant C concave U}
	C:= C_h +1 + 2M_h +8 \kappa_h +  2\kappa_F (M_U+4\kappa_U+ T + C_F T^2 + 16 \kappa_U^2).
\end{equation}
We are going to prove that $v^s \in \mathcal{S}_{h,g,F}(\Omega_S)$. 

\smallskip

\noindent {\it Boundary values of $v^s$.}  It follows from \eqref{eq: local semiconcave condition h} that 
for all $z\in \partial \Omega, t \in ]0,S[$,  
\begin{flalign*}
\frac{ h(st,z) +  h(s^{-1}t,z)}{2} & \leq h\left( \frac{(s+s^{-1})t}{2},z\right ) + C_h\log \left ( \frac{s+s^{-1}}{2} \right )\\	 
& \leq h\left( \frac{(s+s^{-1})t}{2},z\right )  + C_h (s-1)^2 \\
& \leq h(t,z) + (C_h+1)(s-1)^2,
\end{flalign*}
where in the last line we use \eqref{eq: local Lipschitz condition h reformulation}. We claim that for all $(t,z) \in ]0,S[ \times \partial \Omega$,   
\begin{flalign*}
	  s^{-1}h(st,z)+sh(s^{-1}t,z) \leq h(st,z) +h(s^{-1}t,z) +  (2M_h +3 \kappa_h)(s-1)^2.
	\end{flalign*}
	Indeed, write $s = 1 - \sigma$ and observe that $s^{-1} = 1 + \sigma + O (\sigma^2),$ where $\vert O (\sigma^2) \vert \leq 2 \sigma^2$ for $\vert \sigma \vert\leq  1 \slash 2$.
Thus for all $(t,z) \in ]0,S[ \times \partial \Omega$, 
\begin{flalign*}
s^{-1}h(st,z)&+sh(s^{-1}t,z)  \leq  (1 + \sigma) h (st,z) + (1-\sigma) h (s^{-1} t,z) + 2 M_h \sigma^2 \\
& \leq  h (st,z) + h (s^{-1}t,z) + \sigma (h (st,z) - h (s^{-1}t,z)) + 2 M_h \sigma^2.
\end{flalign*}
Using   \eqref{eq: local Lipschitz condition h reformulation}, we obtain  
\begin{flalign*}
s^{-1}h(st,z)+sh(s^{-1}t,z)  \leq   h (st,z) + h (s^{-1}t,z) +  (2 M_h + 4 \kappa_h) (s-1)^2,
\end{flalign*}
which proves   the claim.

Since $U^*\big |_{\partial_0 \Omega_T}\leq h$, the above estimate implies that   $(v^s)^* \leq h$ on $\partial_0 \Omega_S$. 
Using  similarly the estimate in Theorem \ref{thm: U is local Lip in t},  we obtain the following estimate which will be useful later:  for all $(t,z) \in ]0,S[ \times \bar{\Omega}$, 
\begin{eqnarray}	\label{eq: local Lipschitz U reformulation}
	 \vert (U(st,z) +U(s^{-1}t,z)) & - & (s^{-1}U(st,z)+sU(s^{-1}t,z)) \vert \\
	  & \leq & (2M_U+ 4\kappa_U)(s-1)^2. \nonumber
\end{eqnarray}

\smallskip

\noindent {\it Estimating the Monge-Amp\`ere measure of $v^s$.} 
It follows from Proposition \ref{pro: subsolution slice} that for almost all $t\in ]0,S[$, 
$$
(dd^c s^{-1}U(st,\cdot))^n \geq e^{n \log s^{-1} + \partial_{\tau} U(st,\cdot) + F(st, \cdot, U(st,\cdot))} gdV.
$$
Using Lemma \ref{lem: mixed MA} we infer
\begin{flalign*}
 (dd^c v^s(t,\cdot))^n \geq e^{a(s) + a(s^{-1})} g dV,	
\end{flalign*}
where 
$$
a(s) = \frac{1}{2} \left ( \partial_{\tau} U(st,\cdot) + F(st, \cdot, U(st,\cdot)) \right ).
$$
 By the semi-convexity assumption  \eqref{eq: F local semi-convex} on $F$, for  $\lambda \in ]0,1[$, $t_1,t_2 \in [0,T]$, $r_1,r_2 \in [-2M_U, 2M_U]$ we have 
\begin{flalign*}
	F(\lambda(t_1,r_1) + (1-\lambda)(t_2,r_2)) & \leq \lambda F(t_1,r_1) + (1-\lambda) F(t_2,r_2) \\
	&+ C_F \lambda(1-\lambda) \left ( (t_1-t_2)^2 +(r_1-r_2)^2 \right ). 
\end{flalign*}
Applying this for $(t,r) \mapsto F(t,z,r)$, $z\in \Omega$,  $\lambda = 1/2$, $t_1=st$, $t_2=s^{-1}t$, $r_1= U(st,z)$, $r_2=U(s^{-1}t,z)$,  we obtain
\begin{flalign*}
	& \; \;  \frac{1}{2} F(st, z, U(st,z)) + \frac{1}{2} F(s^{-1}t,z, U(s^{-1}t,z)) \geq  \\ 
	 & F\left ( \frac{(s+s^{-1})t}{2}, z,  (U(st,z)+U(s^{-1}t,z))/2\right )     \\
	 & - \frac{C_F}{4} \left (  t^2(s-s^{-1})^2 + (U(st,\cdot) -U(s^{-1}t,\cdot))^2 \right ). 
\end{flalign*} 
Using \eqref{eq: Lip F}, \eqref{eq: local Lipschitz U reformulation}, and the fact that $F$ is increasing in $r$, 
we thus get
\begin{flalign*}
	& \; \;  \frac{1}{2} F(st, z, U(st,z)) + \frac{1}{2} F(s^{-1}t,z, U(s^{-1}t,z)) \geq  \\ 
	 & F (t, z,  v^s(t,z) )   - \kappa_F (M_U+ 2\kappa_U+t) (s-1)^2 \\
	 & - \frac{C_F}{4} \left (  t^2(s-s^{-1})^2 + (U(st,\cdot) -U(s^{-1}t,\cdot))^2 \right )\\
	 &\geq F(t,z,v^s(t,z)) - (\kappa_F (M_U+2\kappa_U+ T)  + 2 C_F( T^2 + 2 \kappa_U^2)) (s-1)^2.  
\end{flalign*}
The choice of $C$ \eqref{eq: constant C concave U} ensures that
\begin{flalign*}
	a(s) + a(s^{-1}) & \geq  \partial_t v^s(t,\cdot) +  F(t,\cdot, v^s(t,\cdot)). 
\end{flalign*}

Altogether we conclude that $v^s\in \mathcal{S}_{h,g,F}(\Omega_S)$. Using Proposition \ref{pro: identity principle} we 
infer $v^s \leq U$ in $\Omega_S$. From this and \eqref{eq: local Lipschitz U reformulation} we obtain 
that for all $(t,z)\in \Omega_S$,
$$
\frac{U(st,z) + U(s^{-1}t,z)}{2} -U(t,z) \leq (C+2M_U +8\kappa_U)(s-1)^2.
$$
An elementary computation then yields (letting $s\to 1$) that $\forall (t,z) \in \Omega_S$,
$$
t^2 \partial^2_t U(t,z) \leq (9\kappa_U+2M_U +C) .
$$
We finally let $S\to T$ to conclude the proof. 
 \end{proof}

\section{Space regularity of parabolic envelopes}

We establish the first steps of a balayage process by studying solutions constructed in small balls, and
establishing space regularity of $U_{h,g,F, \mathbb{B}_T}$ :
assuming adequate regularity conditions on the data we   prove
that $U_{h,g,F, \mathbb{B}_T}$ is locally $\mathcal{C}^{1,1}$ in $z\in \mathbb{B}$. 
 
 We assume that $T<+\infty$, and $h$ satisfies \eqref{eq: local Lipschitz condition h} 
 and \eqref{eq: local semiconcave condition h}.

\subsection{Continuity in the space variable}\label{subsect: U is Lipschitz z}

Let $(Y,d_Y)$ be a metric space.
The  uniform partial modulus of continuity in the space variable $y \in Y$ of a function
$u : [0,T[ \times Y \longrightarrow \R$ is
$$
\eta (u,\delta) := \sup \{\vert u (t,y_1) - u (t,y_2)\vert \setdef  t\in [0,T[, \ y_1, y_2 \in Y,\  d_Y (y_1,y_2) \leq \delta\}.
$$

In particular, the uniform partial  modulus of continuity of $F$ is defined as above with $Y:=\Omega \times \mathbb{R}$. 

\begin{thm} \label{thm: U is Lip z}  
Assume the following conditions :
\begin{itemize}
	\item $G:= \log g$ is continuous in $\Omega$ ; 
	\item there exists  $u\in \mathcal{S}_{h,g,F}(\Omega_T)\cap \mathcal{C}([0,T[\times \bar{\Omega})$,  such that $u=h$ on $\partial_0 \Omega_T$.
\end{itemize}

 Then  $U := U_{h,g,F,\Omega_T}$ is continuous on $[0,T[ \times \bar {\Omega}$ and  
\begin{equation}
	\label{eq: modulus of continuity of U}
	 \eta (U,\delta) \leq  \eta (u,\delta) + \eta (H,\delta) +  (\eta (F,\delta) + \eta(G,\delta)) T.
\end{equation}
 \end{thm}
 Recall that $H_t$ is the unique harmonic function in $\Omega$ with $H_t=h_t$ on $\partial\Omega$. 
 
\smallskip 
 
 A continuous subsolution which agrees with $h$ on $\partial_0 \Omega$ is called a {\it subbarrier}.
Such  a subbarrier (for the whole boundary $\partial_0 \Omega_T$ as required in the Theorem)
 exists when $h$ is uniformly Lipschitz in $[0,T[$ and continuous on $\partial_0 \Omega_T$ by Lemma \ref{lem: sub-barriers Dirichlet}, Lemma \ref{lem: sub-barriers Cauchy} and Lemma \ref{lem: super-barrier Dirichlet}.

\begin{proof} 
It follows from  Theorem \ref{thm: boundary value of U} that $U$ continuously extends to the boundary $\partial_0 \Omega_T$ so that $U=h$ on $\partial_0  \Omega_T$. We use  the perturbation method of Walsh \cite{Walsh_1969Envelopes} to extend this property to the interior and prove that $U$ is continuous on $[0,T[\times \bar{\Omega}$.

Fix  $\delta > 0$ small enough.
Since $u  = h = U$ in $[0,T[ \times \partial \Omega$, 
we infer that for all $t\in [0,T[$, $z \in {\Omega}, \zeta \in \partial \Omega$ with  $\vert z -  \zeta\vert \leq \delta$,
\begin{equation}
	\label{eq: Lip ineq of U}
 U (t,\zeta) = u (t,\zeta) \leq u (t,z) + \eta (u,\delta) \leq U (t,z) + \eta (u,\delta).
\end{equation}
 Fix $\xi \in \C^n$   such that $\vert \xi\vert \leq \delta$ and set $\Omega_\xi := \Omega - \xi$ and  consider
\[
W(t,z) := 
\begin{cases}
	U  (t,z), \, \,   \text{ if } t \in [0,T[, z \in \Omega \setminus \Omega_\xi,\\
	 \max \{ U (t,z), U (t , z + \xi) -  \eta (u,\delta)\},   \text{ if }  t \in [0,T[, z \in \Omega \cap \Omega_\xi.
\end{cases} 
\]

By \eqref{eq: Lip ineq of U} the two definitions coincide when $(t ,z) \in [0, T[ \times \Omega$ and $z + \xi \in \partial \Omega$. Therefore  $W\in \mathcal{P}(\Omega_T)$. We are going to prove that $W-O(\delta)(t+1) \in \mathcal{S}_{h,g,F}(\Omega_T)$ for some small error term $O(\delta)$.
 
 \smallskip
 
 \noindent {\it The subsolution property.} 
 By Lemma \ref{lem: time derivarive of max}, for a.e. $(t,z) \in [0, T[ \times (\Omega \cap \Omega_\xi)$,
 \begin{eqnarray*}
 \partial_t W (t,z)  = {\mathbbm{1}}_{\{U  (t,z)<  \tilde U (t,z)\}} \partial_t \tilde U (t, z) 
  +  {\mathbbm{ 1}}_{\{U  (t,z)\geq \tilde U (t,z) \}} \partial _t U (t,z),
 \end{eqnarray*}
 where
 $$
 \tilde U (t,z) :=  U (t , z + \xi) - \eta(u,\delta) \, \, \, \text{in} \, \, \Omega \cap \Omega_\xi.
 $$
Moreover
 \begin{eqnarray*}
 && e^{\partial _t \tilde U (t, z) + F (t, z, \tilde U (t,z)) + G (z)}  d t \wedge d V (z) \\
  &&\leq  e^{\partial_t U (t, z + \xi) + F (t, z + \xi, U (t,z + \xi)) + \eta (F,\delta) + G (z + \xi) + \eta(G,\delta)} d t \wedge d V (z) \\
  && \leq e^{\eta (F,\delta) + \eta(G,\delta)} d t \wedge (dd^c \tilde U)^n, 
\end{eqnarray*}
in the weak sense on $]0,T[ \times (\Omega \cap \Omega_\xi)$.
We thus obtain  
$$
e^{\partial_t W (t,z) +  F (t, z, W (t,z)) + G(z)}  d t \wedge d V (z)   \leq e^{b(\delta)} d t \wedge (dd^c W)^n,
$$ 
i.e. the function defined on $[0,T[ \times \Omega$ by  
$
W_1 (t, z) :=  W (t,z) - b(\delta) t,
$
is a subsolution to \eqref{eq: CMAF} in $]0,T[ \times \Omega$. 
Here $b(\delta):=\eta (F,\delta) + \eta(G,\delta)$.

\smallskip

\noindent {\it Estimating boundary values.}    It follows from Theorem \ref{thm: boundary value of U} that  
\[
\lim_{(t,z')\to (0, z)} U(t,z') =  h_0(z) , \ z \in \Omega.
\] 
By definition of $W$ and the assumption that $h_0=u$ on $\{0\}\times \Omega$, we obtain 
$$
\lim_{(t,z')\to (0, z)} W(t,z') \leq h_0(z), \ \text{for all}\  z\in \Omega.
$$
Fix  $(\tau,\zeta) \in [0,T[ \times \partial  \Omega$. 

Since $U \leq H$ in $\Omega_T$ and $U = h$ in $[0,T[\times \partial \Omega$, we infer
\begin{eqnarray*}
	\lim_{]0,T[ \times (\Omega \cap \Omega_{\xi}) \ni (t,z) \to (\tau,\zeta)} W(t,z) \leq \max(h(\tau,\zeta), H (\tau,\zeta+\xi)) \leq h(\tau,\zeta) +  \eta (H,\delta),
\end{eqnarray*}
and 
$$
\lim_{ ]0,T[ \times (\Omega \setminus \Omega_{\xi}) \ni (t, z) \to (\tau,\zeta)} W(t,z) = \lim_{ ]0,T[ \times(\Omega \setminus \Omega_{\xi}) \ni (t,z) \to (\tau,\zeta)} U(t,\zeta)=h(\tau,\zeta).
$$

From the computations above we conclude that  $W_1- \eta (H,\delta) \in \mathcal{S}_{h,g,F}(\Omega_T)$. Thus $W_1 - \eta (H,\delta) \leq U$ in $\Omega_T$, hence 
 \begin{equation*} 
 U (t, z + \xi)  -\eta(u,\delta)- \eta (H,\delta) - (\eta (F,\delta) + \eta (G,\delta)) t  \leq U (t,z),
 \end{equation*}
 for $ (t,z) \in [0,T[ \times (\Omega \cap \Omega_\xi)$  and  $\xi \in \C^n$ with   $\vert \xi\vert \leq \delta$. This gives \eqref{eq: modulus of continuity of U}.
 
 \smallskip
 
 The continuity of $U$ on $[0,T[\times \bar{\Omega}$ follows from Theorem \ref{thm: U is subsolution}, the continuity of $h_0$, the continuity of each slice $U(t,\cdot)$ on $\bar{\Omega}$ and the fact that $U$ is locally uniformly Lipschitz in $t\in ]0,T[$. 
 \end{proof}
 
 \begin{cor} \label{cor: U is Lip z}
 Assume that
 \begin{enumerate}
 \item  $G:=\log g$ is Lipschitz in $\Omega$;  \label{cor U Lip z G}
 \item  the family   $\{h(\cdot, z) \setdef  z\in \partial \Omega\}$ is  uniformly Lipschitz in $[0,T[$;  \label{cor U Lip z h t}
 \item  $h_0$ is Lipschitz on $\bar{\Omega}$; \label{cor U Lip z h 0}
 \item the family $\{h(t,\cdot) \setdef t\in ]0,T[\}$ is uniformly $\mathcal{C}^{1,1}$ on $\partial \Omega$; \label{cor U Lip z h C11}
 \item  the function $F$ is Lipschitz on $[0,T[ \times \Omega \times J$, for each $J\Subset \mathbb{R}$. \label{cor U Lip z F}
 \end{enumerate}
 
  Then the family $ \{ U(t,\cdot) \setdef  t \in [0,T[\}$ is uniformly Lipschitz on $\bar \Omega$. 
 \end{cor}
 
 \begin{proof}
 	It follows from Lemma \ref{lem: sub-barriers Dirichlet}, Lemma \ref{lem: sub-barriers Cauchy} 
 	and assumption \eqref{cor U Lip z h t} that there exists $u\in \mathcal{S}_{g,h,F}(\Omega_T) \cap \mathcal{C}([0,T[\times \bar{\Omega})$ with  $u\big |_{\partial_0 \Omega_T} =h$. \cite[Theorem 5.2]{GZbook} and  \eqref{cor U Lip z h 0}, \eqref{cor U Lip z h C11} ensure that the family $\{u(t,\cdot) \setdef t\in [0,T[\}$ is uniformly Lipschitz on $\bar{\Omega}$.  We now invoke Theorem \ref{thm: U is Lip z} to finish the proof.
 \end{proof}

\subsection{ $\mathcal{C}^{1,1}$-regularity in the space variable} \label{subsect: U is locally C1,1}

We prove the following regularity result:

\begin{theo}\label{thm: U is C1,1}
Assume $ \Omega = \B $ is the  unit ball, $T<+\infty$,  and
\begin{enumerate}
 \item  $G:=\log g \in C^{1,1} (\bar{\B})$;
 \item  $h$ satisfies the assumptions of Corollary \ref{cor: U is Lip z};
 \item  $F$ is Lipschitz and  semi-convex in $[0,T[ \times \bar{\B} \times J$, for each $J\Subset \mathbb{R}$. \label{item F  U C11}
\end{enumerate}
	 
 Then the  envelope $U_{h,g,F,\B_T}$ is locally uniformly  ${\mathcal C}^{1,1}$   in $z\in \B$. 

\end{theo}

By scaling and translating, the result still holds for any ball $B(z_0,r) \Subset \mathbb{C}^n$. 
In the proof below we use $C$ to denote various uniform constants which may be different from place to place.

\begin{proof}
The proof   is a parabolic analogue of   
\cite[Theorem 6.7]{Bedford_Taylor_1976Dirichlet}.  We follow closely the presentation of \cite[Theorem 5.3.1]{GZbook}. 
Recall from Corollary  \ref{cor: U is Lip z} that the family $\{U (t,\cdot) \setdef  t \in [0,T[\}$ is uniformly Lipschitz on $\bar{\B}$.

\smallskip
 
\noindent {\it Automorphisms of the ball $\B$.}
 For $a \in \B$, we set
$$
 \mathcal T_a (z) = \frac{P_a(z) -a + \sqrt{1-| a |^2 }(z-P_a(z))}{1-\langle z, a \rangle } \;\;\; ;\;\; P_a(z) = \frac{\langle z,a \rangle}{| a |^2} a 
 $$
where $ \langle \cdot, \cdot  \rangle $ denote the Hermitian product in  $\C^n$. 
It is well known (see \cite[Lemma 4.3.1]{Klimek_1991_Pluripotential}) that
$ \mathcal T_a $ is a holomorphic automorphism of the unit ball  
such that $\mathcal T_a (a) = 0$ and $\mathcal{T}_a(\partial \mathbb{B}) = \partial \mathbb{B}$. 
Note that $\mathcal T_0$ is the identity.
We set
$$
  \xi = \xi (a,z) := a - \langle z,a \rangle z.
$$
Observe that $\xi (- a,z) = - \xi (a,z)$. If   $\vert a\vert \leq 1 \slash 2$ then
$$
 \mathcal T_a (z)= z-\xi + O(\vert a \vert^2) ,
 $$
  where $ O(\vert a \vert^2) \leq C_0 \vert a\vert^2$, with $C_0$ a numerical constant independent of $z \in \B$ when $\vert a \vert \leq 1 \slash 2$.
Thus $\mathcal T_{\pm a}$ is the translation by $\mp \xi$ up  to small
second order terms, when $\vert a\vert$ is small enough.

We set, for $(t,z) \in \B_T$,
$$
 V_a (t,z) := \frac{1}{2} (U (t,\mathcal T_a  (z)) +  U(t,T_{- a} (z)). 
$$
We are going to prove that, for a uniform constant  $C>0$, the function $V_a-C|a|^2(t+1)$ belongs to $\mathcal{S}_{h,g,F}(\B_T)$. We proceed in two steps. 

\smallskip
 
\noindent {\it Step 1:  Boundary values of $V_a$}.
%
If $q$ is $\mathcal{C}^{1,1}(\bar{\B})$ then, as in \cite[Page 145]{GZbook}, 
\begin{equation} \label{eq: C11}
\left | q  (\mathcal T_a  (z)) + q (\mathcal T_{- a} (z))- 2 q (z) \right | \leq  2 C (q) \vert a\vert^2,
\end{equation}
where $C (q) > 0$ depends on the uniform $\mathcal{C}^{1,1}$-norm of $q$ on $\bar{\B}$. 

Since the family $\{h(t,\cdot)\setdef t\in [0,T[\}$ is uniformly Lipschitz in $\partial \B$,
applying  (\ref{eq: C11}) yields
\begin{equation*}
 h (t,\mathcal T_a  (z)) + h(t,\mathcal T_{-a}  (z)) \leq 2 h (t,z) + 2 C (h) | a |^2, 
\end{equation*}
for $z\in \partial \B$,  $|a|$ small enough, 
where $C (h) > 0$ depends on the uniform $\mathcal{C}^{1,1}$-bound of $h (t,\cdot)$ in a neighborhood of $\partial \B$. 
We infer, for all $(t,\zeta) \in \partial_0 \B_T$,
$$
 V_a (t,\zeta) \leq  h (t,\zeta) + C(h) | a |^2.
$$

\smallskip

\noindent {\it Step 2: Estimating the Monge-Amp\`ere measure of $V_a$.} 
Since $U$ is a subsolution to \eqref{eq: CMAF} a direct computation  shows that 
\begin{flalign*}
	(dd^c & U\circ \mathcal{T}_a)^n = |\det \mathcal{T}_a'|^2 (dd^c U)^n \circ \mathcal{T}_a\\
	&\geq  |\det \mathcal{T}_a'|^2 \exp \left( \partial_t U(t,\mathcal{T}_a(z)) + F(t,\mathcal{T}_a(z),U(t,\mathcal{T}_a(z)) +G(\mathcal{T}_a(z))  \right).  
\end{flalign*}
Since the function $(a,z) \mapsto \theta(0,z):= \log |\det \mathcal{T}_a'(z)|^2 + \log |\det \mathcal{T}_{-a}'(z)|^2$ is smooth in $\B_{1/2}\times \bar{\B}$ and  $\theta(0,z)=0$, the Taylor expansion yields
$$
\theta(a,z)+ \theta(-a,z) = O(|a|^2). 
$$ 

The assumption \eqref{item F U C11} provides us with a uniform constant $C$ such that 
\begin{flalign*}
	& \frac{1}{2}\left \{ F(t,\mathcal T_a(z),U(t, \mathcal{T}_{a}(z)) +  F(t, \Tc_{-a}(z), U(t,\Tc_{-a}(z)) \right \} \\
	& \geq F\left (t, \frac{\Tc_a(z)+ \Tc_{-a}(z)}{2}, V_a(t,z)\right ) \\
	& - C (\|\Tc_a(z)-\Tc_{-a}(z)\|^2+  (U(t,\Tc_{a}(z)-U(t,\Tc_{-a}(z))^2)\\
	& \geq F(t,z,V_a(t,z)) - C |a|^2,
\end{flalign*}
where in the last inequality we have used $\Tc_a(z)+\Tc_{-a}(z)-2z = O(|a|^2)$, $\Tc_a(z)-\Tc_{-a}(z)=O(|a|)$, and the Lipschitz regularity (in $z\in \bar{\B}$) of $U$.
Using this, and applying Lemma \ref{lem: mixed MA} and the uniform estimate \eqref{eq: C11} to the function $G$ we obtain 
\begin{eqnarray*}
 (dd^c V_a (t, \cdot))^n
\geq  \exp \left \{ \partial_t V_a + F(t,z,V_a(t,z)) +G(z) -C |a|^2 \right \}.
\end{eqnarray*}
By the computations above we conclude that the function 
 $$ 
 \B_T \ni (t,z) \mapsto W_a (t,z) =  V_a (t,z) - C|a|^2(t+1)
 $$
 belongs to $\mathcal{S}_{h,g,F}(\B_T)$. 
Therefore, for all $(t,z)\in \B_T$,
 $$ 
  V_a (t,z) - (T+1) C | a |^2  \leq   U (t,z). 
 $$

\medskip

%
%
%
%

From this estimate, we proceed as in \cite[page 146-147]{GZbook} 
to prove that the second order partial derivatives (in $z$)  of $U$ are locally bounded in $\B$.
\end{proof}

We  now show  that $U$ admits a Taylor expansion up to order $(1,2)$ :
 
 \begin{lem} \label{lem: Taylor expansion 1} Assume $(h,g,F,\B_T)$ is as in Theorem \ref{thm: U is C1,1}. Then
 the envelope  $U$ admits the following Taylor expansion at almost every point 
 $(t_0,z_0) \in \B_T$,
 \begin{flalign*}
 U(t,z) &= U(t_0,z_0)  +  (t-t_0)\partial_t U(t_0,z_0) + \Re P(z-z_0) + L(z-z_0) \\
	 & + o(|t-t_0|+|z-z_0|^2),   	
 \end{flalign*}
 where $P$ is a polynomial of degree $2$ and $L$ is the Levi form of $U(t_0,z)$ at $z_0$. 
 \end{lem}

 \begin{proof}
 It follows from Theorem \ref{thm: U is local semiconcave} that $U$ is locally uniformly semi-concave in $t\in ]0,T[$. Theorem \ref{thm: U is C1,1} ensures that $U$ is locally uniformly Lipschitz in $z\in \B$, hence, for all $t\in ]0,T[$, $U(t,\cdot)$ is twice differentiable at a.e. $z\in \B$. 
 
 Let $A_1$ be the set of points $(t_0,z_0) \in \Omega_T$ such that $U (\cdot,z_0)$ is not differentiable at $t_0$
and  $A_2$ be the set of points $(t_0,z_0) \in \Omega_T$ such that $U (t_0,\cdot)$ is not twice differentiable at $z_0$.
It follows from Fubini's Theorem that  the set $A := A_1 \cup A_2$ is of Lebesgue measure zero in $\Omega_T$.

We show that  the Taylor expansion  holds at any point $(t_0,z_0) \notin A$.
Fix $\e > 0$ and $(t_0,z_0) \notin A$. 
We first write for $(t,z) \in \Omega_T$,
$$
U (t,z) - U (t_0,z_0) = U (t,z) - U (t_0,z) + U (t_0,z) - U(t_0,z_0).
$$

Since $(t_0,z_0) \notin A_1$, the function $U(t_0,\cdot)$ is twice differentiable at $z_0$. 
Thus there exists $r > 0$ such that  for $\vert z - z_0\vert < r$,
\begin{equation} \label{eq:Asymp1}
\vert U (t_0,z) - U(t_0,z_0) -  \Re P (z -z_0) -  L (z-z_0) \vert \leq \e \vert z-z_0\vert^2.
\end{equation}

On the other hand since $(t_0,z_0) \notin A_2$, $\partial_t U(t_0,z_0)$ exists and we have
\begin{eqnarray*}
U (t,z) - U(t_0,z) - (t-t_0) \partial_t U (t_0,z_0) &=&  \int_{t_0}^t (\partial_{\tau}^+ U (\tau,z) - \partial_t U (t_0,z_0)) d \tau \\
 &=& \int_{t_0}^t (\partial_{\tau}^- U (\tau,z) - \partial_t U (t_0,z_0)) d \tau.
\end{eqnarray*}

 Since $\partial_{\tau}^+ U $ is lsc and $\partial_{\tau}^- U $ is usc,
  we can choose $r$ so small that for $\vert t - t_0\vert + \vert z- z_0\vert < r$,
\begin{equation} \label{eq:Asymp2}
| U (t,z) - U(t_0,z) - (t-t_0) \partial_t U (t_0,z_0) \vert \leq \e \vert t-t_0\vert.
\end{equation}
The Taylor expansion thus follows from (\ref{eq:Asymp1}) and (\ref{eq:Asymp2}).
  \end{proof}

 \section{Pluripotential solutions} \label{sect: pluripotential solutions}

 We finally prove in this section that  $U_{h,g,F,\Omega_T}$ is the unique pluripotential solution to the Cauchy-Dirichlet problem for \eqref{eq: CMAF} which is locally uniformly semi-concave.

 \subsection{The case of Euclidean  balls}
  We first treat the case when $\Omega$ is a euclidean ball in $\mathbb{C}^n$. 
  By scaling and translating, it suffices to treat the case of the unit ball.  
   
 \begin{theorem} \label{thm: U is solution ball reg}
   Let $\Omega=\mathbb{B}$ be the unit ball in $\mathbb{C}^n$, $T<+\infty$, and assume that
   \begin{enumerate}
   	\item $G := \log g$ is $\mathcal{C}^{1,1}$ in $\bar{\B}$; 
   	\item $h$ is  uniformly  $\mathcal{C}^{1,1}$  in $z\in \partial \B$,  $h_0$ is $\mathcal{C}^{1,1}$ in $\bar{\B}$;
   	\item $h$ is uniformly Lipschitz in $t\in [0,T[$ and  $\partial^2_t h \leq Ct^{- 2}$ on $]0,T[\times \partial \B$; 
   	\item $F$ is 
   	Lipschitz and
   	 semi-convex in $(t,z,r) \in [0,T[\times \bar{\B} \times J$ for each $J\Subset \mathbb{R}$.  
   \end{enumerate}
    Then 
for almost every $(t,z) \in \B_T$,  
 $$
 \mathrm{det} \left( \frac{\partial^2 U}{\partial z_j \partial \bar  z_k}  (t,z) \right) = e^{\dot U (t,z) + F (t,z,U(t,z)) + G(z)}.
 $$

In particular $U$ is a pluripotential solution to the Cauchy-Dirichlet problem for the parabolic equation \eqref{eq: CMAF} with boundary data $h$.
 \end{theorem}

 \begin{proof} 
Theorem \ref{thm: U is subsolution}  and the Lipschitz assumption on $h$ ensure that $U$ is a subsolution to  \eqref{eq: CMAF} with 
  $U=h$ on $\partial_0\Omega_T$. 
It  follows from Corollary \ref{cor: U is Lip z} and Theorem \ref{thm: U is C1,1} that $U$ is uniformly Lipschitz
 in $z\in \bar{\mathbb{B}}$ and locally $\mathcal{C}^{1,1}$ in $\mathbb{B}$. 
In particular  $U$ is  twice differentiable in $z$ 
 almost everywhere in $\Omega_T$, hence
 $$
 (dd^c U )^n =  \text{det} (U_{j,\bar k} (t,z)) d V (z).
 $$
 
As $U$ is also almost everywhere differentiable in $t$  
and a subsolution to the parabolic equation \eqref{eq: CMAF}, we infer by Proposition \ref{pro: subsolution slice},
 \begin{equation} \label{eq:Ineg-ponctuel}
 \text{det} (U_{j,\bar k} (t,z)) \geq  e^{\partial_tU (t,z) + F (t,z,U(t,z)) + G (z)},
 \end{equation}
almost everywhere in $\B_T$. 

We want to prove that equality holds in (\ref{eq:Ineg-ponctuel}).
We use the notation of the proof of Lemma \ref{lem: Taylor expansion 1} and set $E=\B_T \setminus A$.
Arguing by contradiction we assume that   
 $$
 \text{det} \left(U_{j,\bar k} (t_0,z_0) - \e I_n\right) > e^{\dot U (t_0,z_0) + F (t_0,z_0,U(t_0,z_0)) + G (z_0)+\varepsilon},
 $$
 at some point $(t_0,z_0)  \in E$,
 for a small constant $\varepsilon > 0$.
 
 We use a bump construction to produce a subsolution $v\in \mathcal{S}_{h,g,F}(\B_T)$ which satisfies $v (t_0,z_0) > U (t_0,z_0)$ providing a contradiction.
It follows from Lemma \ref{lem: Taylor expansion 1} that 
\begin{flalign}
		 U(t,z) -U(t_0,z_0)  &= (t-t_0)\partial_t U(t_0,z_0) + \Re P(z-z_0) + L(z-z_0)\nonumber \\
	&  + o(|t-t_0|+|z-z_0|^2). \label{eq: Taylor expansion 1}
\end{flalign}

Set $D_r := \{(t,z) \setdef  \vert t - t_0\vert + \vert z - z_0\vert^2  < r\}$ and define  
\begin{flalign*}
w (t,z) & :=  U (t_0,z_0) + \partial_t U (t_0,z_0) (t-t_0) + \Re P(z-z_0) \\
&+ L (z-z_0) + \delta  - \gamma (\vert z - z_0\vert^2+ |t-t_0|),
\end{flalign*}
where  $\delta ,\gamma >0$ are constants to be specified later. 
Note that if $\gamma$ is small enough then $w\in \mathcal{P}(D_r)$. 
For any $(t,z) \in D_r$, the Taylor expansion \eqref{eq: Taylor expansion 1} ensures that
$$
U (t,z) \geq w (t,z) + \gamma (\vert t - t_0\vert) + \vert z - z_0\vert^2) - \delta + o (r).
$$
Hence for any $(t,z) \in D_r \setminus D_{r \slash 2}$,
$$
U (t,z) \geq w (t,z) + \gamma r/2-  \delta + o (r) > w (t,z),
$$
if $\delta = \gamma r \slash 4$, and $r > 0$ is small enough. On the other hand for $(t,z) \in D_r$,  
\[
(dd^c w)^n = (dd^c U -\gamma |z-z_0|^2)^n (t_0,z_0),
\]
and for $(t,z) \in D_r$, $t \neq t_0$,
$$
\partial_t w (t,z) = \partial_t U(t_0,z_0) - \gamma (t-t_0) \slash \vert t - t_0\vert.
$$
Thus if $\gamma < \e$,   we obtain for any $(t,z) \in D_r$,
\begin{eqnarray*}
(dd^c  w (t,z))^n  &\geq & e^{\partial_t w (t,z) + \gamma (t-t_0) \slash \vert t - t_0\vert))  +  F (t_0,z_0,U(t_0,z_0))+ G(z_0)+\varepsilon} dV \\
&\geq & e^{\partial_t w(t,z) -\gamma + F(t,z,w(t,z))  +G(z)+ R(t,z) + \varepsilon} dV,
\end{eqnarray*}
where
$$
  R (t,z) := F (t_0,z_0,U(t_0,z_0)) -  F (t,z,w(t,z)) +  (G (z_0) - G (z)).
$$

Since $U$ and $F$  are locally Lipschitz, there exists   $A > 0$ such that for $r > 0$ small enough  and $(t,z) \in D_r$, 
$$
R (t,z) \geq -  A\sqrt{ r} \geq \gamma-\varepsilon.
$$ 
The function $ w $ is therefore a subsolution to \eqref{eq: CMAF} in $D_r$.

The previous estimates ensure that the function  
$$
v (t,z) := \left\{ 
\begin{array}{lcr}
\max \{U (t,z) , w (t,z)\} & \text{ if }  & (t,z) \in D_r \\
U (t,z) & \text{ if } & (t,z) \in \B_T \setminus D_r
\end{array} 
\right.
$$
belongs to $\mathcal{S}_{h,g,F}(\B_T)$, hence $v\leq U$ in $\B_T$. In particular, $w \leq U$ in $D_r$ which is a contradiction since $w (t_0,z_0) = U (t_0,z_0) + \delta > U (t_0,z_0)$.
 \end{proof}

We now relax the regularity assumptions in Theorem \ref{thm: U is solution ball reg}. 

\begin{pro}
	\label{pro: U is solution ball}
	Assume $\Omega=\mathbb{B}$ is the unit ball in $\mathbb{C}^n$, $T<+\infty$, and
	\begin{itemize}
	\item $G:=\log g$ is continuous in $\bar{\mathbb{B}}$; 
	\item $h$ is   continuous on $\partial_0 \mathbb{B}_T$ and  satisfies \eqref{eq: local Lipschitz condition h} and \eqref{eq: local semiconcave condition h};
		\item $F$ extends as a continuous function on $[0,T[\times \bar{\B} \times \mathbb{R}$ which is uniformly Lipschitz and  
		 uniformly semi-convex in  $(t,r) \in [0,T[ \times J$ for each $J\Subset \mathbb{R}$. 
	\end{itemize}
			Then $U_{h,g,F,\B_T}$ is a continuous solution to \eqref{eq: CMAF} with boundary values $h$. 
\end{pro}

\begin{proof} 
It follows from Theorem \ref{thm: U is subsolution} that $U \in \mathcal{S}_{h,g,F} (\B_T)$ satisfies the boundary conditions \eqref{eq: Dirichlet condition}, \eqref{eq: Cauchy condition}. 
 It remains to prove that
 $U$ is continuous on $[0,T[\times \bar{\Omega}$ and solves  \eqref{eq: CMAF} in $\Omega_T$. By Proposition \ref{pro: identity principle} it suffices to prove these statements in $\B_{S}$ for each fixed $S<T$.   We  proceed in several steps. 
 

\smallskip

\noindent{\it Step 1.} Assume that $h(\cdot,z)$  
is uniformly Lipschitz in $t\in [0,T[$. It follows from Theorem \ref{thm: U is Lip z} that $U$ is continuous  on $[0,T[\times \bar{\Omega}$. The goal is to prove that $U$ solves \eqref{eq: CMAF} in $\B_S$. We proceed by approximation as follows.

Let $(G_j)=(\log g_j)$ be a sequence of smooth functions uniformly converging to $G$ on $\bar{\mathbb{B}}$.  
Extending $F$ continuously in an open neighborhood of $[0,S]\times \bar{\B}\times \mathbb{R}$ and taking convolution  in  $(t,z,r)$  we can find a sequence $F_j : [0,S] \times \bar{B} \times \mathbb{R}$ of  functions which are smooth in $(t,z,r)$ and
\begin{itemize}
	\item Lipschitz   and 
	  semi-convex in $[0,S] \times \bar{\B}\times  J$ for each $J\Subset \mathbb{R}$; 
	\item uniformly converges   to $F$ on $[0,S] \times \bar{\B} \times J$, for each $J\Subset \mathbb{R}$.
\end{itemize}

We extend $h$ as a continuous function in $[0,T[ \times \{|z|\geq 1/4\}$ by setting
\[
h(t,z) := h\left (t,\frac{z}{|z|}\right) , z\in \mathbb{C}^n , |z| \geq 1/4; 
\]

The extension $h$ satisfies \eqref{eq: local Lipschitz condition h} and \eqref{eq: local semiconcave condition h} for all $|z|\geq 1/4$ (with the same constants $\kappa_h, C_h$ as the original function $h$ defined on $\partial_0 \B_T$).  
Taking convolution  in the $z$ variable
 we can find a sequence  $(\hat{h}^j)$ of functions in $[0,T[\times \{|z|>1/3\}$ which
 are smooth in $z$ and
\begin{itemize}
	\item are uniformly Lipschitz in $t$; 
	\item satisfy \eqref{eq: local semiconcave condition h} with the same uniform constant $C_h$;
	\item uniformly converge  to $h$ on $[0,S] \times \partial \B$.
\end{itemize}
Fix $j\in \mathbb{N}$ and define $h^j$ by
$$
\left\{
\begin{array}{cl}
h^j(t,z):=\hat{h}_j(t,z)  &  \text{ if } (t,z)\in ]0,T[\times \partial \mathbb{B} \\
h^j(0,z)=h_0+H^j & \text{ if } (t,z) \in \{ 0 \} \times  \mathbb{B},
\end{array}
 \right.
 $$
 where $H^j$ is the maximal plurisubharmonic function in $\mathbb{B}$ with boundary values $\hat{h}^j(0,\cdot)-h_0$. Observe that $h^j$ is a Cauchy-Dirichlet boundary data on $\B_T$ which satisfies the assumptions of Theorem \ref{thm: U is solution ball reg}.  Note also that $h^j$ uniformly converges  to $h$ on $\partial_0 \mathbb{B}_T$,
 since $H^j$ uniformly converges  to $0$. 

Set $U^j:= U_{h^j,g_j,F_j}(\B_{S})$, $j\in \mathbb{N}$.  Theorem \ref{thm: U is solution ball reg} ensures that $U^j$ is a pluripotential solution to the equation \eqref{eq: CMAF} and $U^j=h^j$ on $\partial_0 \B_{S}$. It also follows from Theorem \ref{thm: U is local Lip in t} and Theorem \ref{thm: U is local semiconcave} that $U^j$ is 
 locally uniformly semi-concave in $t\in ]0,S]$. Moreover, \eqref{eq: constant C Lipschitz U} and \eqref{eq: constant C concave U} ensure that the Lipschitz and semi-concave constants of $U^j$ are uniform.  
By definition of the envelope, $U^j$ uniformly converges  to $U$ as $j\to +\infty$. It thus follows from Proposition \ref{pro: convergence semiconcave}, Proposition \ref{prop: dt MA converge}, Remark \ref{rem: dt MA converge} and Lemma \ref{lem: convergence derivative of concave functions} that $U$ is a pluripotential solution to  \eqref{eq: CMAF} in $\B_S$. 

\smallskip

\noindent {\it Step 2.}  To treat the general case, we approximate $h$ by a family of functions $h^{\varepsilon}$ which are Lipschitz up to zero, in such a way that they satisfy \eqref{eq: local Lipschitz condition h} and \eqref{eq: local semiconcave condition h} with  constants independent of $\varepsilon$. 

We proceed as in the proof of Theorem \ref{thm: U is subsolution}.  
Fix $S>0$ and $\varepsilon>0$ such that $S+\varepsilon<T$, and define 
$$
\left\{
\begin{array}{ll}
h^{\varepsilon}(t,\zeta) =h(t+\varepsilon,\zeta)  &  \text{ if } (t,\zeta)\in [0,S]\times \partial \mathbb{B} \\
h^{\varepsilon}(0,z)=h_0(z)+ \phi_{\varepsilon}(z)  & \text{ if } \ z\in \mathbb{B},
\end{array}
 \right.
$$ 
where $\phi_{\varepsilon}$ is the maximal plurisubharmonic function in $\mathbb{B}$ such that  $\phi_{\varepsilon}(\zeta)= h(\varepsilon, \zeta)-h_0(\zeta)$ on $\partial \mathbb{B}$. Then $h^{\varepsilon}$ uniformly converges to $h$ on $\partial_0 \B_S$.

Observe that  $h^{\varepsilon}$ is a Cauchy-Dirichlet boundary data satisfying \eqref{eq: local Lipschitz condition h} and \eqref{eq: local semiconcave condition h} with  constants independent of $\varepsilon$. 
By construction $h^{\varepsilon}$ is uniformly Lipschitz in $t\in [0,S]$.
 The previous step shows that $U^{\varepsilon}:=U_{h^{\varepsilon},g,F,\mathbb{B}_{S}}$ is a continuous pluripotential solution to \eqref{eq: CMAF} with boundary data $h^{\varepsilon}$. 
 By Proposition \ref{pro: identity principle}, $U^{\varepsilon}$ 
  uniformly converges to $U$ on $\mathbb{B}_S$.  The continuity of $U^{\varepsilon}$  ensures that $U$ is continuous in $\B_S$. Since $h_0$ is continuous, Theorem \ref{thm: boundary value of U} ensures that $U$ is continuous in $[0,S[\times \bar{\B}$.  
It  follows from Theorem \ref{thm: U is local Lip in t} and Theorem 
\ref{thm: U is local semiconcave} that the family $U^{\varepsilon}$
 is locally uniformly  
 semi-concave in $t\in ]0,S[$ with constants independent of $\varepsilon$; see \eqref{eq: kappa U} and \eqref{eq: constant C  U}.  
 Arguing as in the last part of Step 1 we conclude  that $U$ solves 
  \eqref{eq: CMAF} in $\mathbb{B}_S$.  
\end{proof}

 \subsection{The case of bounded  strictly pseudoconvex domains}
 \label{subsect: pseudoconvex existence}
 
 We now consider the case of a smooth bounded  strictly pseudoconvex domain.

  We first prove the existence result in a particular case.

  \begin{pro} \label{pro: U is solution h control 0 and T}   Assume $T<+\infty$,  $h$ satisfies \eqref{eq: local Lipschitz condition h} and \eqref{eq: local semiconcave condition h}.   
  Then $U_{h,g,F} $  is a pluripotential solution to the Cauchy-Dirichlet problem for the parabolic equation \eqref{eq: CMAF} in $\Omega_T$ with boundary conditions \eqref{eq: Dirichlet condition} and \eqref{eq: Cauchy condition}.  
 \end{pro}
 
 \begin{proof}
 	It follows from Theorem \ref{thm: U is local Lip in t}, Theorem \ref{thm: U is subsolution}, Theorem \ref{thm: U is local semiconcave} that $U$ is locally 
 	uniformly semi-concave in $t\in ]0,T[$,
 	 $U\in \mathcal{S}_{h,g,F} (\Omega_T)$ and it satisfies the boundary conditions \eqref{eq: Dirichlet condition} and \eqref{eq: Cauchy condition}. 
It remains to verify that $U$ solves  \eqref{eq: CMAF}.   
 	We proceed in several steps. 
 	
 	\smallskip
 	
 	\noindent{\it Step 1.} We first assume that $h_0$  and $G:=\log g$ are continuous in $\bar{\Omega}$.  
 	Then $U$ is also continuous on $[0,T[ \times \bar{\Omega}$ thanks to Theorem \ref{thm: U is Lip z}.  
 	
 	Let $B\Subset \Omega$ be a small ball and $h_B$ denote the restriction of $U$ on the parabolic boundary of $B_T$.  
 	The boundary data $h_B$ for the Cauchy-Dirichlet problem for \eqref{eq: CMAF}  satisfies the assumption of 
 	Proposition \ref{pro: U is solution ball}. Also,  the restriction of $U$ on $[0,T[\times B$ is a continuous subsolution to the Cauchy Dirichlet problem \eqref{eq: CMAF} in $B_T$ with boundary data  $h_B$. 
 It  follows from Proposition \ref{pro: U is solution ball} that  $U_B:=U_{h_B,g,F,B_T}$ is a pluripotential solution to \eqref{eq: CMAF} with boundary data $h_B$ and $U_B \geq U$ in $B_T$. 
 	
 	The function $V$, which is defined as  $U_B$ in $B_T$ and $U$ in $\Omega_T \setminus B_T$, belongs to  $\mathcal{S}_{h,g,F} (\Omega_T)$. Hence $V=U$  is a pluripotential solution to \eqref{eq: CMAF}. 
 	
 	\smallskip

 	\noindent{\it Step 2.} We  next  assume  $h_0$ is continuous, but we merely assume $g\in L^p$. 
 	
 Let $(g_j)$ be a sequence of strictly positive continuous functions in
  $\bar{\Omega}$ that converges to $g$ in $L^p(\Omega)$. 
  Set $U^j:=U_{h,g_j,F}$ and $U:=U_{h,g,F}$.   
 		Since the $L^p$-norm of $g_j$ is uniformly bounded, 
 		Theorem \ref{thm: U is local Lip in t} and Theorem 
 		\ref{thm: U is local semiconcave} ensure that the functions $U^j$ 
 		are locally uniformly  
 		 semi-concave 
 		(with constants independent of $j$). 
 		  It thus follows from  Proposition \ref{pro: Montel property of P} 
 		  that  a subsequence of $U^j$, still denoted by $U^j$, converges almost everywhere in $\Omega_T$  to 
 		   a function $V\in \mathcal{P}(\Omega_T)$. Lemma \ref{lem:L1Slice-L1} ensures that $U^j_t$ converges in $L^1(\Omega)$ to $V_t$, for all $t\in ]0,T[$.  By Proposition \ref{pro: convergence semiconcave}, for almost all $t\in ]0,T[$, $\partial_t U^j(t,\cdot)$ converges pointwise to $\partial_t V(t,\cdot)$. Thus, for almost all $t\in ]0,T[$,  
 		 $$
 		  e^{\partial_t U^j(t,\cdot) +F(t,\cdot,U^j)}g_j
 		  \stackrel{L^p(\Omega)}{\longrightarrow}
 		  e^{\partial_t V(t,\cdot) +F(t,\cdot,V)}g.
 		  $$
 		  
A result due to Ko{\l}odziej  
(see \cite[End of the proof of Theorem 3]{Kol96}, see also \cite[Theorem 2.8]{DK14})  ensures that $U^j(t,\cdot)$ uniformly converges to $V(t,\cdot)$ and $(dd^c U^j(t,\cdot))^n$ converges  in the sense of positive measures  to $(dd^c V(t,\cdot))^n$. 
Thus $dt \wedge (dd^c U^j)^n$ weakly converges in $\Omega_T$ to $dt\wedge (dd^c V)^n$ (see the proof of Proposition \ref{prop: dt MA converge}). Hence $V$ solves \eqref{eq: CMAF}  in $\Omega_T$. Lemma \ref{lem: super-barrier Dirichlet} and Corollary \ref{cor: stability of boundary inequality} ensure that  $V^*\big|_{\partial_0 \Omega_T}\leq h$. Thus $V\leq U$. 
 		   
 		   \smallskip

 		To prove that $U\leq V$ we now use a perturbation argument 
 		following an idea of Ko{\l}odziej \cite{Kol96} 
 		(see also \cite{GLZ_stability}). 
 		For each $j$ let $\theta_j$ be the unique continuous psh  function in $\bar{\Omega}$, 
 		vanishing on $\partial {\Omega}$ such that $(dd^c \theta_j)^n=|g_j-g|dV$. It  follows from \cite{Kol98} that 
 		 \[
 		\lim_{j\to +\infty} \sup_{\bar{\Omega}} |\theta_j| =0. 
 		 \]
 Fix $0<S<T$, $\varepsilon>0$ small enough and set, for $(t,z) \in \Omega_S$,
	$$
	W^{j} (t,z):= W^{j,\varepsilon}(t,z):= U(t+\varepsilon, z)-\delta(\varepsilon)t + C(\varepsilon)  \theta_j(z), 
	$$
	where $\delta(\varepsilon)>0, C(\varepsilon)>0$ are constants   to be chosen in such a way that $\delta(\varepsilon)\to 0$ but $C(\varepsilon)$ may blow up as $\varepsilon\to 0$. The goal is to prove that $W^j\in \mathcal{S}_{h,g_j,F} (\Omega_S)$.   	It follows from Lemma \ref{lem: uniform convergence} that $U_t$ uniformly converges on $\bar{\Omega}$ to $h_0$, ensuring that
	$$
	b(\varepsilon) : = \sup_{\partial_0 \Omega_S} |U(t+\varepsilon,z) -h(t,z)|  \stackrel{\varepsilon\to 0}{\longrightarrow} 0. 
	$$
	
	A direct computation shows that 
	\begin{flalign*}
	 (dd^c W^j)^n & \geq   (dd^c U(t+\varepsilon,\cdot))^n + C(\varepsilon)^n (dd^c \theta_j)^n \\
	&\geq   e^{\partial_{t} U(t+\varepsilon,\cdot) + F(t+\varepsilon,\cdot,U(t+\varepsilon,\cdot))} g(z) dV+ C(\varepsilon)^n |g-g_j|dV.
	\end{flalign*}
	By the Lipschitz condition \eqref{eq: Lip F} on $F$ we can write
	\begin{flalign*}
			| F(t+\varepsilon,\cdot,  U(t+\varepsilon,\cdot)) -   F(t, \cdot, U(t+\varepsilon,\cdot)) | 
			  \leq  \varepsilon \kappa_F. 
	\end{flalign*}
	Since $r\mapsto F(t,z,r)$ is increasing,  
	$$
	F(t,\cdot,U(t+\varepsilon,\cdot)) \geq  F(t,\cdot, W^j(t,\cdot)) - A \varepsilon,
	$$	 
	where $A>0$ depends on $\kappa_F, M_U$. We choose  $\delta(\varepsilon):=  b(\varepsilon)+ A \varepsilon$. Then 
	\begin{flalign*}
	  (dd^c W^j)^n  & \geq 
	   e^{ \partial_{t} U(t+\varepsilon,\cdot) + F(t, \cdot, W^j(t,\cdot)) - A \varepsilon}g(z) dV + C(\varepsilon)^n |g-g_j|dV \\
	 & \geq  e^{ \partial_{t} W^j(t,\cdot) + F(t, \cdot, W^j(t,\cdot))}g(z) dV + C(\varepsilon)^n |g-g_j|dV
	\end{flalign*}
	and $W^j\big |_{\partial_0 \Omega_S} \leq h$.
We now choose  
	$$
	C(\varepsilon):=  \left( \sup_{\Omega_S} \exp\left \{  \partial_{t} U(t+\varepsilon,z) + F(t, z, U(t+\varepsilon,z)) \right \} \right)^{1/n}  <+\infty.
	$$
	Since $r\mapsto F(t,z,r)$ is increasing  we obtain
	\begin{flalign*}
	  (dd^c W^j)^n  & \geq  e^{\partial_{t} W^j(t,\cdot) + F(t, \cdot, W^j(t,\cdot) )}g dV  + e^{\partial_{t} W^j(t,\cdot) + F(t, \cdot, W^j(t,\cdot) )} |g-g_j| dV \\
	& \geq   e^{\partial_{t} W^j(t,\cdot) + F(t, \cdot, W^j(t,\cdot)) }g_j dV. 
	\end{flalign*} 
Thus $W^j \in \mathcal{S}_{h,g_j,F} (\Omega_S)$. Together with Proposition \ref{pro: identity principle} this yields
	\begin{equation}
		\label{eq: compare U_j and U 1}
		W^{j,\varepsilon}  \leq U_{h,g_j,F, \Omega_T}, \ \text{for all} \ (t,z)\in \Omega_S.
	\end{equation}
In \eqref{eq: compare U_j and U 1} we first let $j\to +\infty$ and then $\varepsilon\to 0$ to arrive at $U\leq V$. Hence $U=V$ is a pluripotential solution to the parabolic Monge-Amp\`ere equation \eqref{eq: CMAF} with boundary data $h$.  
	
	\smallskip

		\noindent{\it Step 3.} We finally remove the continuity assumption on $h_0$. 
	Using Lemma \ref{lem: continuous approximation}  we find  a sequence $h^j$ of continuous Cauchy-Dirichlet boundary data for $\Omega_T$ such that $h^j=h$ on $[0,T[\times \partial\Omega$ and $h^j$ decreases pointwise to $h$. 
	 The previous step ensures that $U^j:= U(h^j,g,F)$ solves \eqref{eq: CMAF}.  
	 Theorem \ref{thm: U is local Lip in t} and Theorem \ref{thm: U is local semiconcave} provide uniform 
	 concavity constants for $U^j$. Since $h^j$ decreases to $h$,  $U\leq U^j$ decreases to some $V\in \mathcal{P}(\Omega_T)$. We thus have $V^*\big |_{\partial_0 \Omega_T}\leq h$, and  Proposition \ref{pro: convergence semiconcave} and Proposition \ref{prop: dt MA converge} reveal that  $V$ solves  \eqref{eq: CMAF}.  
Thus $V$ is a candidate defining $U$, hence $U=V$. 
	 \end{proof}

  We are now ready to prove a general  existence result.  
  Here $T$ may take the value $+\infty$. 
   We assume that, for each $0<S<T$, there exists a constant $C(S)>0$ such that for all $(t,z) \in ]0,S] \times \partial \Omega$,
  \begin{equation}
  	\label{eq: general condition h}
   t|\partial_t h(t,z)|\leq C(S) \, \, \, ; \, \, \ t^2 \partial_t^2 h(t,z)\leq C(S).
  \end{equation}
 \begin{theorem} \label{thm: U is solution general}
 If $h$ satisfies \eqref{eq: general condition h}  then $U:=U_{h,g,F}$  is a pluripotential solution to the Cauchy-Dirichlet problem for  \eqref{eq: CMAF} in $\Omega_T$ with boundary condition $h$. Moreover, $U$ is continuous in $]0,T[\times \bar{\Omega}$ and locally uniformly semi-concave in $t\in ]0,T[$. 
 
  In particular, if $h_0$ is continuous on $\bar{\Omega}$ then $U$ is continuous on $[0,T[\times \bar{\Omega}$.
 \end{theorem}
 
\begin{proof}
	For $S\in ]0,T[$ we define  
	$U^S:= U_{h,g,F,\Omega_S}$. Proposition 
	 \ref{pro: U is solution h control 0 and T} ensures that $U^{S}$ solves
	  \eqref{eq: CMAF} with $U^S=h$ on $\partial_0 \Omega_S$.  It follows 
	  from Proposition \ref{pro: identity principle} that, for $0<S_1<S_2<T$, 
	  $U^{S_1}=U^{S_2}$ on $\Omega_{S_1}$.  Letting $S\to T$ we obtain
	   a function $V\in \mathcal{P}(\Omega_T)$ which solves 
	   \eqref{eq: CMAF} and satisfies $V=h$ on $\partial_0 \Omega_T$. 
	   Obviously $U\leq U^S$, for all $S\in ]0,T[$, hence $U\leq V$. But $V$ 
	   is also a candidate defining $U$, hence $V\leq U$. Therefore $V=U$ solves \eqref{eq: CMAF} in $\Omega_T$. Moreover, by Theorem \ref{thm: U is local Lip in t} and Theorem \ref{thm: U is local semiconcave}, $U^S$ is locally uniformly Lipschitz and semiconcave in $t\in ]0,S[$, hence so is $U$. 	   
	    
	It follows from Proposition \ref{pro: subsolution slice} and Remark \ref{rem: supersolution slice} that 
$$
(dd^c U_t)^n =e^{\partial_t U_t + F(t,\cdot,U_t)}gdV
$$
for almost every $t\in ]0,T[$.
	Since $\partial_t U$ is locally bounded  and $h_t$ is continuous on $\partial{\Omega}$ for all $t\in ]0,T[$,   \cite{Kol98} ensures that $U_t$ is continuous on $\bar{\Omega}$ for almost all $t\in ]0,T[$. Since $U$ is locally uniformly Lipschitz in $t$ we infer that $U$ is continuous in $]0,T[\times \bar{\Omega}$. 
	
	If $h_0$ is continuous on  $\bar{\Omega}$ then Theorem \ref{thm: U is subsolution} and the continuity of $U(t,\cdot)$ 
	(for each $t\in ]0,T[$ fixed)
	ensure that $U$ is continuous on $[0,T[\times \bar{\Omega}$. 
\end{proof}

\subsection{Uniqueness}\label{sect: uniqueness}

We have proved in Section \ref{subsect: pseudoconvex existence} the existence of a pluripotential solution  to \eqref{eq: CMAF} which is 
 locally uniformly semi-concave in $t$.
Our next goal is to prove that this is the unique such solution :

\begin{theorem}
\label{thm: parabolic comparison principle}
Let $\Phi, \Psi \in \mathcal{P}(\Omega_T)\cap L^{\infty}(\Omega_T)$ with boundary data $h_{\Phi}, h_{\Psi}$. Assume that 
\begin{enumerate}
\item $\Psi$ is locally uniformly semi-concave in $t\in ]0,T[$;
\item $\Phi$ is a subsolution while $\Psi$ is a supersolution to \eqref{eq: CMAF} in $\Omega_T$;
\item $h_{\Phi}$ satisfies \eqref{eq: general condition h}. \label{item hPhi}
\end{enumerate}  
Then $h_{\Phi}\leq h_{\Psi} \Longrightarrow \Phi \leq \Psi$.
\end{theorem}
Here $h_{\Phi},h_{\Psi}$ are Cauchy Dirichlet boundary data in $\Omega_T$. In particular, $h_{\Psi}(t,\cdot)$ is continuous on $\partial \Omega$, and the supersolution property of $\Psi$ implies that $\Psi$ is continuous in $]0,T[ \times \bar{\Omega}$ (see Theorem \ref{thm: U is solution general}).

An important consequence of this comparison principle is the following uniqueness result :

\begin{cor}\label{cor: uniqueness}
	 Assume that $\Phi,\Psi\in \mathcal{P}(\Omega_T)$ are two pluripotential solutions to \eqref{eq: CMAF}  with boundary values $h$ satisfying \eqref{eq: general condition h}.
	 If $\Phi,\Psi$ are locally uniformly semi-concave in $t\in ]0,T[$ then $\Phi=\Psi$ in $\Omega_T$. 
\end{cor}

\begin{proof} 
Let $U:= U_{h,g,F,\Omega_T}$. Then Theorem \ref{thm: U is solution general} ensures that $U$ solves  \eqref{eq: CMAF} and $U,\Phi,\Psi$ are continuous on $]0,T[\times \bar{\Omega}$. By definition, $\Phi,\Psi \leq U$. It follows from Theorem \ref{thm: parabolic comparison principle} that $U\leq \Phi,\Psi$, hence equality.
\end{proof}

 We first establish Theorem \ref{thm: parabolic comparison principle} under extra assumptions :

\begin{lemma}
	\label{lem: CP first case}
	With the same assumptions as in Theorem \ref{thm: parabolic comparison principle}, assume moreover that  $\Phi$ is  $\mathcal{C}^1$ in $t$, continuous on $]0,T[\times \bar{\Omega}$, and $\Psi$ is continuous on $[0,T[ \times \bar{\Omega}$. Then 
	$h_{\Phi}\leq h_{\Psi} \Longrightarrow \Phi \leq \Psi$.
\end{lemma}
	The first assumption (that $\Phi$ is $\mathcal{C}^1$ in $t$) means that $(t,z) \mapsto \partial_t \Phi(t,z)$ exists and it is continuous on $]0,T[ \times \Omega$. 
\begin{proof}
We fix $S\in ]0,T[$, $\varepsilon>0$ small enough, and prove that  
	$$
	\Phi\leq \Psi+ 2\varepsilon t \text{ in }  \Omega_{S}. 
	$$
	The function 
	$$
	[0,S]\times \bar{\Omega} \ni (t,z) \mapsto W(t,z):=\Phi(t,z)-\Psi(t,z)-2\varepsilon t
	$$ 
	is upper semi-continuous and bounded.
We are done if the maximum is attained on $\partial_0 \Omega_{S}$. We thus assume that $\max W$
is reached at some point $(t_0,z_0) \in ]0,S]\times {\Omega}$. We want to prove that $W(t_0,z_0)\leq 0$. Assume, by contradiction that it is not the case. 
Then the set 
	$$
	K:= \{z\in \Omega \setdef W(t_0,z) = W(t_0,z_0)\}
	$$
is compact  and
the maximum principle ensures that
	$$
	\partial_t \Phi(t_0,z) \geq \partial_t^{-} \Psi(t_0,z) +2\varepsilon, \ \text{for all}\ z\in K.
	$$
Since $\Psi$ is locally uniformly semi-concave in $t\in ]0,T[$ and continuous on $[0,T[\times \bar{\Omega}$, the left derivative $\partial_t^- \Psi(t,z)$ exists and it is upper semi-continuous in $]0,T[\times \Omega$.  
Hence we can find $r>0$  so small   that
	 $$
	\partial_t \Phi(t_0,z) \geq \partial_t^{-} \Psi(t_0,z) +\varepsilon, \ \text{for all}\ z\in B,
	$$
	where $B=B_r:=\{z\in \Omega \setdef {\rm dist}(z,K)<r\}$. 
	
	Since $\Phi$ is a subsolution (which is $\mathcal{C}^1$ in $t$) while $\Psi$ is a supersolution to \eqref{eq: CMAF}, 
	Proposition \ref{pro: subsolution slice} and Remark \ref{rem: supersolution slice}  ensure that
	$$
	(dd^c \varphi)^n \geq e^{F(t_0,z,\varphi(z))-F(t_0,z,\psi(z))+\varepsilon}(dd^c \psi)^n,
	$$
		setting $\varphi:=\Phi(t_0,\cdot), \psi:= \Psi(t_0,\cdot)$.
	 Since $\varphi$ and $\psi$ are continuous in $\Omega$, $F$ is increasing in $r$, and $\varphi(z)\geq \psi(z)+2\varepsilon t_0$ on $K$, up to shrinking $B$  we can assume that 
	$$
	(dd^c \varphi)^n \geq  e^{\varepsilon}(dd^c \psi)^n \ \text{in} \ B. 
	$$
Set now $\varphi_r:=\varphi + m_r$, where  $m_r:=\min_{\partial B}(\psi-\varphi)$. 
Since $\psi\geq \varphi_r$ on $\partial B$, the comparison principle \cite{Bedford_Taylor_1976Dirichlet}  yields
	\[
	\int_{\{\psi<\varphi_r\}\cap B}e^{\varepsilon} (dd^c \psi)^n \leq \int_{\{\psi<\varphi_r\}\cap B} (dd^c \varphi_r)^n \leq \int_{\{\psi<\varphi_r\}\cap B} (dd^c \psi)^n. 
	\]
Therefore $(dd^c \psi)^n$ does not charge the set $\{z\in B  \setdef  \psi(z)<\varphi_r(z)\}$ and the domination 
principle (see e.g. \cite[Proposition 1.2]{GLZ_stability}) 
  yields 
 $\varphi_r\leq \psi$  in $B$.  In particular 
	\begin{equation*}
		\varphi(z_0)-\psi(z_0) +\min_{\partial B}(\psi-\varphi)=\varphi_r(z_0) -\psi(z_0) \leq 0. 
	\end{equation*}
	
	Since $K\cap \partial B =\emptyset$, we obtain, for all $z\in \partial B$,  $W(t_0,z) < W(t_0,z_0)$ 
	hence
	$$
\f(z)-\p(z) < \f(z_0)-\p(z_0)	 \leq \max_{\partial B} (\f-\p),
	$$
	a contradiction.
	Thus   $\Phi\leq \Psi+2\varepsilon t$ 
	and we conclude by letting $\varepsilon\rightarrow 0$.
\end{proof}

We next establish an estimate for supersolutions to \eqref{eq: CMAF}. 

\begin{lemma}\label{lem: CP second case}
	Assume  $\Psi\in \mathcal{P}(\Omega_T)$ has boundary data $h_{\Psi}$.  If $\Psi$ is a pluripotential supersolution to \eqref{eq: CMAF} then for all $(t,z)\in \Omega_T$,
	$$
	\Psi(t,z) \geq h_{\Psi}(0,z) -c(t), 
	$$
	where $c(t)>0$ satisfies $\lim_{t\to 0^+} c(t)=0$.
\end{lemma}

\begin{proof}
	Fix $0<S<T$. For $s>0$ small enough we set 
	\begin{equation}
		\label{eq: varepsilon(t)}
		\delta(s) := \sup \{ |h_{\Psi}(\tau,z) -h_{\Psi}(t,z)|  \setdef  z \in  \partial \Omega, \ t,\tau \in [0,S], \ |t-\tau|\leq  s \}.
	\end{equation}
	Since $h_{\Psi}$ is continuous on $[0,T[\times  \partial \Omega$, 
	we have $\lim_{s\to 0^+} \delta(s) =0$.
	\smallskip
	
Fix $s\in ]0,(T-S)/2[$. We are going to prove that
	\[
	\Psi(s,z) \geq  h_{\Psi}(0,z)  -\delta(s) +  s(\rho(z)-C) + n (s\log (s/T) -s), 
	\]
	where  $\rho$ is defined in \eqref{eq: rho} and $C$ is a uniform constant. 	
	
	Fix $\e\in ]0,s]$ and let $h^{\varepsilon}$ denote the restriction of $(t,z)\mapsto \Psi(t+\e,z)$ on $\partial_0 \Omega_{s}$.  Then $h^{\varepsilon}$ is a continuous boundary data on $\Omega_{s}$.  Set, for $(t,z)\in \Omega_s$,
	$$
	u^{\varepsilon}(t,z) := \Psi(\varepsilon,z)-\delta(s) +  t(\rho(z)-C_1) + n (t\log (t/T) -t), 
	$$
	where $C_1$ is a positive constant. By definition of $\delta(s)$ we have
	$$
	u^{\varepsilon}(t,z) \leq \Psi(t+\e,z)  = h^{\varepsilon}(t,z), \ \text{for all}\ (t,z)\in \partial_0 \Omega_s. 
	$$ 
Arguing as in the proof of Lemma \ref{lem: sub-barriers Cauchy} we see that for $C_1>0$ big enough (depending on $M_F$), $u^{\varepsilon}$ is a  pluripotential subsolution to \eqref{eq: CMAF} in $\Omega_{s}$.  Moreover, $u_{\varepsilon}$ is of class $\mathcal{C}^1$ in $t\in [0,s]$. On the other hand, a direct computation shows that, for $C_2>0$ large enough and under control (depending on $\kappa_F$), the function 
	$$
[0,s]\times \Omega \ni (t,z) \mapsto	w^{\varepsilon}(t,z) := \Psi(t+\varepsilon,z) +C_2\e t
	$$
	is a  pluripotential supersolution to \eqref{eq: CMAF} and $w^{\varepsilon} \geq  h^{\varepsilon}$ on $\partial_0 \Omega_{s}$. By assumption on $\Psi$,  $w^{\varepsilon}$ is continuous on $[0,s]\times \bar{\Omega}$. It thus follows from Lemma \ref{lem: CP first case} that $w^{\varepsilon} \geq u^{\varepsilon}$  on $[0,s]\times \Omega$.  
	We conclude by  letting $\varepsilon\to 0$.
\end{proof}

We next remove the continuity assumption on $\Psi$ in Lemma \ref{lem: CP first case}.

\begin{lemma}
	\label{lem: CP third case}
		With the same assumptions as in Theorem \ref{thm: parabolic comparison principle}, assume moreover that  $\Phi$ is  $\mathcal{C}^1$ in $t$ and continuous on $]0,T[\times \bar{\Omega}$. Then 
	$$
	h_{\Phi}\leq h_{\Psi} \Longrightarrow \Phi \leq \Psi.
	$$
\end{lemma}

\begin{proof}
Since $h_{\Psi}$ is continuous on $[0,T[\times \partial \Omega$, the proof of Theorem \ref{thm: U is solution general} shows that $\Psi$ is continuous in $]0,T[\times \bar{\Omega}$, but it may not be continuous on $[0,T[\times \bar{\Omega}$. 
	We use an idea in  \cite{DiNezza_Lu_2017KRflow}, exploiting the regularity of $\Psi$ at positive times close to zero.
	We fix $S\in ]0,T[$ and prove that $\Phi \leq \Psi$ on $\Omega_{S}$. 
	
	Fix $s\in ]0,(T-S)/2[$ and set, for $(t,z)\in [0,S]\times \bar{\Omega}$, 
	$$
	v(t,z) : =\Psi(t+s,z)+c(s) +\delta(s)+Ast, 
	$$
	where $\delta(s)$ is defined in \eqref{eq: varepsilon(t)},  $A>0$ is a constant, and $c(s)>0$ is as in Lemma \ref{lem: CP second case} (which ensures $\Psi(s,z) \geq h_{\Psi}(0,z) -c(s)$).  
	
 From the definition of $\delta(s)$ it follows that $v(t,z) \geq \Psi(t,z) = h_{\Psi}(t,z)$  on $[0,S]\times \partial \Omega$. 
For $A>0$ large enough (depending on $\kappa_F$), a direct computation shows that $v$ is a supersolution to \eqref{eq: CMAF}. 
	Since $v$ is continuous on $[0,S]\times \bar{\Omega}$, Lemma \ref{lem: CP first case}  then applies 
and yields $\Phi(t,z)\leq v(t,z)$ on $[0,S]\times \Omega$. 
We conclude by letting $s\to 0$.
\end{proof}

We are now ready to prove the comparison principle. 
 
\begin{proof}[Proof of Theorem \ref{thm: parabolic comparison principle}]  
We can assume without loss of generality that $\Phi=U_{h_{\Phi},g,F}$.   From assumption \eqref{item hPhi} and Theorem \ref{thm: U is solution general} we deduce that $U_{h_{\Phi},g,F}$ is continuous on $]0,T[\times \bar{\Omega}$. 
 We would like to apply Lemma \ref{lem: CP third case} but  $\Phi$ is a priori not $\mathcal{C}^1$ in $t$.
We are going to  regularize $\Phi$ by taking convolution  in $t$.
	
	Fix $0<S<T$. For $s>0$ near $1$ we set, for $(t,z)\in \Omega_S$,
	$$
	W^s(t,z):= s^{-1} \Phi(st),z) - C|s-1|(t+1).
	$$
If $C>0$ is large enough, the proof of Theorem \ref{thm: U is local Lip in t} ensures that $W^s\in \mathcal{S}_{h_{\Phi},g,F} (\Omega_S)$. Let $\{\chi_{\varepsilon}\}_{\varepsilon>0}$ be a family of smoothing kernels in $\mathbb{R}$ approximating the Dirac mass $\delta_0$. 
For $\varepsilon>0$ small enough we define
	\begin{equation}
		\label{eq: Phi epsilon}
		\Phi^{\varepsilon}(t,z) := \int_{\mathbb{R}} W^s(t,z) \chi_{\varepsilon}(s)ds. 
	\end{equation}
 We are going to prove that $\Phi^{\e}$ (or $\Phi^{\e}-O(\e)$) is again a subsolution and use the previous step to conclude.
 
 \smallskip
 
 Let ${\mathcal H}$ denote the space of hermitian positive definite matrix
$H$ that are normalized by $\det H=1$, and let $\Delta_H$ denote the Laplace operator 
\[
\Delta_H \f:=\frac{1}{n}\sum_{j,k=1}^n h_{jk} \frac{\partial^2 \f}{\partial z_j \partial \bar{z}_k}.
\]
Fix $H\in \mathcal{H}$. Since $W^s\in \mathcal{S}_{h_{\Phi},g,F} (\Omega_T)$,
Proposition \ref{pro: subsolution slice} and \cite[Main Theorem]{Guedj_Lu_Zeriahi_2017subsolution}  yield
 $$
 \Delta_H W^s(t,z) \geq \exp \left(\frac{\partial_t W^s(t,z) +F(t,z,W^s(t,z))}{n} \right) g(z)^{1/n}.
 $$
 By definition of $\Phi^{\varepsilon}$ we obtain, using the convexity of the exponential,
\begin{flalign*}
&\Delta_H  \Phi^{\varepsilon}(t,z) = \int_{\mathbb{R}} \Delta_H W^s(t,z) \chi_{\varepsilon}(s)ds 	\\
& \geq  g(z)^{1/n} \int_{\mathbb{R}} \exp \left(\frac{\partial_t W^s(t,z) +F(t,z,W^s(t,z))}{n} \right) \chi_{\varepsilon}(s)ds\\
&\geq g(z)^{1/n}\exp \left (\frac{1}{n}\left (\int_{\mathbb{R}} \left (\partial_t W^s(t,z) +F(t,z,W^s(t,z)) \right)\chi_{\varepsilon}(s)ds \right )\right ).
\end{flalign*}

\noindent{\it Step 1. } {\it To simplify we first treat the case when $F$ is convex in $r$}.
Thus
 \begin{flalign*}
&\Delta_H  \Phi^{\varepsilon}(t,z) = \int_{\mathbb{R}} \Delta_H W^s(t,z) \chi_{\varepsilon}(s)ds 	\\
& \geq g(z)^{1/n}\exp \left (\frac{1}{n}\left (\partial_t \Phi^{\varepsilon}(t,z) + F\left (t,z,\int_{\mathbb{R}} W^s(t,z)\chi_{\varepsilon}(s) ds  \right )\right )  \right )\\
& =g(z)^{1/n}\exp \left (\frac{1}{n}\left (\partial_t \Phi^{\varepsilon}(t,z) + F(t,z,\Phi^{\varepsilon}(t,z))  \right )  \right ).
\end{flalign*}
 
Using Proposition \ref{pro: subsolution slice} and \cite[Main Theorem]{Guedj_Lu_Zeriahi_2017subsolution} again, we infer   that $\Phi^{\varepsilon}$
is a subsolution to \eqref{eq: CMAF}  in $\Omega_S$. 

We now check that $(\Phi^{\varepsilon})^*-O(\varepsilon)\leq h_{\Phi}$ on $\partial_0 \Omega_S$. Indeed, for $z\in \partial\Omega$ we have $W^s(t,z)\leq h_{\Phi}(t,z)$, for all $s$, thus $\Phi^{\varepsilon}(t,z)\leq h_{\Phi}(t,z)$ for all $(t,z)\in [0,S]\times \partial \Omega$. It remains to check that $(\Phi^{\varepsilon})^*(0,z)\leq h_{\Phi}(0,z)$, for all $z\in \Omega$. It follows from Theorem \ref{thm: U is solution general} that $U_{h_{\Phi},g,F,\Omega_T}$ has boundary value $h_{\Phi}$, hence, for $C$ large enough
\[
\lim_{t\to 0} W^s(t,z)\leq h_{\Phi}(0,z) -C|s-1|, \ \text{for all}\ z \in \bar{\Omega}.
\]
From the definition of $\Phi^{\varepsilon}$ in \eqref{eq: Phi epsilon} it follows that 
$$
\lim_{t\to 0} \Phi^{\varepsilon}(t,z) \leq h_{\Phi}(0,z), \ \forall z\in \bar{\Omega}. 
$$
Hence $\Phi^{\varepsilon}-O(\varepsilon) t\in \mathcal{S}_{h_{\Phi},g,F} (\Omega_S)$.
Moreover, $\Phi^{\varepsilon}$ is of class $\mathcal{C}^1$ in $t\in ]0,S[$ and 
$\Phi^{\varepsilon}$ converges pointwise to $\Phi$ as $\varepsilon\to 0$. 
Using Lemma \ref{lem: CP third case} we obtain $\Phi^{\varepsilon}\leq \Psi$ in $\Omega_S$. 
The conclusion follows by letting $\varepsilon\to 0$.

\smallskip

\noindent{\it Step 2. }
{\it We now treat the  case when $F$ is merely uniformly semi-convex in $r$}.  
 It follows from \eqref{eq: local Lipschitz condition h reformulation} and Theorem \ref{thm: U is local Lip in t} that the functions $s\mapsto W^s(t,z)$, $(t,z)\in \Omega_S$, are uniformly Lipschitz in $[1/2,3/2]$. Thus for all $(t,z)\in \Omega_S$,
 \[
 |\Phi^{\varepsilon}(t,z) - \Phi(t,z) | \leq C \varepsilon,
 \]
 for some uniform constant $C$, hence
 \begin{equation}
 	\label{eq: W star square and W square star}
 	 \int_{\mathbb{R}} (W^s)^2(t,z) \chi_{\varepsilon} (s) ds -  \left (\int_{\mathbb{R}} W^s(t,z) \chi_{\varepsilon} (s) ds \right )^2 = O(\varepsilon). 
 \end{equation}
 
 Recall (assumption \eqref{eq: semi convex F}) that the function $r \mapsto F(t,z,r) + C_F r^2$ is convex 
 in a large interval $J\Subset \mathbb{R}$, for fixed $(t,z)\in \Omega_S$. Jensen's inequality    yields
 \begin{flalign*}
 	\int_{\mathbb{R}} & \left ( F(t,z,W^s(t,z))  + C_F (W^s(t,z))^2 \right ) \chi_{\varepsilon}(s) ds \\
 & \geq  F \left (t,z, \int_{\mathbb{R}} W^s (t,z) \chi_{\varepsilon}(s) ds  \right ) + C_F \left( \int_{\mathbb{R}} W^s (t,z) \chi_{\varepsilon}(s) ds \right)^2.  
 \end{flalign*}
 Using this and \eqref{eq: W star square and W square star}  we  obtain
 \begin{equation*}
 	\int_{\mathbb{R}}   F(t,z,W^s(t,z)) \chi_{\varepsilon}(s) ds - F\left (t,z,\int_{\mathbb{R}}W^s(t,z) \chi_{\varepsilon}(s) ds\right) \geq O(\varepsilon). 
 \end{equation*}

We  repeat the previous step to conclude that $\Phi^{\varepsilon} -O(\varepsilon) t \in \mathcal{S}_{h_{\Phi},g,F}(\Omega_S)$.
\end{proof}


\begin{thebibliography}{99}

\bibitem[BT76]{Bedford_Taylor_1976Dirichlet}
E. Bedford and B.~A. Taylor, \emph{The {D}irichlet problem for a complex
  {M}onge-{A}mp\`ere equation}, Invent. Math. \textbf{37} (1976), no.~1, 1--44.
   

\bibitem[BT82]{Bedford_Taylor_1982Capacity}
E. Bedford and B.~A. Taylor, \emph{A new capacity for plurisubharmonic
  functions}, Acta Math. \textbf{149} (1982), no.~1-2, 1--40.  
  
 \bibitem[B{\l}o05]{Blo05} Z.  B{\l}ocki, {\it Weak solutions to the complex Hessian equation}, Ann. Inst. Fourier, Grenoble 55, 5 (2005), 1735--1756.
 
  \bibitem[CT15]{CT15} T. Collins, V. Tosatti,   {\it K\"ahler currents and null loci.} Invent. Math. 202 (2015), no. 3, 1167--1198.
 




\bibitem[Dem91]{Dem91} J. P. Demailly, Potential Theory in Several Complex Variables. Course  available on the author's webpage.

\bibitem[Din09]{Dinew_2009_Mixed_Inequality}
S.  Dinew, \emph{An inequality for mixed monge--ampere measures},
  Mathematische Zeitschrift \textbf{262} (2009), no.~1, 1--15.
  
  \bibitem[DK14]{DK14} S. Dinew, S. Ko{\l}odziej, A priori estimates for complex Hessian equations. Anal. PDE 7 (2014), no. 1, 227--244. 

\bibitem[DL17]{DiNezza_Lu_2017KRflow}
E. Di~Nezza and C.~H. Lu, \emph{Uniqueness and short time regularity
  of the weak {K}\"ahler-{R}icci flow}, Adv. Math. \textbf{305} (2017),
  953--993.  

\bibitem[Do17]{Son1}Hoang-Son Do, \emph{{Weak solution of Parabolic complex Monge-Amp\`ere equation}},
 Indiana Univ. Math. J. {\bf 66} (2017), no. 6, 1949--1979.

\bibitem[Do16]{Son2} Hoang-Son Do, \emph{Weak solution of parabolic complex {M}onge-{A}mp\`ere equation {II}}, Internat. J. Math. \textbf{27} (2016), no.~12, 1650098.

  

   \bibitem[EGZ11]{EGZ11}  P.~Eyssidieux, V.~ Guedj, A.~ Zeriahi, {\em Viscosity solutions to Degenerate Complex Monge-Amp\`ere Equations.} 
  Comm. Pure Appl. Math.  {\bf 64}  (2011),  no. 8, 1059--1094. 
  

\bibitem[EGZ15]{Eyssidieux_Guedj_Zeriahi_2015FlowI}
P.  Eyssidieux, V.  Guedj, and A.  Zeriahi, \emph{Weak solutions to
  degenerate complex {M}onge-{A}mp\`ere flows {I}}, Math. Ann. \textbf{362}
  (2015), no.~3-4, 931--963.  

  
  \bibitem[EGZ18]{Eyssidieux_Guedj_Zeriahi_2017FlowIII}
P.  Eyssidieux, V.  Guedj, A.  Zeriahi, \emph{Convergence of weak K\"ahler-Ricci flows on minimal models of positive Kodaira dimension.}  C.M.P. {\bf 357} (2018), no. 3, 1179--1214.


\bibitem[GT01]{GT01} D. Gilbarg, N.S.  Trudinger,  Elliptic Partial Differential Equations of Second Order, Springer-Verlag Berlin Heidelberg, Edition 2, pages XIII, 518. 

\bibitem[GLZ17]{Guedj_Lu_Zeriahi_2017subsolution}
V.  Guedj, C. H. Lu, A. Zeriahi, \emph{{Weak subsolutions to
  complex Monge-Amp\`ere equations}}, arXiv:1703.06728 (2017). 
  Journal of the Math. Soc. Japan. to appear.
  
\bibitem[GLZ18]{GLZ_stability} V.  Guedj, C. H. Lu, A. Zeriahi, {\it Stability of solutions to   complex Monge-Amp\`ere flows}, 
Ann. Inst. Fourier, to appear.

\bibitem[GLZ2]{Guedj_Lu_Zeriahi_paraboliccompact}
V.  Guedj, C. H. Lu, A. Zeriahi, \emph{{Pluripotential K\"ahler-Ricci flows}}, preprint (2018).

   \bibitem[GLZ3]{GLZ3} V.~ Guedj, C.H .Lu, A.~ Zeriahi,  
   {\em Viscosity vs Pluripotential solutions to complex Monge-Amp\`ere flows}. In preparation.
 

\bibitem[GZ]{GZbook}
V.  Guedj, A. Zeriahi, \emph{Degenerate complex {M}onge-{A}mp\`ere
  equations}, EMS Tracts in Mathematics, vol.~26, European Mathematical Society
  (EMS), Z\"urich, 2017.  


\bibitem[Ham82]{Ham82} R.Hamilton,  {\em Three-manifolds with positive Ricci curvature.}
  J. Diff. Geom. {\bf 17} (2) (1982), 255--306.
  
  \bibitem[HL11]{HL11}  F. R. Harvey, H. B. Lawson Jr, {\it The   Dirichlet duality and the nonlinear Dirichlet problem on Riemannian manifolds},  J. Diff. Geom. 88 no. 3 (2011), 395--482.
  
  \bibitem[HL13]{HL13}  F. R. Harvey, H. B. Lawson Jr, {\it The equivalence of viscosity and distributional subsolutions for convex subequations - a strong Bellman principle}. Bull. Brazilian Math. Soc. 44 no. 4 (2013), 621--652.

  
\bibitem[H\"orm]{Hor07}  L.  H\"ormander, Notions of convexity, Modern Birkh\"auser Boston, Inc., Boston, MA, 2007, Reprint of the 1994 edition.  


\bibitem[Klim]{Klimek_1991_Pluripotential}
M. Klimek, \emph{Pluripotential theory}, London Mathematical Society
  Monographs. New Series, vol.~6, The Clarendon Press, Oxford University Press,
  New York, 1991.

\bibitem[Ko{\l}95]{Kol95}
S. Ko{\l}odziej, \emph{The range of the complex {M}onge-{A}mp\`ere
  operator. {II}}, Indiana Univ. Math. J. \textbf{44} (1995), no.~3, 765--782.
   
\bibitem[Ko{\l}96]{Kol96}
S. Ko{\l}odziej,  \emph{Some sufficient conditions for solvability of the Dirichlet problem for the complex Monge-Amp\`ere operator.} Ann. Polon. Math. 65 (1996), no. 1, 11--21. 


\bibitem[Ko{\l}98]{Kol98}
S.  Ko{\l}odziej, \emph{The complex {M}onge-{A}mp\`ere equation}, Acta
  Math. \textbf{180} (1998), no.~1, 69--117.  

\bibitem[Ko{\l}03]{Kol03}
S.  Ko{\l}odziej, \emph{The {M}onge-{A}mp\`ere equation on compact
  {K}\"ahler manifolds}, Indiana Univ. Math. J. \textbf{52} (2003), no.~3,
  667--686.  
  

  

\bibitem[ST17]{Song_Tian_2017KRflowInvent}
J. Song and G. Tian, \emph{The {K}\"ahler-{R}icci flow through
  singularities}, Invent. Math. \textbf{207} (2017), no.~2, 519--595.
  
    \bibitem [SW13]{SW13} J.~Song, B.Weinkove,  Lecture notes on the K\"ahler-Ricci flow. 
    ``Introduction to the K\"ahler-Ricci flow'', eds S. Boucksom, P. Eyssidieux, V. Guedj, L.N.M.  {\bf 2086}  (2013).

  


  

\bibitem[Wal69]{Walsh_1969Envelopes}
J.~B. Walsh, \emph{Continuity of envelopes of plurisubharmonic functions}, J.
  Math. Mech. \textbf{18} (1968/1969), 143--148.  



\bibitem[Yau78]{Yau78}
S.-T. Yau, \emph{On the {R}icci curvature of a compact {K}\"ahler manifold
  and the complex {M}onge-{A}mp\`ere equation. {I}}, Comm. Pure Appl. Math.
  \textbf{31} (1978), no.~3, 339--411.  

\end{thebibliography}
\end{document}